\documentclass[10pt]{article}
\usepackage[english]{babel}
\usepackage[utf8]{inputenc}
\usepackage[T1]{fontenc}
\usepackage{amsmath,amssymb,amsfonts,amsopn,amsthm}
\usepackage{mathtools}
\usepackage{dsfont}
\usepackage[a4paper, total={6in, 9in}]{geometry}
\usepackage{algorithm}
\usepackage{algpseudocode}

\usepackage[dvipsnames]{xcolor}
\usepackage{graphicx,epstopdf} 
\usepackage[labelformat=simple]{subcaption}
\usepackage[format=plain,
            textfont=it,
            size=footnotesize,labelsep=period]{caption}
\usepackage{tikz}
\usetikzlibrary{spy}
\usepackage{pgfplots}
\usepackage{siunitx,array,booktabs}
\usepackage{ifpdf} 
\usetikzlibrary{patterns}
\pgfdeclarelayer{foreground layer} 
\pgfsetlayers{main,foreground layer}
\usepackage{commath}
\usepackage{bm}
\usepackage{enumitem}
\usepackage[colorinlistoftodos,italian]{todonotes}

\usepackage[colorlinks = true,linkcolor = black,urlcolor  = black,
            citecolor = ForestGreen, anchorcolor = black]{hyperref}
\usepackage[capitalise,noabbrev]{cleveref}
\crefname{equation}{}{}
\Crefname{equation}{Equation}{Equations}

\expandafter\def\csname ver@etex.sty\endcsname{3000/12/31}

\usepackage{autonum}
\usepackage{authblk}

\newcommand{\email}[1]{\href{#1}{#1}}

\pgfplotsset{compat=1.11}

\numberwithin{equation}{section}
\newtheorem{theorem}{Theorem}[section]

\newtheorem{lemma}[theorem]{Lemma}

\newtheorem{definition}[theorem]{Definition}

\newtheorem{remark}[theorem]{Remark}

\crefname{assumption}{Assumption}{Assumptions}
\crefname{remark}{Remark}{Remarks}
\crefname{example}{Example}{Examples}

\newcommand{\plotlinewidth}{1.0}

\newcommand{\plotdashedlinewidth}{0.5}
\newcommand{\plotmarksizeu}{2.5}
\newcommand{\plotmarksized}{3.0}
\newcommand{\plotimscale}{1}
\newcommand{\plotimsizeu}{0.44}
\newcommand{\plotimsized}{0.99} 

\newcommand{\aspectratio}{0.66}
\newcommand{\subfigsize}{0.478}
\newcommand{\subfigsizemed}{0.45}

\newcommand{\legendfontscale}{1}
\newcommand{\legendmarkscale}{0.9}

\newcommand{\legendboxdraw}{} 
\newcommand{\legendboxfill}{} 

\newcommand{\zoomsize}{0.4}
\newcommand{\zoomlinewidth}{2.5}
\newcommand{\zoomdashedlinewidth}{2}

\newcommand{\localcolor}{ForestGreen}
\newcommand{\classicalcolor}{WildStrawberry}
\newcommand{\localmark}{star}
\newcommand{\classicalmark}{o}
\newcommand{\localcolorF}{ForestGreen}
\newcommand{\localcolorS}{NavyBlue}
\newcommand{\localmarkF}{star}
\newcommand{\localmarkS}{triangle}

\newcommand{\colorone}{ForestGreen}
\newcommand{\colortwo}{NavyBlue}
\newcommand{\colorthree}{WildStrawberry}
\newcommand{\colorfour}{BurntOrange}
\newcommand{\markone}{star}
\newcommand{\marktwo}{triangle}
\newcommand{\markthree}{o}
\newcommand{\markfour}{square}

\usepackage{calligra}
\usepackage{csvsimple}

\pgfplotsset{select coords between index/.style 2 args={
    x filter/.code={
        \ifnum\coordindex<#1\fi
        \ifnum\coordindex>#2\fi
    }
}}

\newcommand{\Nb}{\mathbb{N}}
\newcommand{\Rb}{\mathbb{R}}

\newcommand{\Cb}{\mathbb{C}}

\newcommand{\Rn}{\mathbb{R}^{n}}

\newcommand{\bigo}[1]{\mathcal{O}(#1)}
\newcommand{\dt}{\partial_t}

\newcommand{\nld}[1]{\Vert #1\Vert}

\newcommand{\oz}{\omega_0}
\newcommand{\ou}{\omega_1}
\newcommand{\vz}{\upsilon_0}
\newcommand{\vu}{\upsilon_1}
\newcommand{\fs}{f_S}
\newcommand{\ff}{f_F}
\newcommand{\rhos}{\rho_S}
\newcommand{\rhof}{\rho_F}
\newcommand{\rhoe}{\rho_\eta}
\newcommand{\costs}{c_S}
\newcommand{\costf}{c_F}
\newcommand{\bx}{\bm{x}}
\newcommand{\fe}{f_{\eta}}
\newcommand{\bfe}{\overline{f}_\eta}
\newcommand{\ye}{y_\eta}

\renewcommand{\Pi}{P} 
\newcommand{\Pe}{R} 
\newcommand{\Ps}{\Phi} 
\newcommand{\OF}{\Omega_F}

\newcommand{\mRKCop}{\operatorname{mRKC}}
\newcommand{\RKCop}{\operatorname{RKC}}

\title{Explicit stabilized multirate method for stiff differential equations}

\author{Assyr Abdulle\thanks{ANMC, Institute of Mathematics, École Polytechnique Fédérale de Lausanne, 1015 Lausanne, Switzerland \email{giacomo.rosilhodesouza@usi.ch}}, ~Marcus J. Grote\thanks{Department of Mathematics and Computer Science, University of Basel, Spiegelgasse 1, 4051 Basel, Switzerland \email{marcus.grote@unibas.ch}\\ Our esteemed colleague, Assyr Abdulle, passed away on September 1, 2021 during the revision of this article. A wonderful mentor and friend, his enthusiasm for applied mathematics and for music will always remain dearly missed.}, ~Giacomo Rosilho de Souza$^*$ }

\date{}

\begin{document}

\maketitle

\begin{abstract}

Stabilized Runge--Kutta methods are especially efficient for the numerical solution of large systems of stiff nonlinear differential equations because they are fully explicit. 
For semi-discrete parabolic problems, for instance, stabilized Runge--Kutta methods overcome the stringent stability condition of standard methods without sacrificing explicitness. However, when stiffness is only induced by a few components, as in the presence of spatially local mesh refinement, their efficiency deteriorates. 
To remove the crippling effect of a few severely stiff components on the entire system of differential equations, we derive a modified equation, whose stiffness solely depend on the remaining mildly stiff components. By applying stabilized Runge--Kutta methods to this modified equation, we then devise an explicit multirate Runge--Kutta--Chebyshev (mRKC) method whose stability conditions are independent of a few severely stiff components.
Stability of the mRKC method is proved for a model problem, whereas its efficiency and usefulness are demonstrated through a series of numerical experiments.

\end{abstract}

\section{Introduction}\label{sec:intro}
We consider the system of stiff (nonlinear) differential equations,
\begin{equation}\label{eq:ode}
y' =f(y)\vcentcolon = \ff(y)+\fs(y), \qquad\qquad
y(0)=y_0,
\end{equation}
where $f:\Rn\rightarrow\Rn$ splits into an expensive but only mildly stiff part, $\fs$, associated with relatively slow (S) time-scales, and a cheap but severely stiff part, $\ff$, associated with fast (F) time-scales. 
Typical applications include chemical reactions and electrical circuits with disparate time-scales, but also 
spatial discretizations of diffusion dominated (parabolic) partial differential equations (PDEs) with local mesh refinement. Semi-discrete parabolic PDEs, in particular, lead to large systems of stiff ordinary differential equations, 
where the eigenvalues of the Jacobian matrix, $\partial f / \partial y$, lie in a narrow strip along the negative real axis whose extent scales as $H^{-2}$ for a mesh size $H$. In the presence of local mesh refinement,
$\fs,\ff$ in \cref{eq:ode} then correspond to discrete diffusion operators in the coarse and locally refined region of the mesh, respectively. Although $\ff$ involves only a small number of degrees of freedom, 
the extreme eigenvalue of its Jacobian will determine the spectral radius $\rho$ of $\partial f / \partial y$.

In contrast to multiscale methods \cite{AbP12,E03,EnT05,Van03}, 
{\it we do not assume any scale separation} in \cref{eq:ode}; hence, $\ff$ may contain {\it both fast and slow} scales. 
In a situation of local mesh refinement, for instance, when $\fs$ and $\ff$ each represent the discrete Laplacian in the coarse and refined regions, both Jacobians in fact contain small eigenvalues in magnitude; hence, the spectrum of $\partial f / \partial y$ cannot simply be split into fast and slow modes, as in \cref{fig:noscalesep}. This stands in sharp contrast to the underlying assumption of recently introduced multiscale methods for stiff (dissipative) ODEs, such as the  heterogeneous multiscale method (HMM) 
\cite{E03,EnT05} or  the projective method \cite{GIK03} which all require scale separation, as in \cref{fig:scalesep}. 

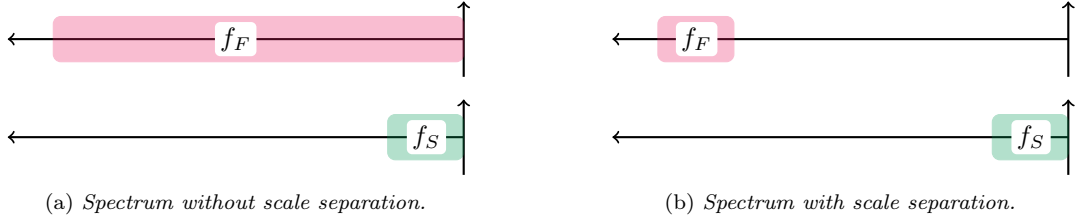
\begin{figure}
		\begin{subfigure}[t]{\subfigsize\textwidth}
			\centering
			\begin{tikzpicture}[scale=\plotimscale]
			\coordinate (a) at (0,0);
			\coordinate (b) at (6,0);
			\coordinate (c) at (6,0.5);
			\coordinate (d) at (6,-0.5);
			\coordinate (e) at (0,-1.3);
			\coordinate (f) at (6,-1.3);
			\coordinate (g) at (6,-0.8);
			\coordinate (h) at (6,-1.8);
			\coordinate (i) at (5,-1.3);
			\coordinate (l) at (3,0);
			\draw[thick,<-] (a)--(b);
			\draw[thick,<-] (c)--(d);
			\draw[thick,<-] (e)--(f);
			\draw[thick,<-] (g)--(h);
			\draw[\colorthree,fill=\colorthree,rounded corners=1mm,opacity=0.3] ([xshift=6mm,yshift=-3mm]a) rectangle ([yshift=3mm]b);
			\draw[\colorone,fill=\colorone,rounded corners=1mm,opacity=0.3] ([xshift=0mm,yshift=-3mm]i) rectangle ([yshift=3mm]f);
			\node[fill=white,rounded corners=2pt,inner sep=2pt] at ([xshift=0mm]l) {{$\ff$}};
			\node[fill=white,rounded corners=2pt,inner sep=2pt] at ([xshift=5.1mm]i) {{$\fs$}};
			\end{tikzpicture}
			\caption{Spectrum without scale separation.}
			\label{fig:noscalesep}
		\end{subfigure}\hfill%
		\begin{subfigure}[t]{\subfigsize\textwidth}
			\centering
			\begin{tikzpicture}[scale=\plotimscale]
			\coordinate (a) at (0,0);
			\coordinate (b) at (6,0);
			\coordinate (c) at (6,0.5);
			\coordinate (d) at (6,-0.5);
			\coordinate (e) at (0,-1.3);
			\coordinate (f) at (6,-1.3);
			\coordinate (g) at (6,-0.8);
			\coordinate (h) at (6,-1.8);
			\coordinate (i) at (5,-1.3);
			\coordinate (l) at (3,0);
			\draw[thick,<-] (a)--(b);
			\draw[thick,<-] (c)--(d);
			\draw[thick,<-] (e)--(f);
			\draw[thick,<-] (g)--(h);
			\draw[\colorthree,fill=\colorthree,rounded corners=1mm,opacity=0.3] ([xshift=6mm,yshift=-3mm]a) rectangle ([xshift=-14mm,yshift=3mm]l);
			\draw[\colorone,fill=\colorone,rounded corners=1mm,opacity=0.3] ([xshift=0mm,yshift=-3mm]i) rectangle ([yshift=3mm]f);
			\node[fill=white,rounded corners=2pt,inner sep=2pt] at ([xshift=-18.9mm]l) {{$\ff$}};
			\node[fill=white,rounded corners=2pt,inner sep=2pt] at ([xshift=5.1mm]i) {{$\fs$}};
			\end{tikzpicture}
			\caption{Spectrum with scale separation.}
			\label{fig:scalesep}
		\end{subfigure}%
	\caption{Stiff problems with identical spectral radii but different eigenvalues distribution.}
	\label{fig:spectrum}
\end{figure}

Standard explicit methods are notoriously inefficient for stiff differential equations due to their stringent stability constraint on the step size, $\tau$, which for parabolic problems must be proportional to $H^{2}$. Implicit methods, on the other hand, are unconditionally stable but require at every time step the solution of an $n \times n$ linear (or possibly nonlinear) system of equations, a high price to pay when $n$ is large. 
Moreover, when sheer size calls for using iterative methods, the overall performance heavily relies on the availability of efficient preconditioners while the convergence of Newton-like nonlinear iterations is not even guaranteed for larger step sizes.

{\it Stabilized Runge--Kutta (RK) (or Chebyshev) methods} fall somewhere between explicit and implicit methods: they are explicit and thus avoid the solution of large systems of equations, while their stability interval on the negative real axis is proportional to $s^2$ for an $s$-stage method. Thanks to this remarkable quadratic dependency, the work load (number of stages $s$) per time step only needs to scale linearly with $H^{-1}$ for parabolic PDEs, in contrast to the quadratic increase in the number of time steps required by standard explicit integrators. Stabilized RK methods are thus particularly efficient for the
time integration of large-scale, possibly nonlinear, parabolic PDEs \cite{DDD13}.
Several stabilized RK methods have been proposed in the literature, such as DUMKA methods, based on the composition of Euler steps \cite{Leb94,LeM94,Med98}, Runge--Kutta--Chebyshev (RKC) methods, based on the linear combination of Chebyshev polynomials \cite{SSV98,HoS80,VHS90} and orthogonal Runge--Kutta--Chebyshev methods (ROCK), based on optimal orthogonal stabilized functions \cite{Abd02,AbM01}; note that ROCK and RKC methods differ only beyond order one. 
Still, when applied to \cref{eq:ode}, the number of stages $s$ of any standard stabilized RK method will, yet again, be determined by the stiffest part $\ff$ and the method eventually become inefficient.

To overcome the stringent step size restriction due to the cheap but stiffer part, $\ff$, while retaining the efficiency of explicit time integration for $\fs$, \emph{multirate methods} 
use a smaller step size, or even an entirely different scheme, for integrating $\ff$. 
Since the early work of Rice \cite{Ric60} and the work of Gear and Wells \cite{GeW84} who proposed a number of multirate strategies for the interlaced time integration of the ``fast'' and ``slow'' components using classical multistep schemes,
various explicit, implicit or hybrid multirate schemes have been developed based on Runge--Kutta methods using splitting techniques for ``fast'' and ``slow'' components or extrapolation techniques \cite{And79,EnL97,GKR01,GuR93,Hof76,Kva99,SHV07,SaM10,SkA89}.
All these methods require a predictor step and either interpolate or extrapolate between ``fast'' and ``slow'' state variables, which is prone to instability. 
Although some of the implicit-explicit (IMEX) methods are provably stable, they are more cumbersome to implement and rapidly become too expensive as the number of ``fast'' unknowns increases.

More recently, Günther and Sandu exploited the generalized additive Runge--Kutta (GARK) framework \cite{SaG15} to devise multirate GARK (MrGARK) methods \cite{GuS16}. Many explicit, implicit and hybrid schemes are developed (up to fourth order accuracy) in \cite{RLS21,SRS19}, although some degree of implicitness is typically required to achieve a larger stability domain. In \cite{RSS20,San19,SeR19},
multirate infinitesimal step (MIS) methods \cite{KnW98,WKG09}, which assume the fast variables integrated exactly, are recast into the GARK framework and further generalized. The resulting multirate infinitesimal (MRI)-GARK methods extend exponential integrators \cite{HoO10} to the nonlinear case, where again the fast (nonlinear) dynamics are integrated exactly. When these semi-discrete schemes are used in practice, that is, in a fully discrete setting, some implicitness or very small steps sizes are typically required for stability.

Local adaptivity and mesh refinement are certainly key to the efficient numerical solution of PDEs with
heterogeneous media or complex geometry. Locally refined meshes, however, also cause a severe
bottleneck for any standard explicit time integration, as the maximal time-step is
dictated by maybe a few small elements in the mesh. To overcome the crippling effect 
of local mesh refinement, various multirate (or local time-stepping) methods \cite{GMM15} were 
proposed following the original local adaptive mesh refinement (AMR) strategy for first-order hyperbolic conservation laws by Berger and Oliger \cite{BeO84} --- see \cite{GaH13} for a review.
For parabolic problems, Ewing et al. \cite{ELV94,ELV90} derived and analyzed implicit 
finite difference schemes when local refinement is utilized in space and time. 
Dawson, Du and Dupont \cite{DQD91} combined implicit time integration in subdomains with
an explicit treatment of the interfaces, which leads to a decoupled but conditionally stable system. 
In \cite{MAM06,ShV00}, various predictor-corrector and domain decomposition methods were combined to iteratively correct the solution or its boundary values at artificial interfaces.  
By using static-regridding, Trompert and Verwer \cite{TrV91,TrV93a,TrV93b,TrV93c} developed a number of multirate time-stepping strategies for local uniform grid refinement (LUGR), where a first integration is performed on a global coarse grid and the accuracy is iteratively improved locally on 
nested and increasingly finer subgrids.

In contrast to the above implicit, or locally implicit, multirate strategies, 
fully explicit RKC time integration was recently combined with the AMR approach \cite{AbR20a,Mir17} 
to tackle diffusion dominated problems. Again, the mesh is divided into two distinct regions, 
the “coarse region,” which contains the larger elements
and corresponds to the mildly stiff part $\fs$, and the “fine region,” which contains the smallest elements
and thus corresponds to the severly stiff component $\ff$. In either subregion, the 
number of stages $s$ in chosen according to the local mesh size, while 
``ghost cell'' values at the coarse-to-fine interface are obtained by interpolating in time between stage values.
For certain problems, however, time interpolation of missing stage values from the other RKC method
can cause numerical instabilities \cite{AbR20a}.

To overcome the stringent stability condition due to a few severely stiff degrees of freedom, we first introduce in 
\cref{sec:afequ} a {\it modified equation}, where the 
spectral radius of its right-hand side, or {\it averaged force}, is bounded by that of the slower term $\fs$ yet still remains a good approximation of \cref{eq:ode}. 
Evaluation of this averaged force requires the solution to a stiff, but cheap, auxiliary problem over short time and forms the basis of our multirate strategy. 
In fact, the numerical integration of the modified equation by any explicit method via such a multirate approach
will be more efficient than integrating \cref{eq:ode} directly with the same explicit method.
In \cref{sec:mrkca}, we devise a multirate Runge--Kutta--Chebyshev (RKC) method explicit in both the fast and slow scales by utilizing
two different RKC methods to integrate the modified equation and evaluate the averaged force.
The resulting multirate RKC (mRKC) method assumes no scale separation, 
requires no interpolation between stages, and remains accurate even if the roles of $\ff$ and $\fs$
change in time. 
The fully discrete stability and accuracy analysis of the  mRKC scheme is given in 
\Cref{sec:stab_conv}.
Finally, in \cref{sec:numexp}, we apply our mRKC method to a
series of test problems from both stiff ordinary and partial differential equations to demonstrate its usefulness and efficiency.

\section{Averaged force, modified equation and multirate algorithm}\label{sec:afequ}
First, we introduce the {\it modified equation} where $f$ in \cref{eq:ode}
is replaced by an {\it averaged force}, $\fe$, which depends on a free parameter $\eta\geq 0$. 
For $\eta=0$ it holds $\fe=f$ whereas for $\eta>0$, the spectrum of $\fe$ is compressed 
and thus $\fe$ is less stiff
than $f$, see \cref{fig:spectrumfeta}. In fact for $\eta>0$ sufficiently large, the spectral radius $\rhoe$ of the 
Jacobian of $\fe$ is bounded by the spectral radius $\rhos$ of the Jacobian of $\fs$, 
i.e. $\rhoe\leq \rhos$; then, the stiffness of the modified equation depends solely on $\fs$ 
and its integration by any explicit method is cheaper than \cref{eq:ode} integrated with the same method. 
Since the condition $\rhoe\leq \rhos$ is already satisfied for $\eta$ relatively small, $\fe$ actually remains a good approximation of $f$. 
Next, we devise a multirate strategy based on the modified equation, which is implemented in \cref{sec:mrkca} using two separate RKC methods. Finally, we analyze the properties of $\fe$, derive a priori error bounds for the solution of the modified equation and perform a stability analysis.

\begin{figure}
		\begin{subfigure}[t]{0.3\textwidth}
			\centering
			\begin{tikzpicture}[scale=\plotimscale]
			\coordinate (a) at (0,0);
			\coordinate (b) at (3.5,0);
			\coordinate (c) at (3.5,0.5);
			\coordinate (d) at (3.5,-0.5);
			\coordinate (e) at (0,-1.3);
			\coordinate (f) at (3.5,-1.3);
			\coordinate (g) at (3.5,-0.8);
			\coordinate (h) at (3.5,-1.8);
			\coordinate (i) at (2.5,-1.3);
			\coordinate (l) at (1.75,0);
			\coordinate (p) at (4,0);
			\coordinate (eq) at (4,-1.3);
			\coordinate (bar1) at (0,-2.15);
			\coordinate (bar2) at (4.2,-2.15);
			\coordinate (q) at (0,-3);
			\coordinate (r) at (3.5,-3);
			\coordinate (s) at (3.5,-2.5);
			\coordinate (t) at (3.5,-3.5);
			\coordinate (lb) at (1.75,-3);
			\draw[thick,<-] (a)--(b);
			\draw[thick,<-] (c)--(d);
			\draw[thick,<-] (e)--(f);
			\draw[thick,<-] (g)--(h);
			\draw[\colorthree,fill=\colorthree,rounded corners=1mm,opacity=0.3] ([xshift=2mm,yshift=-3mm]a) rectangle ([yshift=3mm]b);
			\draw[\colorone,fill=\colorone,rounded corners=1mm,opacity=0.3] ([xshift=0mm,yshift=-3mm]i) rectangle ([yshift=3mm]f);
			\node[fill=white,rounded corners=2pt,inner sep=2pt] at ([xshift=0mm]l) {{$\ff$}};
			\node[fill=white,rounded corners=2pt,inner sep=2pt] at ([xshift=5.1mm]i) {{$\fs$}};
			\node at (p) {{\large $+$}};
			\node at (eq) {{\large $=$}};
			\draw[thick] (bar1)--(bar2);
			\draw[thick,<-] (q)--(r);
			\draw[thick,<-] (s)--(t);
			\draw[\colorthree,fill=\colorthree,rounded corners=1mm,opacity=0.3] ([xshift=2mm,yshift=-3mm]q) rectangle ([yshift=3mm]r);
			\node[fill=white,rounded corners=2pt,inner sep=2pt] at ([xshift=0mm]lb) {{$f$}};
			\end{tikzpicture}
			\caption{For $\eta=~0$, $f = \fe$ .}
			\label{fig:origspe}
		\end{subfigure}\hfill %
		\begin{subfigure}[t]{0.3\textwidth}
			\centering
			\begin{tikzpicture}[scale=\plotimscale]
			\coordinate (a) at (0,0);
			\coordinate (b) at (3.5,0);
			\coordinate (c) at (3.5,0.5);
			\coordinate (d) at (3.5,-0.5);
			\coordinate (e) at (0,-1.3);
			\coordinate (f) at (3.5,-1.3);
			\coordinate (g) at (3.5,-0.8);
			\coordinate (h) at (3.5,-1.8);
			\coordinate (i) at (2.5,-1.3);
			\coordinate (l) at (1.75,0);
			\coordinate (p) at (4,0);
			\coordinate (eq) at (4,-1.3);
			\coordinate (bar1) at (0,-2.15);
			\coordinate (bar2) at (4.2,-2.15);
			\coordinate (q) at (0,-3);
			\coordinate (r) at (3.5,-3);
			\coordinate (s) at (3.5,-2.5);
			\coordinate (t) at (3.5,-3.5);
			\coordinate (lb) at (1.75,-3);
			\draw[thick,<-] (a)--(b);
			\draw[thick,<-] (c)--(d);
			\draw[thick,<-] (e)--(f);
			\draw[thick,<-] (g)--(h);
			\draw[\colorfour,fill=\colorfour,rounded corners=1mm,opacity=0.3] ([xshift=12mm,yshift=-3mm]a) rectangle ([yshift=3mm]b);
			\draw[\colorone,fill=\colorone,rounded corners=1mm,opacity=0.3] ([xshift=0mm,yshift=-3mm]i) rectangle ([yshift=3mm]f);
			\node[fill=white,rounded corners=2pt,inner sep=2pt] at ([xshift=6mm]l) {$\overline{\ff}$};
			\node[fill=white,rounded corners=2pt,inner sep=2pt] at ([xshift=5.1mm]i) {{$\fs$}};
			\node at (p) {{\large $+$}};
			\node at (eq) {{\large $=$}};
			\draw[thick] (bar1)--(bar2);
			\draw[thick,<-] (q)--(r);
			\draw[thick,<-] (s)--(t);
			\draw[\colorfour,fill=\colorfour,rounded corners=1mm,opacity=0.3] ([xshift=12mm,yshift=-3mm]q) rectangle ([yshift=3mm]r);
			\node[fill=white,rounded corners=2pt,inner sep=2pt] at ([xshift=6mm]lb) {{$\fe$}};
			\end{tikzpicture}
			\caption{For $\eta>0 $, the fast term is replaced by a mean $\overline{\ff}$ and $\fe$ becomes less stiff than $f$.}
			\label{fig:modspea}
		\end{subfigure}\hfill %
	\begin{subfigure}[t]{0.3\textwidth}
		\centering
		\begin{tikzpicture}[scale=\plotimscale]
		\coordinate (a) at (0,0);
		\coordinate (b) at (3.5,0);
		\coordinate (c) at (3.5,0.5);
		\coordinate (d) at (3.5,-0.5);
		\coordinate (e) at (0,-1.3);
		\coordinate (f) at (3.5,-1.3);
		\coordinate (g) at (3.5,-0.8);
		\coordinate (h) at (3.5,-1.8);
		\coordinate (i) at (2.5,-1.3);
		\coordinate (ib) at (2.5,0);
		\coordinate (ic) at (2.5,-3);
		\coordinate (l) at (1.75,0);
		\coordinate (p) at (4,0);
		\coordinate (eq) at (4,-1.3);
		\coordinate (bar1) at (0,-2.15);
		\coordinate (bar2) at (4.2,-2.15);
		\coordinate (q) at (0,-3);
		\coordinate (r) at (3.5,-3);
		\coordinate (s) at (3.5,-2.5);
		\coordinate (t) at (3.5,-3.5);
		\coordinate (lb) at (1.75,-3);
		\draw[thick,<-] (a)--(b);
		\draw[thick,<-] (c)--(d);
		\draw[thick,<-] (e)--(f);
		\draw[thick,<-] (g)--(h);
		\draw[\colorone,fill=\colorone,rounded corners=1mm,opacity=0.3] ([xshift=0mm,yshift=-3mm]ib) rectangle ([yshift=3mm]b);
		\draw[\colorone,fill=\colorone,rounded corners=1mm,opacity=0.3] ([xshift=0mm,yshift=-3mm]i) rectangle ([yshift=3mm]f);
		\node[fill=white,rounded corners=2pt,inner sep=2pt] at ([xshift=5.1mm]ib) {$\overline{\ff}$};
		\node[fill=white,rounded corners=2pt,inner sep=2pt] at ([xshift=5.1mm]i) {{$\fs$}};
		\node at (p) {{\large $+$}};
		\node at (eq) {{\large $=$}};
		\draw[thick] (bar1)--(bar2);
		\draw[thick,<-] (q)--(r);
		\draw[thick,<-] (s)--(t);
		\draw[\colorone,fill=\colorone,rounded corners=1mm,opacity=0.3] ([xshift=0mm,yshift=-3mm]ic) rectangle ([yshift=3mm]r);
		\node[fill=white,rounded corners=2pt,inner sep=2pt] at ([xshift=5.1mm]ic) {{$\fe$}};
		\end{tikzpicture}
		\caption{For $\eta>~0$ sufficiently large, the spectrum of $\fe$ and $\fs$ are comparable.}
		\label{fig:modspeb}
	\end{subfigure}%
	\caption{Spectrum of $f$ and of the averaged force $\fe$ for varying $\eta$.}
	\label{fig:spectrumfeta}
\end{figure}
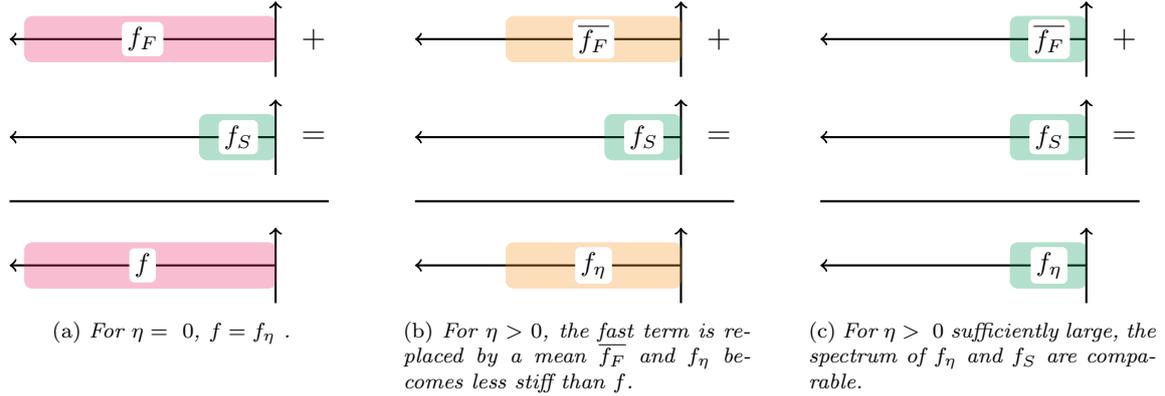

\subsection{Averaged force and modified equation}
We now define an average $\fe$ of $f$ such that the solution $\ye$ of {\it the modified equation},
\begin{equation}\label{eq:odemod}
\ye'=\fe(\ye), \qquad\qquad
\ye(0)=y_0
\end{equation}
is a good approximation of the exact solution $y$ of \cref{eq:ode}, yet the stiffness of \cref{eq:odemod} only depends on $\fs$.
\begin{definition}\label{def:fe}
	For $\eta>0$, the averaged force $\fe:\Rn\rightarrow\Rn$ is defined as
	\begin{equation}\label{eq:deffedif}
		\fe(y)=\frac{1}{\eta}(u(\eta)-y),
	\end{equation}
	where {\it the auxiliary solution} $u:[0,\eta]\rightarrow \Rn$ is defined by {\it the auxiliary equation}
	\begin{align}\label{eq:defu}
	u'&=\ff(u)+\fs(y), & 
	u(0)=y.
	\end{align}
	For $\eta=0$, let $f_0=f$ (note that $f_0=\lim_{\eta\to 0^+}\fe$).
\end{definition}
Therefore, whenever $\fe(\ye(t))$ is evaluated in \cref{eq:odemod}, the auxiliary problem \cref{eq:defu} must be solved over the time interval $[0,\eta]$ with initial value $u(0)=\ye(t)$.
From \cref{eq:defu,eq:deffedif}, it follows that
\begin{equation}\label{eq:deffeint}
\fe(y)=\frac{1}{\eta}\int_0^\eta u'(s)\dif s=\fs(y)+\frac{1}{\eta}\int_0^\eta \ff(u(s))\dif s.
\end{equation}
Hence, $\fe$ is an average of $f$ over the time interval $[0,\eta]$ along the {\it auxiliary solution} $u$.
For $\eta$ sufficiently large, we show in \cref{sec:stab_modeq} that $\rhoe\leq\rhos$.

\paragraph{Multirate strategy.} 
Starting from the modified equation \cref{eq:odemod}, we propose the following explicit multirate strategy for the numerical approximation of \cref{eq:ode}: 
Solve \cref{eq:odemod} with an explicit numerical method, where stability conditions depend on $\fs$ only. Whenever an evaluation of $\fe$ is needed, \cref{eq:defu} is solved with a (possibly different) explicit method, with stability conditions depending on $\ff$. Since $\ff$ is cheap, evaluating $\fe$ in \cref{eq:odemod} by solving \cref{eq:defu} carries about the same computational cost as evaluating $f$  in \cref{eq:ode}. On the other hand, since the stiffness, and hence the step size, needed for \cref{eq:odemod} no longer depends on the fastest scales in the problem, the number of expensive $\fs$ evaluations will be greatly reduced. 

Although the above multirate strategy might at first resemble recently introduced multiscale methods, such as HMM for dissipative ODEs, it is fundamentally different because the averaged equation \cref{eq:odemod}
requires no scale separation: both $\fs$ and $\ff$ may contain slow scales, in contrast to the effective equations derived in \cite{E03}, for instance.
The above multirate strategy might also resemble multirate infinitesimal step (MIS) methods \cite{KnW98,WKG09}. MIS methods, however, only discretize the slow variables
while assuming the fast dynamics to be integrated exactly; hence, the fast variables do not produce any instability. In contrast, the multirate methods introduced here discretize the modified equation \cref{eq:odemod}, where the spectrum of the fast dynamics has been compressed. Moreover, we provide a fully discrete stability and accuracy analysis below.
Finally, an auxiliary problem similar to \cref{eq:defu} also appears in the context of second-order ODEs \cite{HoL99} (see also \cite[VIII.4.2]{HLW06}). In \cite{HoL99}, however, the auxiliary problem is integrated over the entire interval $[-\tau,\tau]$ using smaller step sizes $\tau/N$, while $y_{n+1},y_{n+1}'$ are defined through finite difference approximations of $u$. In contrast, here we solve \cref{eq:defu} over the small time interval $[0,\eta]$ with $\eta\ll \tau$, while $u$ is used to compute $\fe$, which in turn defines the equation for $\ye$.

\subsection{A priori error analysis for the solution of the modified equation}
Here we analyze the effect of the parameter $\eta$ on $\fe$ and
show that $\fe$ satisfies a one-sided Lipschitz condition, which is fundamental for proving convergence and contractivity for general nonlinear problems \cite[IV.12]{HaW02}. Then, we derive bounds on the error introduced by solving \cref{eq:odemod} instead of \cref{eq:ode}, which are independent of the problem's stiffness.

Let $\langle\cdot,\cdot\rangle$ and $\nld{\cdot}$ denote the standard Euclidean scalar product and 
norm in $\Rn$, respectively. To begin, we prove that $\ff$ has a smoothing effect on $\fe$, if it
satisfies a one-sided Lipschitz condition.
\begin{lemma}\label{lemma:phif}
Let $\mu_F\in\Rb$ and $\ff$ satisfy
\begin{equation}\label{eq:contrff}
\langle \ff(z)-\ff(y),z-y\rangle\leq \mu_F\nld{z-y}^2\quad \forall z,y\in\Rn.
\end{equation}
Then 
\begin{equation}\label{eq:defphi}
\begin{aligned}
\nld{\fe(y)}&\leq \varphi(\eta \mu_F)\nld{f(y)}, &\mbox{where}&& \varphi(z)&=\frac{e^z-1}{z} \;\mbox{ for }\;z\neq 0
\end{aligned}
\end{equation}
and $\varphi(0)=1$ is defined by continuous extension. Moreover, if $\ff(y)=A_F\,y$ with $A_F\in\Rb^{n\times n}$, then
\begin{equation}\label{eq:deffephi}
\fe(y)=\varphi(\eta A_F)f(y).
\end{equation}

\end{lemma}
\begin{proof}
Let $v:[0,\eta]\rightarrow\Rn$ be defined by $v(s)=y$ for all $s$. We set
\begin{equation}
\delta \vcentcolon = \nld{v'(s)-\ff(v(s))-\fs(y)}=\nld{f(y)}.
\end{equation}
Since the logarithmic norm of the Jacobian of $\ff$ is bounded by $\mu_F$, we obtain from a classical result on differential inequalities (see \cite[Chapter I.10, Theorem 10.6]{HNW08}) 
\begin{equation}
\nld{u(\eta)-y}=\nld{u(\eta)-v(\eta)}\leq e^{\eta\mu_F}\int_0^\eta e^{-s\mu_F}\delta\dif s=\eta\varphi(\eta\mu_F)\nld{f(y)},
\end{equation}
which yields \cref{eq:defphi} by \cref{eq:deffedif}. 
Now, let $\ff(u)=A_F\, u$ with $A_F\in\Rb^{n\times n}$ nonsingular. Then, the variation-of-constants formula with $\ff(u)=A_F\, u$ in \cref{eq:defu} yields
\begin{equation}
u(\eta)= e^{A_F\eta}\left(y+\int_0^\eta e^{-A_Fs}\fs(y)\dif s\right) = e^{A_F\eta}y+A_F^{-1}(e^{A_F\eta}-I)\fs(y),
\end{equation}
with $I$ the identity matrix. Hence,
\begin{equation}\label{eq:defulinA}
u(\eta) = e^{A_F\eta}y + \eta\varphi(\eta A_F)\fs(y).
\end{equation}
Since $\varphi$ has no poles, $u(\eta)$ as in \cref{eq:defulinA} is well-defined and satisfies \cref{eq:defu} for all matrices $A_F$. By using \cref{eq:deffedif}, we thus obtain
\begin{equation}
\fe(y)=\frac{1}{\eta}(e^{A_F\eta}-I)y+\varphi(\eta A_F)\fs(y)=\varphi(\eta A_F)(A_Fy+\fs(y))=\varphi(\eta A_F)f(y).  \tag*{\qedhere}
\end{equation}
\end{proof}

\begin{remark}
The entire function $\varphi(z)$ is quite common in the theory of exponential integrators \cite{HoO10}. Indeed, when $\ff(y)=A_F\, y$, the solution \cref{eq:defulinA} to the auxiliary problem \cref{eq:defu} corresponds to a single step of the \textit{exponential Euler} method applied to \cref{eq:ode}. Our multirate approach, however,
differs from exponential integrators: First, $u(\eta)$ is just an auxiliary solution used to compute $\fe$, which
is distinct from the solution $\ye$ of the modified equation \cref{eq:odemod}. Second, we do not use an exponential integrator but an RKC method to obtain $u(\eta)$; thus, $\varphi(\eta A_F)$ is never computed explicitly. Third, $\eta$ is not the step size here but a free parameter indicating the length of the integration interval in \cref{eq:defu}.
\end{remark}

The function  $\varphi(z)$, shown in \cref{fig:phi}, satisfies 
\begin{equation}
\varphi(0)=1, \qquad \varphi'(0) = \frac 12, \qquad \lim_{z\to-\infty}\varphi(z)=0, \qquad 0<\varphi(z)<1,\quad \forall z<0. 
\end{equation}
Hence, if $A_F$ is negative definite, multiplication of $f$ by $\varphi(\eta A_F)$ in \cref{eq:deffephi} has a smoothing effect, which can be tuned by varying $\eta\geq 0$ --- see \cref{thm:stifflinscamod} below. A similar property holds for any nonlinear $\ff$ that is contractive, i.e. with $\mu_F\leq 0$ in \cref{eq:contrff}, because of \cref{eq:defphi}. 

Next, we prove under the assumption that the Jacobians of $\ff$ and $\fs$ commute that the averaged force $\fe$ satisfies a one-sided Lipschitz condition. Clearly, commutativity of the Jacobians is a rather strong assumption, rarely satisfied in practice. It is merely used here to provide insight into the behavior of $\fe$ and in fact not needed when subsequently applying the multirate method.
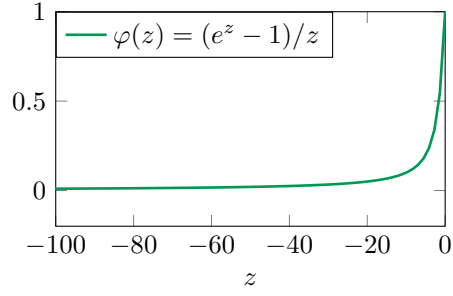
\begin{figure}[t]
	\centering
	\begin{tikzpicture}[scale=\plotimscale]
	\begin{axis}[height=\aspectratio*\plotimsizeu\textwidth,width=\plotimsizeu\textwidth, ymin=-0.2, ymax=1,xmax=0,xmin=-100,legend columns=1,legend style={draw=\legendboxdraw,fill=\legendboxfill,at={(0,1)},anchor=north west},
	xlabel={$z$}, ylabel={},label style={font=\normalsize},tick label style={font=\normalsize},legend image post style={scale=\legendmarkscale},legend style={nodes={scale=\legendfontscale, transform shape}},grid=none]
	\addplot[color=\colorone,line width=\plotlinewidth pt,mark=none] table [x=z,y=phi,col sep=comma] 
	{data/text/Phi_8.csv};\addlegendentry{$\varphi(z)=(e^z-1)/z$}
	\end{axis}
	\end{tikzpicture}
	\caption{The entire function $\varphi(z)$.}
	\label{fig:phi}
\end{figure}
\begin{theorem}\label{thm:fecontr}
Let $A_F\in\Rb^{n\times n}$ be symmetric and $\ff(y)=A_F\, y$. Suppose
\begin{equation}\label{eq:lipcondf}
\langle \frac{\partial f}{\partial y}(w)(z-y),z-y\rangle\leq  \mu \nld{z-y}^2 \quad \forall w,y,z\in\Rn
\end{equation}
with $\mu\leq 0$ and that $A_F\frac{\partial\fs}{\partial y}(w)=\frac{\partial\fs}{\partial y}(w) A_F$ for all $w\in\Rn$. Then,
\begin{equation}
\langle \fe(z)-\fe(y),z-y\rangle \leq \mu_\eta \nld{z-y}^2,
\end{equation}
where $\mu_\eta = \mu \min_{\lambda\in\lambda(A_F)}\{\varphi(\eta\lambda)\}\leq 0$ and $\lambda(A_F)$ is the spectrum of $A_F$.
\end{theorem}
\begin{proof}
Let $y,z\in\Rn$ and $w(r)=rz+(1-r)y$ for $r\in [0,1]$. We have
\begin{align}
\langle \fe(z)-\fe(y),z-y\rangle &= \langle\varphi(\eta A_F)(f(z)-f(y)),z-y\rangle \\
&= \langle \varphi(\eta A_F)^{1/2}\int_0^1 \frac{\partial f}{\partial y}(w(r))(z-y)\dif r,\varphi(\eta A_F)^{1/2}(z-y)\rangle.
\end{align}
Since $\varphi(z)>0$ for all $z$, $\varphi(\eta A_F)$ is symmetric positive definite and $\varphi(\eta A_F)^{1/2}$ exists. By hypothesis, $\varphi(\eta A_F)^{1/2} $ and $\frac{\partial f}{\partial y}(w(r))$ commute. Therefore
\begin{align}
\langle \fe(z)-\fe(y),z-y\rangle &= \int_0^1\langle \frac{\partial f}{\partial y}(w(r))\varphi(\eta A_F)^{1/2}(z-y),\varphi(\eta A_F)^{1/2}(z-y)\rangle\dif r\\
&\leq \mu \nld{\varphi(\eta A_F)^{1/2}(z-y)}^2\leq \mu \min_{\lambda\in\lambda(A)}\{\varphi(\eta\lambda)\}\nld{z-y}^2. \tag*{\qedhere}
\end{align}
\end{proof}
\Cref{thm:fecontr} shows that $\fe$ indeed satisfies a one-sided Lipschitz condition, if the Jacobians of $\ff$, $\fs$ commute and \cref{eq:lipcondf} holds, which is slightly stronger than requiring that $f$ satisfies a one-sided Lipschitz condition. Indeed, $f$ is one-sided Lipschitz if, and  only if, \cref{eq:lipcondf} holds for all $z,y$ and $w\in [z,y]$, see \cite[I.10]{HNW08}. Next, we bound the error between the solutions of \cref{eq:ode,eq:odemod}.
\begin{theorem}\label{thm:erryeta}
Under the assumptions of \cref{thm:fecontr}, it holds
\begin{equation}\label{eq:bounderr}
\nld{y(t)-\ye(t)}\leq  \max_{\lambda\in\lambda(A_F)}|1-\varphi(\eta\lambda)|\int_0^t e^{\mu_\eta (t-s)}\nld{f(y(s))}\dif s,
\end{equation}
with $\mu_\eta= \mu \min_{\lambda\in\lambda(A_F)}\{\varphi(\eta\lambda)\}\leq 0$.
\end{theorem}
\begin{proof}
Let $y(t)$ be the solution to \cref{eq:ode}. From \cref{lemma:phif}, we have
\begin{align}
\nld{y'(t)-\fe(y(t))}&=\nld{f(y(t))-\fe(y(t))}=\nld{(I-\varphi(\eta A_F))f(y(t))}\\
&\leq \max_{\lambda\in\lambda(A_F)}|1-\varphi(\eta\lambda)|\nld{f(y(t))} =\vcentcolon \delta(t).
\end{align}
Since the logarithmic norm of the Jacobian of $\fe$ is bounded by $\mu_\eta = \mu \min_{\lambda\in\lambda(A_F)}\{\varphi(\eta\lambda)\}$, as implied by \cref{thm:fecontr}, the estimate \cref{eq:bounderr} follows from classical results on differential inequalities (see \cite[Chapter I.10, Theorem 10.6]{HNW08}).
\end{proof}
Note that the error bound \cref{eq:bounderr} is independent of the stiffness present in $\ff$
and that $\nld{y(t)-\ye(t)}\leq C\eta$ as $\eta\to 0$ because $\varphi(\eta\lambda)=1+\bigo{\eta}$.

\subsection{Stability analysis of the modified equation}\label{sec:stab_modeq}
We now study the stiffness of the modified equation \cref{eq:odemod} given by the spectral radius $\rhoe$ of the Jacobian of $\fe$. In particular, we determine necessary conditions for $\rhoe\leq\rhos$, 
with $\rhos$ the spectral radius of the Jacobian of $\fs$, and hence that the stiffness of the modified equation only depends on the slow components. 
As in \cref{thm:fecontr}, we assume that the Jacobians of $\ff$ and $\fs$ commute. In \cref{sec:stab_conv}, we shall analyze the stability of our multirate method, first under the same commutativity assumption but then also for a $2\times 2$ problem where the Jacobians do not commute. For more general problems, stability is verified numerically in \cref{sec:numexp}.

Let the Jacobians of $\ff$ and $\fs$ commute. Then, they are simultaneously triangularizable and the stability analysis of \cref{eq:ode,eq:odemod} reduces to the scalar \textit{multirate test equation}
\begin{equation}\label{eq:mtesteq}
y' = \lambda y+\zeta y,\qquad\qquad
y(0)=y_0,
\end{equation}
with $\lambda,\zeta\leq 0$ and $y_0\in\Rb$, which corresponds to setting $\ff(y)=\lambda y$ and $\fs(y)=\zeta y$; thus, $\rhof=|\lambda|$ and $\rhos=|\zeta|$. Since we do not assume any scale separation, 
$\lambda$ can take any nonpositive value. 

Since \cref{eq:mtesteq} satisfies the hypotheses of \cref{lemma:phif} with $\mu_F = A_F=\lambda$, we have
\begin{align}\label{eq:solutesteq}
u(\eta) &= (e^{\eta\lambda}+\varphi(\eta\lambda)\eta \zeta )y, \\ \label{eq:fetesteq}
\fe(y) &= \varphi(\eta\lambda)(\lambda+\zeta)y
\end{align}
and \cref{eq:odemod} reduces to
\begin{equation}\label{eq:modtesteq}
\ye'=\varphi(\eta\lambda)(\lambda+\zeta)\ye,\qquad\qquad
\ye(0)=y_0.
\end{equation}
Next, we detemine conditions on $\eta,\lambda$ and $\zeta$ which guarantee that
\begin{equation}
|\varphi(\eta\lambda)(\lambda+\zeta)|\leq|\zeta|,
\end{equation}
and hence that the stiffness of \cref{eq:modtesteq} exclusively depends on $\rhos=|\zeta|$. 
The following technical lemma is used to prove \cref{thm:stifflinscamod} below.

\begin{lemma}\label{lemma:cont_stab}
Let $w\leq 0$ and $\varphi(z)$ be given by \cref{eq:defphi}. Then, $\varphi(z)(z+w)\in [w,0]$ for all $z\leq 0$ if, and only if, $\varphi'(0)|w|\geq 1$, i.e. $|w|\geq 2$ since $\varphi'(0)=1/2$.
\end{lemma}
\begin{proof}
	For $z,w\leq 0$, the upper bound $\varphi(z)(z+w)\leq 0$ always holds. Hence, we only need to consider
	the lower bound, 
	\begin{equation}\label{eq:phiw}
	\varphi(z)(z+w)\geq w.
	\end{equation}
	Suppose that \cref{eq:phiw} holds for all $z\leq 0$. In a neighborhood of $z=0$, this yields
	\begin{align}
	0&\leq \varphi(z)(z+w)-w 
	= (\varphi(0)+\varphi'(0)z+\bigo{z^2})(z+w)-w \\ \label{eq:phiwasym}
	&= z(1+\varphi'(0)(z+w))+\bigo{z^2(z+w)},
	\end{align}
	where we have used that $\varphi(0)=1$. Dividing \cref{eq:phiwasym} by $z<0$ and letting $z\to 0$ yields $\varphi'(0)w\leq -1$, and thus $\varphi'(0)|w|\geq 1$. 
	
	Now, let $\varphi'(0)|w|\geq 1$ and hence $w\leq -1/\varphi'(0)=-2$.
	For $z=0$, \cref{eq:phiw} trivially holds. For $z<0$, we multiply \cref{eq:phiw} by $z$
	and prove the resulting equivalent condition:
	\begin{equation}
	\alpha(z)=zw+(1-e^z)(z+w)\geq 0\quad \forall z<0.
	\end{equation}
	Since
	\begin{equation}
	\alpha'(z) =1+w-(1+w+z)e^z, \qquad \alpha''(z) = -(2+w+z)e^z,
	\end{equation}
	we have $\alpha(0)=\alpha'(0)=0$. Since $w\leq -2$, we have
	\begin{equation}
	\alpha(z)=\int_0^z \int_0^s \alpha''(r)\dif r\dif s = -\int_z^0 \int_s^0(2+w+r)e^r\dif r\dif s \geq 0,
	\end{equation}
	which concludes the proof.
\end{proof}
\begin{theorem}\label{thm:stifflinscamod}
Let $\zeta<0$. Then, $\varphi(\eta\lambda)(\lambda+\zeta)\in [\zeta,0]$ for all $\lambda\leq 0$ if, and only if, $\eta\geq 2/|\zeta|$.
\end{theorem}
\begin{proof}
Setting $z=\eta\lambda$ and $w=\eta\zeta$, we have that
\begin{align}
\varphi(\eta\lambda)(\lambda+\zeta)&\in [\zeta,0] &\mbox{is equivalent to} && \varphi(z)(z+w)\in &[w,0].
\end{align}
In view of \cref{lemma:cont_stab}, this holds for all $\lambda\leq 0$, if and only if $\eta|\zeta|=|w|\geq 2$.
\end{proof}
\Cref{thm:stifflinscamod} implies that for $\eta\geq 2/\rhos$ the stiffness of \cref{eq:modtesteq} depends only on the slow term $\fs$. 
Since $\eta$ does not depend on $\lambda$ and the result holds for all $\lambda\leq 0$, 
there is no need  for any assumption on scale separation.

\section{A stabilized method based on the modified equation: the multirate Runge--Kutta--Chebyshev method}\label{sec:mrkca}

Although the modified equation \cref{eq:odemod} has reduced stiffness, implementing a multirate strategy 
based on \cref{eq:odemod} and \cref{eq:defu} with classical explicit methods remains inefficient,
as it will lead to step size restrictions due to their inherent stiffness. Instead, 
we introduce here the mRKC method, which is based on two RKC methods and thus has no step size restrictions. Moreover, thanks to the multirate strategy, the number of $\fs$ (expensive) evaluations is independent of the stiffness of $\ff$ and thus no longer suffers from the efficiency loss of any classical stabilized scheme. 
In  \cref{sec:rkc} we briefly recall some basic definitions and properties of the RKC scheme before introducing the mRKC method in \cref{sec:fdrkcsq}

\subsection{Stabilized Runge--Kutta methods}\label{sec:rkc}
Stabilized Runge--Kutta methods \cite{Abd02,AbM01,Leb94,LeM94,Med98,SSV98,HoS80,VHS90} are explicit one-step Runge--Kutta (RK) methods with an extended stability domain along the negative real axis. 
By increasing the number of stages, with respect to classical RK methods, they relax the stringent constraint of standard explicit RK methods on the step size. Their construction relies on Chebyshev polynomials of the first kind, $T_s(x)$, and the classical first-order family of methods, the Runge--Kutta--Chebyshev (RKC) methods \cite{HoS80,VHS90}, is given by the $s$-stage RK method
\begin{equation}\label{eq:rkc}
\begin{aligned}
k_0&=y_n,\\
k_1 &= k_0+\mu_1\tau f(k_0),\\
k_j&= \nu_j k_{j-1}+\kappa_j k_{j-2}+\mu_j\tau f(k_{j-1}) \quad j=2,\ldots,s,\\
y_{n+1}&=k_s,
\end{aligned}
\end{equation}
where $\tau$ is the step size, $\mu_1 = \ou/\oz$ and
\begin{align}\label{eq:defcoeff} 
	\mu_j&= 2\ou  b_j/b_{j-1}, &
	\nu_j&= 2\oz b_j/b_{j-1},  &
	\kappa_j&=-b_j/b_{j-2}  &
	\text{for }j&=2,\ldots,s,
\end{align}
with $\varepsilon\geq 0$, $\oz=1+\varepsilon/s^2$, $\ou=T_s(\oz)/T_s'(\oz)$ and $b_j=1/T_j(\oz)$ for $j=0,\ldots,s$. We note that the explicit Euler method is recovered for $s=1$. 

When applied to the test equation $y'=\lambda y$ and using
\begin{align}
T_0(x)&= 1, & T_1(x)&= x, & T_j(x)&= 2xT_{j-1}(x)-T_{j-2}(x),
\end{align}
\cref{eq:rkc} yields $y_{n+1}=R_s(z)y_n$, where $z=\tau\lambda$ and
\begin{equation}\label{eq:stabpolrkc}
	R_s(z)=b_sT_s(\oz+\ou z)
\end{equation}
is the stability polynomial of the method. As $T_s(x)\geq 1$ for $x\geq 1$ and $|T_s(x)|\leq 1$ for $x\in [-1,1]$ then $|R_s(z)|\leq 1$ for $z\in [-\ell^\varepsilon_s,0]$ and $\ell_s^\varepsilon=2\oz/\ou$. As $\beta s^2\leq \ell_s^\varepsilon$, for $\beta=2-4\varepsilon/3$, then $|z|\leq \beta s^2$ is a sufficient condition for stability \cite{Ver96} and the stability domain 
\begin{equation}
\mathcal{S}=\{ z\in\Cb\,:\, |\Pe_s(z)|\leq 1\} 
\end{equation}
increases quadratically, with respect to the stage number $s$, along the negative real axis (see \cref{fig:rkc_damped_stabdom}).
The parameter $\varepsilon\geq 0$ is a damping parameter introduced to increase stability in the imaginary direction \cite{GuL60} (compare \cref{fig:rkc_undamped_stabdom,fig:rkc_damped_stabdom} for $s=10$). In \cref{fig:rkc_stabpol} we also plot the stability polynomial of the RKC scheme for different values of $s$ and $\varepsilon$, observe as the stability domain increases quadratically with $s$ and as the polynomials satisfy $|R_s(z)|<1$ for $\varepsilon>0$.
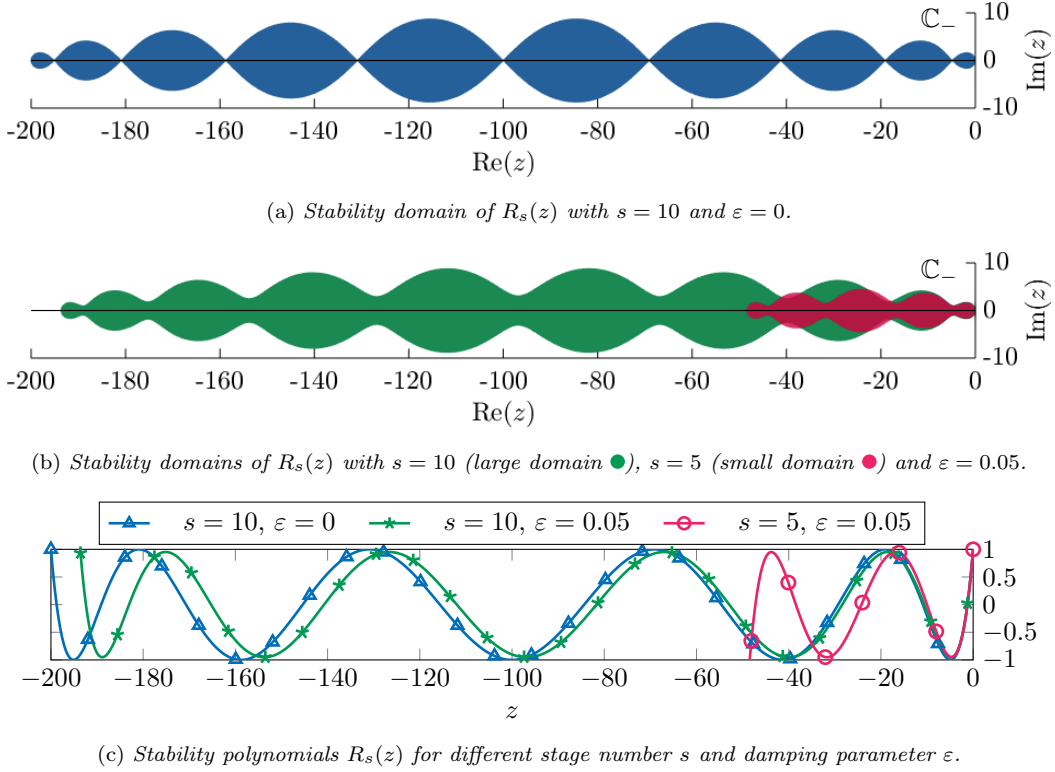
\begin{figure}[!tbp]
	\centering
		\begin{subfigure}[t]{\textwidth}
			\centering
			\begin{tikzpicture}
				\node at (0,0) {\includegraphics[trim=0cm 0cm 0cm 0cm, clip, width=0.9\textwidth]{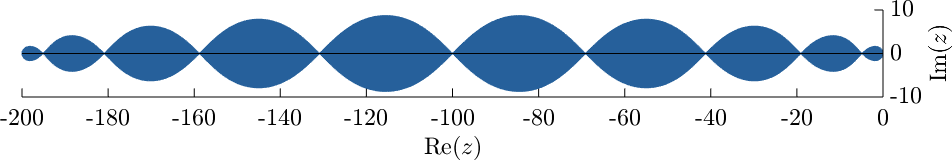}};
				\node at (5.4,0.9) {$\mathbb{C}_{-}$};
			\end{tikzpicture}
			\caption{Stability domain of $\Pe_s(z)$ with $s=10$ and $\varepsilon=0$.}
			\label{fig:rkc_undamped_stabdom}
		\end{subfigure}\vspace{2mm}\\%
	\begin{subfigure}[t]{\textwidth}
		\centering
		\begin{tikzpicture}
			\node at (0,0) {\includegraphics[trim=0cm 0cm 0cm 0cm, clip, width=0.9\textwidth]{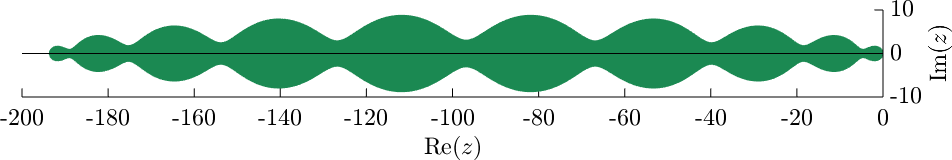}};
			\node at (0,0) {\includegraphics[trim=0cm 0cm 0cm 0cm, clip, width=0.9\textwidth]{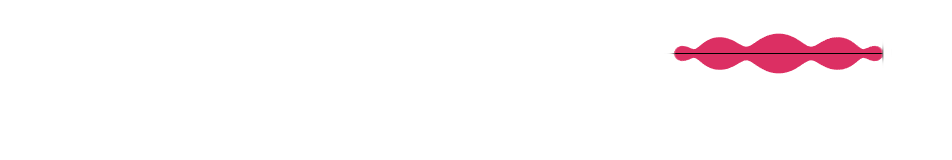}};
			\node at (5.4,0.9) {$\mathbb{C}_{-}$};
		\end{tikzpicture}
		\caption{Stability domains of $\Pe_s(z)$ with $s=10$ (large domain \begin{tikzpicture}
				\node[circle,draw=\colorone, fill=\colorone, inner sep=0pt,minimum size=5pt] at (0,0) {};
			\end{tikzpicture}), $s=5$ (small domain \begin{tikzpicture}
				\node[circle,draw=\colorthree, fill=\colorthree, inner sep=0pt,minimum size=5pt] at (0,0) {};
			\end{tikzpicture}) and $\varepsilon=0.05$.}
		\label{fig:rkc_damped_stabdom}
	\end{subfigure}\vspace{3mm}\\%
\begin{subfigure}[t]{\textwidth}
	\centering
	\begin{tikzpicture}[scale=\plotimscale]
		\begin{axis}[height=0.2\textwidth,width=0.9\textwidth, ymin=-1, ymax=1,xmax=0,xmin=-200,legend columns=3,legend cell align=left,legend style={draw=\legendboxdraw,fill=\legendboxfill,at={(0.5,1.04)},anchor=south,column sep=10pt},yticklabel pos=right,
			xlabel={$z$}, ylabel={},label style={font=\normalsize},tick label style={font=\normalsize},legend image post style={scale=\legendmarkscale},legend style={nodes={scale=\legendfontscale, transform shape}},grid=none]
			\addplot[color=\colortwo,line width=\plotlinewidth pt,mark=\marktwo, mark size=\plotmarksizeu pt, mark repeat=20,mark phase=0] table [x=xk,y=px,col sep=comma] 
			{data/text/RKC_1_s_10_nodamping_polynomial_data.csv};\addlegendentry{$s=10$, $\varepsilon=0$}
			\addplot[color=\colorone,line width=\plotlinewidth pt,mark=\markone, mark size=\plotmarksizeu pt, mark repeat=20,mark phase=0] table [x=xk,y=px,col sep=comma] 
			{data/text/RKC_1_s_10_damping_polynomial_data.csv};\addlegendentry{$s=10$, $\varepsilon=0.05$}
			\addplot[color=\colorthree,line width=\plotlinewidth pt,mark=\markthree, mark size=\plotmarksizeu pt, mark repeat=20,mark phase=0] table [x=xk,y=px,col sep=comma] 
			{data/text/RKC_1_s_5_damping_polynomial_data.csv};\addlegendentry{$s=5$, $\varepsilon=0.05$}
		\end{axis}
	\end{tikzpicture}
	\caption{Stability polynomials $\Pe_s(z)$ for different stage number $s$ and damping parameter $\varepsilon$.}
	\label{fig:rkc_stabpol}
\end{subfigure}%
	\caption{Stability domains and polynomials of the RKC method, for different stage number $s$ and damping $\varepsilon$.}
	\label{fig:rkc_dom_pol}
\end{figure}

For a general right-hand side $f:\Rn\rightarrow \Rn$ the number of stages $s$ in \cref{eq:rkc} is typically chosen such that $\tau\rho\leq \beta s^2$, where $\rho$ is the spectral radius of the Jacobian of $f$. In \eqref{eq:rkc} and below, we consider autonomous problems for convenience only and refer to \cite{VHS90} for the RKC method in nonautonomous form.  The three term recurrence relation allow for low memory requirements even for very large $s$ and good internal stability properties \cite{HoS80}.

\subsection{The multirate RKC method}\label{sec:fdrkcsq}
The multirate RKC scheme is obtained by discretizing \eqref{eq:odemod} with an $s$-stage RKC method, where $\fe$, given by \cref{def:fe}, is approximated by solving problem \eqref{eq:defu} with one step of an $m$-stage RKC method. In this section, we first define the mRKC algorithm and then compare its efficiency to that of the standard RKC method \eqref{eq:rkc}. 

\subsubsection*{The mRKC Algorithm}
Let $\tau>0$ be the step size and $\rhof,\rhos$ the spectral radii of the Jacobians of $\ff,\fs$, respectively (they can be cheaply estimated employing nonlinear power methods \cite{Lin72,Ver80}). 
Now, let the number of stages $s,m$ be the smallest integers satisfying
\begin{align}\label{eq:defsmeta}
\tau\rhos &\leq \beta s^2, & \eta\rhof &\leq  \beta m^2, &\text{with}  && \eta &= \frac{6\tau}{\beta s^2} \frac{m^2}{m^2-1},
\end{align}
for the standard RKC parameter settings $\beta=2-4\varepsilon/3$ and $\varepsilon=0.05$ -- see Section \ref{sec:rkc}. The value for $\eta$ will be clear from the stability analysis in \cref{sec:stabanalysis}.

One step of the mRKC scheme is then given by
\begin{align}\label{eq:defmrkc}
\begin{split}
k_0&=y_n,\\
k_1 &= k_0+\mu_1\tau\bfe(k_0),\\
k_j&= \nu_j k_{j-1}+\kappa_j k_{j-2}+\mu_j\tau \bfe(k_{j-1}) \quad j=2,\ldots,s,\\
y_{n+1} &=k_s,
\end{split}
\end{align}
where the parameters $\mu_j,\nu_j,\kappa_j$ are defined in \eqref{eq:defcoeff} and 
\begin{equation}\label{eq:defbfe}
\bfe(y)=\frac{1}{\eta}(u_\eta-y)
\end{equation}
corresponds to the numerical counterpart of $\fe(y)$ in \eqref{eq:deffedif}. 
The approximation $u_\eta$ of $u(\eta)$ is computed at each evaluation of $\bfe$ by applying one step, of size $\eta$, of the $m$-stage RKC scheme to \eqref{eq:defu}. Hence, $u_\eta$ is given by
\begin{align}\label{eq:defbue}
\begin{split}
u_0 &= y,\\
u_1 &= u_0+\alpha_1\eta(\ff(u_0)+\fs(y)),\\
u_j &= \beta_j u_{j-1}+\gamma_j u_{j-2}+\alpha_j\eta (\ff(u_{j-1})+\fs(y)) \quad j=2,\ldots,m,\\
u_\eta &=u_m.
\end{split}
\end{align}
Here, the parameters $\alpha_j,\beta_j,\gamma_j$ of the $m$-stage RKC scheme \eqref{eq:defbue} are given by
\begin{align}\label{eq:defv01}
\vz&=1+\varepsilon/m^2, & \vu&= T_m(\vz)/T_m'(\vz), & a_j&= 1/T_j(\vz) & \mbox{ for }j=0,\ldots,m&
\end{align}
and 
$\alpha_1 = \vu/\vz$,
\begin{align}\label{eq:defabg} 
		\alpha_j&= 2\vu  a_j/a_{j-1}, & 
		\beta_j&= 2\vz a_j/a_{j-1},    & 
		\gamma_j&=-a_j/a_{j-2}  &
		\text{for }j&=2,\ldots,m.
\end{align}
To compute $m,\eta$ in \eqref{eq:defsmeta}, we let $\eta=6\tau m^2/(\beta s^2(m^2-1))$ in $\eta\rhof\leq \beta m^2$, which implies
\begin{equation}\label{eq:compm}
6\tau\rhof\leq\beta^2s^2(m^2-1).
\end{equation}
Thus, we use \eqref{eq:compm} to compute $m$ and then \eqref{eq:defsmeta} to determine $\eta$.

The mRKC method is given by \eqref{eq:defsmeta}--\eqref{eq:defbue}.
Its stability and first-order accuracy are proved in \cref{thm:stab_mrkc,thm:firstorder} in \cref{sec:stab_conv} below.

\subsubsection*{Efficiency of the multirate RKC method}
Given the spectral radii $\rhof$ and $\rhos$ of the Jacobians of $\ff$ and $\fs$, respectively, we now evaluate the theoretical speed-up in using the mRKC method \cref{eq:defsmeta,eq:defmrkc,eq:defbfe,eq:defbue} over the standard RKC method \eqref{eq:rkc}. In doing so, we set $\varepsilon=0$ and let $s,m$ vary in $\Rb$. 
Now, we let $\costf$ and $\costs$ denote the cost of evaluating $\ff$ and $\fs$, relatively to the cost of evaluating $f$ itself, with $\costf,\costs\in [0,1]$ and $\costf+\costs=1$.
Here, we suppose that the spectral radius $\rho$ of the Jacobian of $f$ is $\rho=\rhof+\rhos$, instead
of setting  $\rho=\rhof$, to allow for a wide range of possible values for $\rhof$ even down to zero.

Since the RKC scheme requires $s=\sqrt{\tau\rho/2}$ evaluations of $f$ per time step, its cost per time step is
\begin{equation}\label{eq:costrkc}
C_{\RKCop}=s(\costf+\costs)=\sqrt{\frac{\tau(\rhof+\rhos)}{2}}.
\end{equation}

For the mRKC method, on the other hand, we infer
from \eqref{eq:defsmeta} with $\beta=2$ that it needs $s=\sqrt{\tau\rhos/2}$ external stages and from \eqref{eq:compm} that it needs $m=\sqrt{3\rhof/\rhos+1}$ internal stages. Since mRKC needs $s$ evaluations of $\fs$ and $s\,m$ evaluations of $\ff$, its cost per time step is
\begin{equation}\label{eq:costmrkc}
C_{\mRKCop}=s\,\costs+s\,m\,\costf = (1-\costf)\sqrt{\frac{\tau\rhos}{2}}+\costf\sqrt{\frac{3\tau\rhof}{2}+\frac{\tau \rhos}{2}}.
\end{equation}
The ratio between \cref{eq:costrkc,eq:costmrkc} yields the \textit{relative speed-up}
\begin{equation}\label{eq:speedup}
S=\frac{C_{\RKCop}}{C_{\mRKCop}}= \frac{\sqrt{\rhof+\rhos}}{(1-\costf)\sqrt{\rhos}+\costf\sqrt{\rhos+3\rhof}}
=\frac{\sqrt{1+r_\rho}}{1+\costf\left(\sqrt{1+3r_\rho}-1\right)},
\end{equation}
with {\it stiffness ratio} $r_\rho=\rhof/\rhos\in [0,\infty)$.

In \cref{fig:speed_up_rhor}, we show the speed-up $S$ as a function of $\costf$ for different values of $r_\rho=\rhof/\rhos$. For $\costf$ sufficiently small, we observe that the mRKC scheme is always faster than RKC ($S>1$). When $\costf\approx 1$, however, the mRKC scheme is slightly slower than RKC ($S<1$), though this case is somewhat irrelevant since by assumption $\ff$ is cheap to evaluate. Nevertheless, we solve the inequality $S>1$, with $S$ as in \eqref{eq:speedup}, for varying $\costf$ to determine the maximal value of $\costf$ that still leads to a reduced cost in using mRKC. We find that the speed-up $S>1$ if, and only if,
\begin{equation}
\costf < \costf^{\max}=\frac{\sqrt{1+r_\rho}-1}{\sqrt{1+3 r_\rho}-1}.
\end{equation}
In \cref{fig:rmax}, we monitor $\costf^{\max}$ as a function of the stiffness ratio $r_\rho$. For small $r_\rho=\rhof/\rhos$, we observe that the evaluation of $\ff$ must be quite cheap. As $\rhof/\rhos$ increases, however, the mRKC method is faster than RKC, even if $\ff$ is relatively expensive to evaluate ($\costf^{\max}>0.5$ for $\rhof/\rhos>8$).

\begin{figure}
		\begin{subfigure}[t]{\subfigsize\textwidth}
			\centering
			\begin{tikzpicture}[scale=\plotimscale]
			\begin{axis}[height=\aspectratio*\plotimsized\textwidth,width=\plotimsized\textwidth,legend columns=1,legend style={draw=\legendboxdraw,fill=\legendboxfill,at={(1,1)},anchor=north east},legend cell align={left}, ymin=-0.2, ymax=9,xticklabel style={yshift=-2pt},
			xlabel={$\phantom{/}\costf\phantom{/}$}, ylabel={$S$},label style={font=\normalsize},tick label style={font=\normalsize},legend image post style={scale=\legendmarkscale},legend style={nodes={scale=\legendfontscale, transform shape}},grid=none]
			\addplot[color=\colorone,line width=\plotlinewidth pt,mark=\markone,mark size=\plotmarksizeu pt,mark repeat=20,mark phase=0] table [x=c,y=S4,col sep=comma] 
			{data/text/S_SbarVScF.csv};\addlegendentry{$\rhof/\rhos=4$}
			\addplot[color=\colortwo,line width=\plotlinewidth pt,mark=\marktwo,mark size=\plotmarksizeu pt,mark repeat=20,mark phase=0] table [x=c,y=S16,col sep=comma] 
			{data/text/S_SbarVScF.csv};\addlegendentry{$\rhof/\rhos=16$}
			\addplot[color=\colorthree,line width=\plotlinewidth pt,mark=\markthree,mark size=\plotmarksizeu pt,mark repeat=20,mark phase=0] table [x=c,y=S64,col sep=comma] 
			{data/text/S_SbarVScF.csv};\addlegendentry{$\rhof/\rhos=64$}
			\addplot[black,dashed,line width=\plotdashedlinewidth pt,domain=0:1] (x,1);\addlegendentry{$1$}
			\end{axis}
			\end{tikzpicture}
			\caption{Theoretical speed-up $S$ of the mRKC method over the standard RKC scheme, with respect to $\costf$ and $\rhof/\rhos$.}
			\label{fig:speed_up_rhor}
		\end{subfigure}\hfill%
		\begin{subfigure}[t]{\subfigsize\textwidth}
			\centering
			\begin{tikzpicture}[scale=\plotimscale]
			\begin{semilogxaxis}[height=\aspectratio*\plotimsized\textwidth,width=\plotimsized\textwidth,legend columns=1,legend style={draw=\legendboxdraw,fill=\legendboxfill,at={(1,1)},anchor=north east},ymin=0.3,ymax=0.6,
			xlabel={$\rhof/\rhos$}, ylabel={$\costf^{\max}$},label style={font=\normalsize},tick label style={font=\normalsize},legend image post style={scale=\legendmarkscale},legend style={nodes={scale=\legendfontscale, transform shape}},grid=none]
			\addplot[color=\colorone,line width=\plotlinewidth pt,mark=none] table [x=r_logv,y=cFmax10logr,col sep=comma] 
			{data/text/cFmax.csv};
			\end{semilogxaxis}
			\end{tikzpicture}
			\caption{Maximal $\costf$ which still yields speed-up $S>1$, w.r.t. $\rhof/\rhos$.}
			\label{fig:rmax}
		\end{subfigure}%
	\caption{The relative speed-up $S$ of the mRKC method over the RKC scheme with respect to $\costf$ and the maximal value for $\costf$ which still leads to an efficiency gain.}
	\label{fig:spedd_up_plots}
\end{figure}
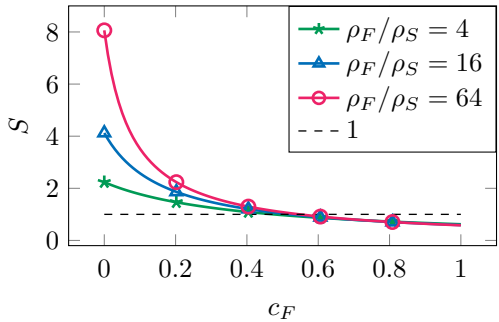
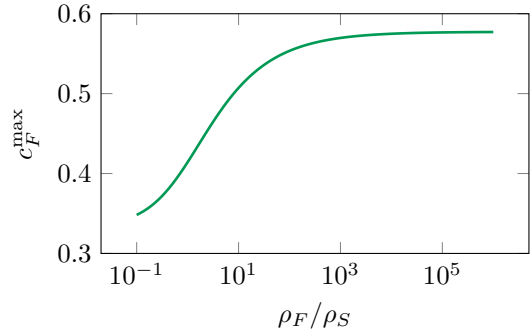

\subsubsection*{Relaxed stability conditions}
The stability conditions \eqref{eq:defsmeta} are necessary when solving a general problem \eqref{eq:ode} without any scale separation. However, in case of scale separation ($\lambda\ll\zeta$), conditions \eqref{eq:defsmeta} can in fact be replaced by
\begin{align}\label{eq:defsmetaweak}
\tau\rhos &\leq\beta s^2, &\eta\rhof&\leq \overline \beta m^2 &\text{with}  && \eta &= \frac{2\tau}{\beta s^2},
\end{align}
$\overline\beta=2-4\overline\varepsilon/3\approx 1.86$ and $\overline\varepsilon=0.1$ -- see \cref{rem:scalesep} for further insight on the derivation of \cref{eq:defsmetaweak}. 
Since the value for $\eta$ in \eqref{eq:defsmetaweak} is smaller than that in \eqref{eq:defsmeta}, $m$ can also be smaller which results in fewer evaluations of $\ff$ in \eqref{eq:defbue} and improved efficiency.
Let $\overline S$ be the relative speed-up in using \eqref{eq:defsmetaweak} instead of \eqref{eq:defsmeta}. In \cref{fig:speedupweak},  we plot $\overline S$ as a function of $\costf$ for different values of $\rhof/\rhos$, as in \cref{fig:speed_up_rhor} for $S$. We observe that $\overline S>1$ for all $\costf\in [0,1-\epsilon]$, for $\epsilon>0$ very small. In \cref{fig:speed_up_comp}, we compare $S$ and $\overline S$ and observe that $\overline S>S$ for all values of $\costf$.

Even when the underlying problem is not scale separable, conditions 
\eqref{eq:defsmetaweak} may in fact be sufficient to guarantee the stability of the mRKC scheme. 
For instance, if \eqref{eq:ode} stems from the spatial discretization of a parabolic problem on a locally refined mesh, where $\fs$ and $\ff$ correspond to the discrete Laplacians in the coarse and locally refined region, and the problem thus is not scale separable, 
\eqref{eq:defsmetaweak} nonetheless suffices to guarantee stability --- see \cref{sec:exp_stab}.
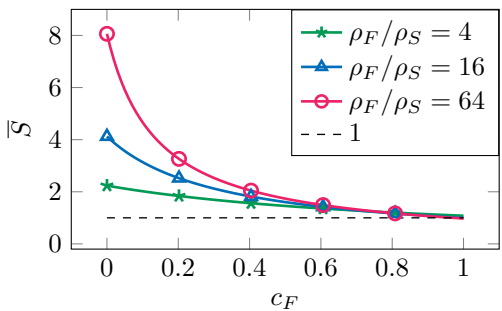
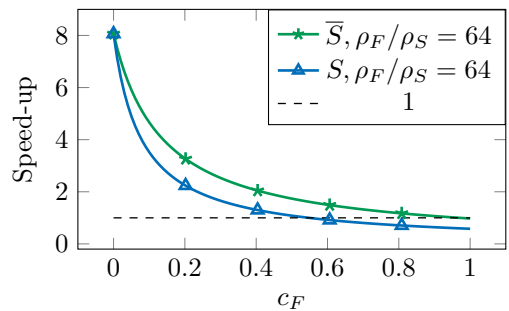
\begin{figure}
		\begin{subfigure}[t]{\subfigsize\textwidth}
			\centering
			\begin{tikzpicture}[scale=\plotimscale]
			\begin{axis}[height=\aspectratio*\plotimsized\textwidth,width=\plotimsized\textwidth,legend columns=1,legend style={draw=\legendboxdraw,fill=\legendboxfill,at={(1,1)},anchor=north east},legend cell align={left}, ymin=-0.2, ymax=9,
			xlabel={$\costf$}, ylabel={$\overline S$},label style={font=\normalsize},tick label style={font=\normalsize},legend image post style={scale=\legendmarkscale},legend style={nodes={scale=\legendfontscale, transform shape}},grid=none]
			\addplot[color=\colorone,line width=\plotlinewidth pt,mark=\markone,mark size=\plotmarksizeu pt,mark repeat=20,mark phase=0] table [x=c,y=Sbar4,col sep=comma] 
			{data/text/S_SbarVScF.csv};\addlegendentry{$\rhof/\rhos=4$}
			\addplot[color=\colortwo,line width=\plotlinewidth pt,mark=\marktwo,mark size=\plotmarksizeu pt,mark repeat=20,mark phase=0] table [x=c,y=Sbar16,col sep=comma] 
			{data/text/S_SbarVScF.csv};\addlegendentry{$\rhof/\rhos=16$}
			\addplot[color=\colorthree,line width=\plotlinewidth pt,mark=\markthree,mark size=\plotmarksizeu pt,mark repeat=20,mark phase=0] table [x=c,y=Sbar64,col sep=comma] 
			{data/text/S_SbarVScF.csv};\addlegendentry{$\rhof/\rhos=64$}
			\addplot[black,dashed,line width=\plotdashedlinewidth pt,domain=0:1] (x,1);\addlegendentry{$1$}
			\end{axis}
			\end{tikzpicture}
			\caption{Theoretical speed-up $\overline S$ of the mRKC method over the standard RKC scheme, with respect to $\costf$ and $\rhof/\rhos$.}
			\label{fig:speedupweak}
		\end{subfigure}\hfill%
		\begin{subfigure}[t]{\subfigsize\textwidth}
			\centering
			\begin{tikzpicture}[scale=\plotimscale]
			\begin{axis}[height=\aspectratio*\plotimsized\textwidth,width=\plotimsized\textwidth,legend columns=1,legend style={draw=\legendboxdraw,fill=\legendboxfill,at={(1,1)},anchor=north east},ymin=-0.2, ymax=9,
			xlabel={$\costf$}, ylabel={Speed-up},label style={font=\normalsize},tick label style={font=\normalsize},legend image post style={scale=\legendmarkscale},legend style={nodes={scale=\legendfontscale, transform shape}},grid=none]
			\addplot[color=\colorone,line width=\plotlinewidth pt,mark=\markone,mark size=\plotmarksizeu pt,mark repeat=20,mark phase=0] table [x=c,y=Sbar64,col sep=comma] 
			{data/text/S_SbarVScF.csv};\addlegendentry{$\overline S, \rhof/\rhos=64$}
			\addplot[color=\colortwo,line width=\plotlinewidth pt,mark=\marktwo,mark size=\plotmarksizeu pt,mark repeat=20,mark phase=0] table [x=c,y=S64,col sep=comma] 
			{data/text/S_SbarVScF.csv};\addlegendentry{$S, \rhof/\rhos=64$}
			\addplot[black,dashed,line width=\plotdashedlinewidth pt,domain=0:1] (x,1);\addlegendentry{$1$}
			\end{axis}
			\end{tikzpicture}
			\caption{Comparison of $\overline{S}$ and $S$.}
			\label{fig:speed_up_comp}
		\end{subfigure}%
	\caption{The relative speed-up $\overline S$ obtained using \cref{eq:defsmetaweak} compared to $S$, obtained with \cref{eq:defsmeta}.}
\end{figure}

\section{Stability and convergence analysis}\label{sec:stab_conv}
{In this section, we perform a stability and convergence analysis of the multirate RKC method introduced in \cref{sec:mrkca}. We will show stability of the scheme on the multirate test equation \eqref{eq:mtesteq} and on a $2\times 2$ model problem. Then we prove its first-order accuracy.}

\subsection{Stability analysis}\label{sec:stabanalysis}
{First, we prove that the mRKC method is stable when it is applied to the multirate test equation \eqref{eq:mtesteq}, which is sufficient when the Jacobians of $\ff$ and $\fs$ are simultaneously triangularizable. Then, we also show stability for a $2\times 2$ model problem where the Jacobians of $\ff$ and $\fs$ are not simultaneously triangularizable, and hence the stability analysis cannot be reduced to \eqref{eq:mtesteq}.}

\subsubsection*{Stability analysis for the multirate test equation}
Since \eqref{eq:defu} is approximated numerically, the stability analysis performed in \cref{sec:stab_modeq} is no longer valid; indeed, $\varphi(z)$ is now replaced by a numerical approximation with different stability properties. Hence, we now compute a closed expression for $u_\eta$ given $y$, as in \eqref{eq:solutesteq} for $u(\eta)$. 
We denote by 
\begin{equation}\label{eq:defPm}
P_m(z)=a_m T_m(\vz+\vu z)
\end{equation}
the stability polynomial of the $m$-stage RKC scheme, with $a_m,\vz,\vu$ from \cref{eq:defv01}.
The next lemma computes the solution $u_\eta$ of \cref{eq:defbue} in the case of the multirate test equation \cref{eq:mtesteq}.

\begin{lemma}\label{lemma:closedbue}
	Let $\lambda,\zeta\leq 0$, $\fs(y)=\zeta y$, $\ff(y)=\lambda y$, $\eta>0$, $m\in\Nb$ and $y\in\Rb$. Then, the solution $u_\eta$ of \cref{eq:defbue}, is given by
	\begin{equation}\label{eq:numutesteq}
	u_\eta = (P_m(\eta \lambda)+ \Ps_m(\eta\lambda)\eta\zeta)y,
	\end{equation}
	where $P_m(z)$ is given in \cref{eq:defPm},
	\begin{equation}\label{eq:PhimTfrac}
	\Ps_m(z) = \frac{P_m(z)-1}{z} \quad \text{for } z\neq 0
	\end{equation}
	and $\Ps_m(0)=1$ is defined by continuous extension.
\end{lemma}
\begin{proof}
A generalization of \cite[Proposition 3.1]{HaW02} to the equation $u'=\lambda u+\zeta y$ (instead of $u'=\lambda u$) yields
\begin{equation}
	u_\eta = P_m(z)y+\eta b^\top (I-zA)^{-1}\mathds{1}\zeta y,
\end{equation}
where $z=\eta\lambda$, $I\in \Rb^{m\times m}$ is the identity matrix, $A,b$ are the coefficients of the Butcher tableau of the $m$-stage RKC scheme and $\mathds{1}\in \Rb^m$ is a vector of ones. Since $P_m(z)=1+z b^\top (I-zA)^{-1} \mathds{1}$ \cite[Proposition 3.1]{HaW02} the result follows.
\end{proof}
%
Note the similarity between \cref{eq:solutesteq} and \cref{eq:numutesteq}, with $e^z,\varphi(z)$ replaced by $P_m(z),\Ps_m(z)$, respectively. 
In \cref{fig:Phim}, we also observe that $\Ps_m(z)$ and $\varphi(z)$ share similar stability properties. Indeed, $\Ps_m(z)$ is the numerical counterpart of $\varphi(z)$,
yet with the exponential replaced by the stability polynomial -- compare \cref{eq:defphi,eq:PhimTfrac}.
\begin{figure}
		\begin{subfigure}[t]{\subfigsize\textwidth}
			\centering
			\begin{tikzpicture}[scale=\plotimscale]
			\begin{axis}[height=\aspectratio*\plotimsized\textwidth,width=\plotimsized\textwidth, ymin=-0.2, ymax=1,xmax=0,xmin=-136,legend columns=1,legend style={draw=\legendboxdraw,fill=\legendboxfill,at={(0,1)},anchor=north west},legend cell align={left},
			xlabel={$z$}, ylabel={},label style={font=\normalsize},tick label style={font=\normalsize},legend image post style={scale=\legendmarkscale},legend style={nodes={scale=\legendfontscale, transform shape}},grid=none]
			\addplot[color=\colorone,line width=\plotlinewidth pt,mark=\markone,mark repeat=32,mark phase=20,mark size=\plotmarksizeu pt] table [x=z,y=phi,col sep=comma] 
			{data/text/Phi_8.csv};\addlegendentry{$\varphi(z)=(e^z-1)/z$}
			\addplot[color=\colortwo,line width=\plotlinewidth pt,mark=\marktwo,mark repeat=32,mark phase=4,mark size=\plotmarksizeu pt] table [x=z,y=Phim,col sep=comma] 
			{data/text/Phi_8.csv};\addlegendentry{$\Ps_8(z)=(P_8(z)-1)/z$}
			\addplot[black,dashed,line width=\plotdashedlinewidth pt,domain=-0.2:1] (-123.91,x);
			\end{axis}
			\end{tikzpicture}
		\end{subfigure}\hfill%
		\begin{subfigure}[t]{\subfigsize\textwidth}
			\centering
			\begin{tikzpicture}[scale=\plotimscale]
			\begin{axis}[height=\aspectratio*\plotimsized\textwidth,width=\plotimsized\textwidth, ymin=-0.2, ymax=1,xmax=0,xmin=-172.26,legend columns=1,legend style={draw=\legendboxdraw,fill=\legendboxfill,at={(0,1)},anchor=north west},legend cell align={left},
			xlabel={$z$}, ylabel={},label style={font=\normalsize},tick label style={font=\normalsize},legend image post style={scale=\legendmarkscale},legend style={nodes={scale=\legendfontscale, transform shape}},grid=none]
			\addplot[color=\colorone,line width=\plotlinewidth pt,mark=\markone,mark repeat=32,mark phase=20,mark size=\plotmarksizeu pt] table [x=z,y=phi,col sep=comma] 
			{data/text/Phi_9.csv};\addlegendentry{$\varphi(z)=(e^z-1)/z$}
			\addplot[color=\colortwo,line width=\plotlinewidth pt,mark=\marktwo,mark repeat=32,mark phase=4,mark size=\plotmarksizeu pt] table [x=z,y=Phim,col sep=comma] 
			{data/text/Phi_9.csv};\addlegendentry{$\Ps_{9}(z)=(P_{9}(z)-1)/z$}
			\addplot[black,dashed,line width=\plotdashedlinewidth pt,domain=-0.2:1] (-156.82,x);
			\end{axis}
			\end{tikzpicture}
		\end{subfigure}%
	\caption{Illustration of $\varphi(z)$ and $\Ps_m(z)$ for $m=8$ (left) and $m=9$ (right). The dashed line indicates the end of the stability domain.}
	\label{fig:Phim}
\end{figure}
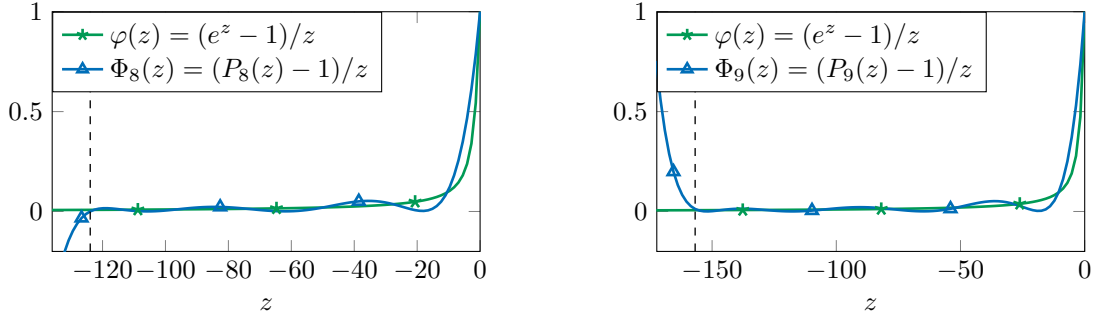

We can now compute the stability polynomial of the mRKC scheme. From \cref{eq:defbfe,eq:numutesteq,eq:PhimTfrac}, we get
\begin{equation}\label{eq:bfetesteq}
\bfe(y)= \frac{1}{\eta}(\Pi_m(\eta\lambda)+\Ps_m(\eta\lambda)\eta\zeta-1)y  =\Ps_m(\eta\lambda)(\lambda+\zeta)y,
\end{equation}
which is the numerical counterpart of $\fe$ in \cref{eq:fetesteq}. Now, we insert \cref{eq:bfetesteq} into \cref{eq:defmrkc}, which leads to 
\begin{equation}
y_{n+1} = \Pe_s(\tau \Ps_m(\eta\lambda)(\lambda+\zeta))y_n,
\end{equation} 
with $R_s(z)$ the stability polynomial of the $s$-stage RKC scheme defined in \cref{sec:rkc}, and hence motivates the following definition.
\begin{definition}
	Let $s,m\in\Nb$, $\tau>0$ be a step size, $\eta>0$ and $\lambda,\zeta\leq 0$. The stability polynomial of the $(s,m)$-stage mRKC scheme \eqref{eq:defmrkc}--\eqref{eq:defbue} is defined as
	\begin{equation}\label{eq:defstabpol}
	R_{s,m}(\lambda,\zeta,\tau,\eta)=R_s(\tau \Ps_m(\eta\lambda)(\lambda+\zeta)),
	\end{equation}
with $R_s(z)$ as in \cref{eq:stabpolrkc} and $\Phi_m(z)$ as in \cref{eq:PhimTfrac}.
\end{definition}
The following lemma is the discrete version of \cref{lemma:cont_stab} and is needed to prove stability of the mRKC scheme in \cref{thm:stab_mrkc} below. Its proof is purely technical and postponed to \cref{app:proofs}.

\begin{lemma}\label{lemma:disc_stab}
	Let $m\in\Nb$ and $w\leq 0$. There exists $\overline\varepsilon_m>0$ such that for $\varepsilon\leq\overline\varepsilon_m$, $\Ps_m(z)(z+w)\in [w,0]$ for all $z\in [-\ell_m^\varepsilon,0]$ if, and only if, $\Ps_m'(0)|w|\geq 1$, i.e.  $|w|\geq 2/P_m''(0)$ since $\Ps_m'(0)=P_m''(0)/2$. 
\end{lemma}
For $\varepsilon=0$, it holds $2/\Pi_m''(0)=6 m^2/(m^2-1)>6$. In the continuous setting, the condition on $w$ in \cref{lemma:cont_stab} was $|w|\geq 2$. For the discrete mRKC scheme, however, $|w|>6$ is necessary because of the milder slope of $\Ps_m(z)$ at the origin, see \cref{fig:Phim}.
\begin{remark}\label{rem:scalesep}
	In the case of scale separation, $z$ is bounded away from the origin and the value of $\Phi_m(z)$ thus considerably smaller than $1$, see \cref{fig:Phim}. Hence, the condition $\Ps_m(z)(z+w)\in [w,0]$ is already satisfied for $|w|\geq 2$ and a slightly larger damping $\overline\varepsilon=0.1$, so that the larger value $|w|\geq 2/P_m''(0)\approx 6$, required by \cref{lemma:disc_stab}, is no longer necessary. By allowing $|w|\geq 2$ instead of $|w|\geq 2/P_m''(0)$ in \cref{thm:stab_mrkc} below, we find that the weaker stability conditions \cref{eq:defsmetaweak} already guarantee stability in the case of scale separation -- see \cite[Section 3.4.5]{Ros20} for further details.
\end{remark}
\begin{theorem}\label{thm:stab_mrkc}
Let $\overline\varepsilon_m$ be as in \cref{lemma:disc_stab} and, for $\varepsilon\geq 0$, let $\varepsilon_m=\min\{\varepsilon,\overline{\varepsilon}_m\}$. Let $\lambda\leq 0$ and $\zeta<0$. Then, for all $\tau>0,s,m$ and $\eta$ such that
\begin{align}\label{eq:defsmetascalar}
\tau |\zeta| &\leq\ell_s^\varepsilon, & \eta|\lambda| &\leq \ell^{\varepsilon_m}_m &\text{with}  && \eta \geq \frac{6\tau}{\ell_s^\varepsilon}\frac{m^2}{m^2-1},
\end{align}
$|\Pe_{s,m}(\lambda,\zeta,\tau,\eta)|\leq 1$, i.e. the mRKC scheme is stable.
\end{theorem}
\begin{proof}
If $\tau \Ps_m(\eta\lambda)(\lambda+\zeta)\in [-\ell_s^\varepsilon,0]$ then $|\Pe_{s,m}(\lambda,\zeta,\tau,\eta)|=|\Pe_s(\tau \Ps_m(\eta\lambda)(\lambda+\zeta))|\leq 1$. Hence, it is sufficient to prove the equivalent condition:
\begin{align}
\Phi_m(\eta\lambda)(\eta\lambda+\eta\zeta)&\in [w(\eta),0], & \mbox{with} && w(\eta)&=-\frac{\eta}{\tau}\ell_s^\varepsilon.
\end{align}
Since $\eta\lambda\in [-\ell_m^{\varepsilon_m},0]$, it holds $|\Pi_m(\eta\lambda)|\leq 1$ and from \eqref{eq:PhimTfrac} we thus deduce that $\Ps_m(\eta\lambda)\geq 0$. Furthermore, \eqref{eq:defsmetascalar} yields $\eta\zeta\geq w(\eta)$ which implies
\begin{equation}
0\geq \Phi_m(\eta\lambda)(\eta\lambda+\eta\zeta)\geq  \Phi_m(z(\eta))(z(\eta)+w(\eta)),
\end{equation}
with $z(\eta)=\eta\lambda$. Hence, it is sufficient to show that $\Phi_m(z(\eta))(z(\eta)+w(\eta))\in [w(\eta),0]$ for all $z(\eta)\in [-\ell^{\varepsilon_m}_m,0]$. From \cref{lemma:disc_stab}, we know that
\begin{equation}
|w(\eta)| \geq  \frac{2}{\Pi_m''(0)}
\end{equation}
is necessary and sufficient. Since $T_m'(v_0)^2/(T_m(v_0)T_m''(v_0))$ is decreasing for $v_0\geq 1$ (see \cref{lemma:ddRincr}), we infer from the definition of $\eta$ in \eqref{eq:defsmetascalar} that
\begin{equation}
|w(\eta)|\geq 6\frac{m^2}{m^2-1}=\frac{2\, T_m'(1)^2}{T_m(1)T_m''(1)}\geq \frac{2\, T_m'(v_0)^2}{T_m(v_0)T_m''(v_0)}=\frac{2}{\Pi_m''(0)}. \tag*{\qedhere}
\end{equation}
\end{proof}

In the continuous setting in \cref{sec:stab_modeq}, $\eta$ directly depends on $\fs$; indeed, the 
condition $|\varphi(\eta\lambda)(\lambda+\zeta)|\leq |\zeta|$ implies $\eta\geq 2/|\zeta|$ (see \cref{thm:stifflinscamod}). Therefore, $\eta$ could rapidly grow as $\zeta\to 0$. In contrast, for the mRKC method, $\eta$ depends only indirectly on $\fs$: $\eta$ depends on the $s$-stage RKC method, which in turn depends on $\fs$. This indirect dependence of $\eta$ on $\fs$ creates a ``protective buffer'', which prevents the explosion of $\eta$ as $\zeta\to 0$; indeed, $\ell_s^\varepsilon\geq 2$ for all $s\in\Nb$.

The restriction $\varepsilon\leq \overline\varepsilon_m$ is necessary for proving \cref{lemma:disc_stab}, but probably not needed in practice. Indeed, we have verified numerically that for any $\varepsilon \geq 0$, $\Ps_m(z)(z+w)\in [w,0]$ for all $z\in [-\ell_m^\varepsilon,0]$ if, and only if, $|w|\geq 2/\Pi_m''(0)$. Hence, we can suppose $\varepsilon_m=\varepsilon$ in \eqref{eq:defsmetascalar} and replace $\ell_s^\varepsilon,\ell_m^{\varepsilon_m}$ by $\beta s^2,\beta m^2$, respectively, which yields \eqref{eq:defsmeta}.
In \cref{fig:Rsm}, we display the stability polynomial $R_{s,m}(\lambda,\zeta,\tau,\eta)$ for $s=5$ and $m=3$ as a function of $\lambda$ for $\varepsilon=0.05$ or $\varepsilon=1$. Here, we set $\tau=1$, $\eta$ to its lower bound in \eqref{eq:defsmetascalar}, and $\zeta=-\ell_s^\varepsilon,-\ell_s^\varepsilon/2$ or $0$. Since $|R_{s,m}(\lambda,\zeta,\tau,\eta)|\leq 1$, the mRKC method is always stable.
\begin{figure}
		\begin{subfigure}[t]{\subfigsize\textwidth}
			\centering
			\begin{tikzpicture}[scale=\plotimscale]
			\begin{axis}[height=\aspectratio*\plotimsized\textwidth,width=\plotimsized\textwidth, ymin=-1.0, ymax=1.0,xmax=0,xmin=-125.08,legend columns=3,legend style={draw=\legendboxdraw,fill=\legendboxfill,at={(0.5,1.05)},anchor=south},legend cell align={left},
			xlabel={$\lambda$}, ylabel={},label style={font=\normalsize},tick label style={font=\normalsize},legend image post style={scale=\legendmarkscale},legend style={nodes={scale=\legendfontscale, transform shape}},grid=none]
			\addplot[color=\colorone,line width=\plotlinewidth pt,mark=\markone,mark repeat=24,mark phase=0,mark size=\plotmarksizeu pt] table [x=lambdas,y=Rsm1,col sep=comma] 
			{data/text/R_5_3_damp_0p05.csv};\addlegendentry{$\zeta=-\ell^\varepsilon_s$}
			\addplot[color=\colortwo,line width=\plotlinewidth pt,mark=\marktwo,mark repeat=24,mark phase=8,mark size=\plotmarksizeu pt] table [x=lambdas,y=Rsm2,col sep=comma] 
			{data/text/R_5_3_damp_0p05.csv};\addlegendentry{$\zeta=-\ell^\varepsilon_s/2$}
			\addplot[color=\colorthree,line width=\plotlinewidth pt,mark=\markthree,mark repeat=24,mark phase=16,mark size=\plotmarksizeu pt] table [x=lambdas,y=Rsm3,col sep=comma] 
			{data/text/R_5_3_damp_0p05.csv};\addlegendentry{$\zeta=0$}
			\end{axis}
			\end{tikzpicture}
			\caption{Small damping $\varepsilon=0.05$.}
			\label{fig:Rsma}
		\end{subfigure}\hfill%
		\begin{subfigure}[t]{\subfigsize\textwidth}
			\centering
			\begin{tikzpicture}[scale=\plotimscale]
			\begin{axis}[height=\aspectratio*\plotimsized\textwidth,width=\plotimsized\textwidth, ymin=-1.0, ymax=1.0,xmax=0,xmin=-54.36,legend columns=3,legend style={draw=\legendboxdraw,fill=\legendboxfill,at={(0.5,1.05)},anchor=south},legend cell align={left},
			xlabel={$\lambda$}, ylabel={},label style={font=\normalsize},tick label style={font=\normalsize},legend image post style={scale=\legendmarkscale},legend style={nodes={scale=\legendfontscale, transform shape}},grid=none]
			\addplot[color=\colorone,line width=\plotlinewidth pt,mark=\markone,mark repeat=24,mark phase=0,mark size=\plotmarksizeu pt] table [x=lambdas,y=Rsm1,col sep=comma] 
			{data/text/R_5_3_damp_1.csv};\addlegendentry{$\zeta=-\ell^\varepsilon_s$}
			\addplot[color=\colortwo,line width=\plotlinewidth pt,mark=\marktwo,mark repeat=24,mark phase=8,mark size=\plotmarksizeu pt] table [x=lambdas,y=Rsm2,col sep=comma] 
			{data/text/R_5_3_damp_1.csv};\addlegendentry{$\zeta=-\ell^\varepsilon_s/2$}
			\addplot[color=\colorthree,line width=\plotlinewidth pt,mark=\markthree,mark repeat=24,mark phase=16,mark size=\plotmarksizeu pt] table [x=lambdas,y=Rsm3,col sep=comma] 
			{data/text/R_5_3_damp_1.csv};\addlegendentry{$\zeta=0$}
			\end{axis}
			\end{tikzpicture}
			\caption{High damping $\varepsilon=1$.}
			\label{fig:Rsmb}
		\end{subfigure}%
	\caption{Stability polynomial $R_{s,m}(\lambda,\zeta,\tau,\eta)$ of the mRKC method vs. $\lambda$ for $\zeta=-\ell_s^\varepsilon,-\ell_s^\varepsilon/2$ or $0$ and $s=5$, $m=3$, $\tau=1$, $\eta$ as in \cref{eq:defsmetascalar} and damping $\varepsilon=0.05$ (left) or $\varepsilon=1$ (right).}
	\label{fig:Rsm}
\end{figure}
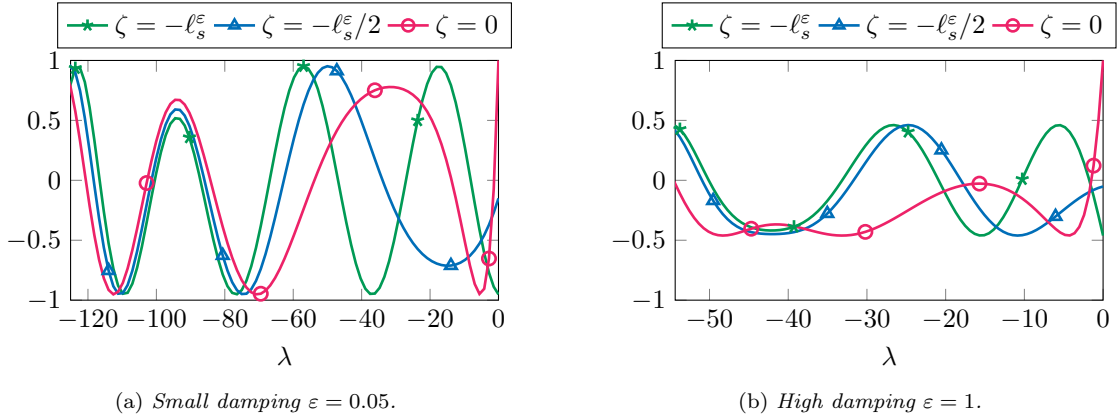

\subsubsection*{Stability analysis for a $2\times 2$ model problem}
Here, we consider a $2\times 2$ linear model problem where the Jacobians of $\ff$ and $\fs$ are not simultaneously triangularizable. Then, the stability analysis cannot be reduced to the scalar multirate test equation \eqref{eq:mtesteq}, yet we shall show that the same stability conditions still hold. Moreover, we introduce a coupling term between the fast and slow variables and show that the same stability conditions are necessary even when the coupling is weak.

Thus, we consider the system of differential equations
\begin{align}\label{eq:twodim}
y'&=Ay, &\mbox{with} && A&=\begin{pmatrix}
\zeta & \sigma\\ \sigma & \lambda
\end{pmatrix}
\end{align}
and $y(0)=y_0\in\Rb^2$. We let $\lambda,\zeta<0$, $\sigma\in\Rb$ the coupling term, and assume that $\sigma^2\leq\lambda\zeta$ to ensure that both eigenvalues of $A$ are negative or zero. We note $D\in\Rb^{2\times 2}$ the diagonal matrix satisfying $D_{11}=0$ and $D_{22}=1$ and consider the splitting defined by $\ff(y)=A_F\, y$ and $\fs(y)=A_S\, y$, where
\begin{align} \label{eq:tauAfAc1}
A_F &:=DA = \begin{pmatrix}
0 & 0 \\ \sigma & \lambda
\end{pmatrix},
& 
A_S &:= (I-D)A =\begin{pmatrix}
\zeta & \sigma \\ 0 & 0
\end{pmatrix}.
\end{align}
Observe that $\rhof=|\lambda|$ and $\rhos=|\zeta|$. The matrices $A_F,A_S$ are simultaneously triangularizable if, and only if, they have a common eigenvector, which occurs only for $\sigma=0$ or $\sigma^2=\lambda\zeta$. We set $\sigma=0.1\sqrt{\lambda\zeta}$, so that the present stability analysis cannot be reduced to the scalar multirate test equation \cref{eq:mtesteq}. Furthermore, as the eigenvalues of $A$ are negative or zero for all $|\sigma|\leq \sqrt{\lambda\zeta}$, the current coupling $\sigma=0.1\sqrt{\lambda\zeta}$ can be considered to be weak when compared to the maximal coupling $\sqrt{\lambda\zeta}$.

Given $y\in \Rb^2$, we obtain $\bfe(y)$ by replacing $\lambda,\zeta$ in \eqref{eq:bfetesteq} by $A_F,A_S$, respectively. This yields
\begin{align}\label{eq:defAeta}
\bfe(y) &= A_\eta y, &\mbox{with}&& A_\eta&=\Ps_m(\eta A_F)A y,
\end{align}
and since $\Ps_m$ is a polynomial, $A_\eta$ is well-defined. From \eqref{eq:defmrkc} it follows $y_{n+1} = \Pe_s(\tau A_\eta)y_n$. If the eigenvalues of $\tau A_\eta$ are in the interval $[-\ell_s^\varepsilon,0]$, the mRKC method is stable. For convenience, we set $\tau=1$, $|\zeta|=\ell_s^\varepsilon$ with $s=10$, and also fix $m=8$ and $\eta=\frac{6\tau}{\ell_s^\varepsilon}\frac{m^2}{m^2-1}$ (as in \eqref{eq:defsmetascalar}). Then, the mRKC method is stable if the spectral radius $\rho_\eta$ of $A_\eta$ satisfies $\rho_\eta\leq |\zeta|$ for all $\eta\lambda\in [-\ell_m^\varepsilon,0]$, or equivalently $\eta\rhoe\leq \eta|\zeta|= |w|$. 

In \cref{fig:stabweakcoupa}, we display $\eta\rhoe$ as a function of $z=\eta\lambda\in [-\ell_m^\varepsilon,0]$ and observe that $\eta\rhoe\leq |w|$; thus, the mRKC scheme is stable. Hence, the stability conditions \eqref{eq:defsmeta} guarantee stability of the scheme even though the Jacobians of $\ff,\fs$ are not simultaneously triangularizable.

Next, in \cref{fig:stabweakcoupb}, we consider a value of $\eta$ smaller than that dictated by \eqref{eq:defsmetascalar}. For $\overline\eta=0.9\eta$, we again display $\overline\eta\rho_{\overline\eta}$ as a function of $\overline z=\overline\eta\lambda\in [-\ell_m^\varepsilon,0]$. Then, a small region of instability appears for $\overline z$ close to zero, where $\overline\eta\rho_{\overline\eta}>|\overline w|$. Hence, the stability conditions \eqref{eq:defsmetascalar} are necessary even for systems of equations with a weak coupling $\sigma=0.1\sqrt{\lambda\zeta}$, where $\sqrt{\lambda\zeta}$ corresponds to the maximal coupling strength.
Similar instabilities as in \cref{fig:stabweakcoupb} occur for even weaker couplings
$\sigma=0.01\sqrt{\lambda\zeta}$, $\sigma=0.001\sqrt{\lambda\zeta}$ and for larger $\overline\eta=0.95\eta$.

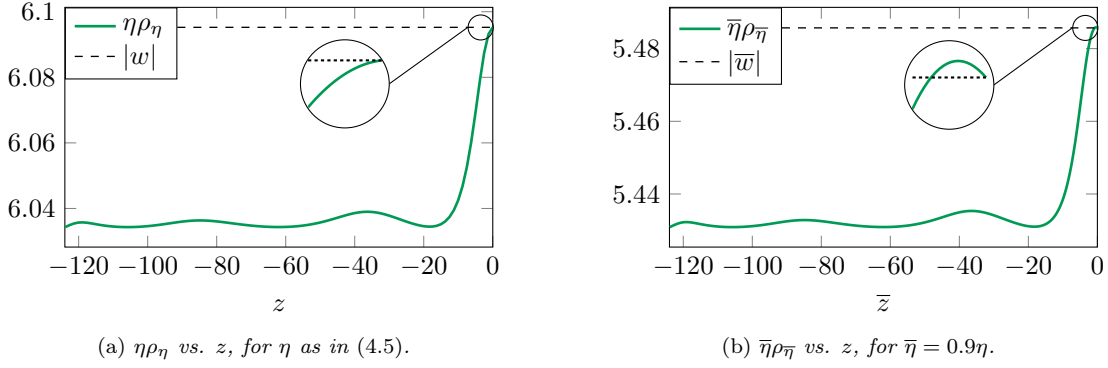
\begin{figure}
		\begin{subfigure}[t]{\subfigsize\textwidth}
			\centering
			\begin{tikzpicture}[scale=\plotimscale,every pin/.style={fill=white}]
			\begin{axis}[height=\aspectratio*\plotimsized\textwidth,width=\plotimsized\textwidth,xmax=0,xmin=-124,legend columns=1,legend style={draw=\legendboxdraw,fill=\legendboxfill,at={(0,1)},anchor=north west},legend cell align={left},
			xlabel={$\phantom{\overline z}z\phantom{\overline z}$}, ylabel={},label style={font=\normalsize},tick label style={font=\normalsize},legend image post style={scale=\legendmarkscale},legend style={nodes={scale=\legendfontscale, transform shape}},grid=none]
			\addplot[color=\colorone,line width=\plotlinewidth pt,mark=none] table [x=etalambda,y=rhoB,col sep=comma] 
			{data/text/stab_weak_coupling_a.csv};\addlegendentry{$\eta\rhoe$}
			\addplot[color=black,dashed,line width=\plotdashedlinewidth pt] table [x=etalambda,y=w,col sep=comma] 
			{data/text/stab_weak_coupling_a.csv};\addlegendentry{$|w|$}
			\coordinate (pta) at (axis cs:-7.3,6.0952);
			\coordinate (ptb) at (axis cs:-30,6.078);
			\end{axis}
			\draw[] (pta)node[circle,anchor=west,draw]{}--(ptb)node[anchor=east,draw,circle,fill=white,inner sep=0,outer sep=0]{
				\begin{tikzpicture}[scale=\zoomsize,trim axis left,trim axis right]
				\begin{axis}[tiny,hide axis,xlabel={},ylabel={},ticks=none,xmin=-1,xmax=0]
				\addplot[color=\colorone,line width=\zoomlinewidth pt,mark=none] table [x=etalambda,y=rhoB,col sep=comma] 
				{data/text/stab_weak_coupling_a_zoom.csv};
				\addplot[color=black,dashed,line width=\zoomdashedlinewidth pt] table [x=etalambda,y=w,col sep=comma] 
				{data/text/stab_weak_coupling_a_zoom.csv};
				\end{axis}
				\end{tikzpicture}
			};
		\end{tikzpicture}
		\caption{$\eta\rhoe$ vs. $z$, for $\eta$ as in \cref{eq:defsmetascalar}.}
		\label{fig:stabweakcoupa}
	\end{subfigure}\hfill%
	\begin{subfigure}[t]{\subfigsize\textwidth}
		\centering
		\begin{tikzpicture}[scale=\plotimscale,every pin/.style={fill=white}]
		\begin{axis}[height=\aspectratio*\plotimsized\textwidth,width=\plotimsized\textwidth,xmax=0,xmin=-124,legend columns=1,legend style={draw=\legendboxdraw,fill=\legendboxfill,at={(0,1)},anchor=north west},legend cell align={left},
		xlabel={$\overline z$}, ylabel={},label style={font=\normalsize},tick label style={font=\normalsize},legend image post style={scale=\legendmarkscale},legend style={nodes={scale=\legendfontscale, transform shape}},grid=none]
		\addplot[color=\colorone,line width=\plotlinewidth pt,mark=none] table [x=etalambda,y=rhoB,col sep=comma] 
		{data/text/stab_weak_coupling_b.csv};\addlegendentry{$\overline\eta\rho_{\overline\eta}$}
		\addplot[color=black,dashed,line width=\plotdashedlinewidth pt] table [x=etalambda,y=w,col sep=comma] 
		{data/text/stab_weak_coupling_b.csv};\addlegendentry{$|\overline w|$}
		\coordinate (ptc) at (axis cs:-7.3,5.4857);
		\coordinate (ptd) at (axis cs:-30,5.47);
		\end{axis}
		\draw[] (ptc)node[circle,anchor=west,draw]{}--(ptd)node[anchor=east,draw,circle,fill=white,inner sep=0,outer sep=0]{
			\begin{tikzpicture}[scale=\zoomsize,trim axis left,trim axis right]
			\begin{axis}[tiny,hide axis,xlabel={},ylabel={},ticks=none,xmin=-1,xmax=0]
			\addplot[color=\colorone,line width=\zoomlinewidth pt,mark=none] table [x=etalambda,y=rhoB,col sep=comma] 
			{data/text/stab_weak_coupling_b_zoom.csv};
			\addplot[color=black,dashed,line width=\zoomdashedlinewidth pt] table [x=etalambda,y=w,col sep=comma] 
			{data/text/stab_weak_coupling_b_zoom.csv};
			\end{axis}
			\end{tikzpicture}
		};
	\end{tikzpicture}
	\caption{$\overline\eta\rho_{\overline\eta}$ vs. $z$, for $\overline\eta=0.9 \eta$.}
	\label{fig:stabweakcoupb}
\end{subfigure}%
\caption{Verification that conditions \cref{eq:defsmetascalar} are sufficient and necessary for the stability of mRKC applied to the $2\times 2$ test problem \cref{eq:twodim}.}
\label{fig:stabweakcoup}
\end{figure}

\subsection{Convergence analysis}
We end this section by proving that the mRKC scheme is first-order accurate. 
\begin{theorem}\label{thm:firstorder}
The mRKC scheme is first-order accurate. 
\end{theorem}
\begin{proof}
We estimate the local error after one step. From \cref{def:fe} with $y$ replaced by $y_0$ in \cref{eq:defu} follows $u(\eta) = y_0 + \eta f(y_0)+\bigo{\eta^2}$ and thus $\fe(y_0) = f(y_0)+\bigo{\eta}$.
Let $y(\tau)$ and $\ye(\tau)$ be the solutions of \cref{eq:ode,eq:odemod} at time $\tau$, respectively, then 
\begin{equation}
	y(\tau)-\ye(\tau) = \tau(f(y_0)-\fe(y_0))+\bigo{\tau^2} = \bigo{\eta\tau+\tau^2}.
\end{equation}
Let $\bar y_\eta(\tau)$ be the solution of $\bar y_\eta' = \bfe(\bar y_\eta)$ with $\bar y_\eta(0)=y_0$ and $\bfe$ as in \cref{eq:defbfe,eq:defbue}. Since the RKC scheme \cref{eq:defbue} is first-order accurate then $u_\eta=u(\eta)+\bigo{\eta^2}$ and $\bfe(y_0)=\fe(y_0)+\bigo{\eta}$, which yields
\begin{equation}
	\ye(\tau) -\bar y_\eta(\tau) = \tau(\fe(y_0)-\bfe(y_0))+\bigo{\tau^2} = \bigo{\eta\tau+\tau^2}.
\end{equation}
Finally, let $y_1$ be the solution after one step of the mRKC scheme \cref{eq:defmrkc,eq:defbfe,eq:defbue}, which can also be seen as the solution after one step of the RKC scheme applied to $\bar y_\eta' = \bfe(\bar y_\eta)$. Using the fact that the RKC scheme \cref{eq:defmrkc} has first-order accuracy then
\begin{equation}
	\bar y_\eta(\tau) -y_1 =  \bigo{\tau^2}.
\end{equation}
By triangular inequality we obtain $|y(\tau)-y_1|= \bigo{\eta\tau+\tau^2}$ and from \cref{eq:defsmeta} follows $\eta\leq 8\tau$, thus $|y(\tau)-y_1|= \bigo{\tau^2}$ and the scheme is first-order accurate.
\end{proof}
Typically $s\gg 1$, i.e. $\eta\ll\tau$, and the error made when approximating $f$ by the averaged force $\fe$ is negligible. In fact, we observe that the difference between the RKC and the mRKC solutions in our numerical experiments in \cref{sec:numexp} is always very small.

\section{Numerical Experiments}\label{sec:numexp}
In this section we compare the mRKC scheme from \cref{sec:fdrkcsq} against the classical RKC method of \cref{sec:rkc} through a series of experiments. First, we apply mRKC to a stiff nonlinear dynamical system to verify convergence in the standard ``ODE sense'' and underpin its efficiency. Then, we apply mRKC to the heat equation and verify convergence in the ``PDE sense'', i.e. when both the mesh size $H$ and the time step $\tau$ decrease simultaneously. In the third experiment, we compare the performance and efficiency of the mRKC and RKC schemes when applied to a linear diffusion problem in complex geometry; here, we also compare mRKC to a second-order accurate RKC scheme (RKC2) \cite{SSV98,VHS90} and the implicit Euler method. In the fourth experiment, we apply the mRKC scheme with the RKC, RKC2 and the implicit Euler method to a nonlinear integro-differential problem.
Finally, we study numerically the stability of mRKC when it is applied to various advection-diffusion-reaction problems.

Both the RKC and mRKC methods need bounds on the spectral radii of the Jacobians of $\ff$ and $\fs$
 to determine the number of stages $s,m$ needed for stability. In our experiments, we estimate them with a cheap nonlinear power method \cite{Lin72,Ver80}. The numerical experiments in \cref{sec:exp_conv2,sec:exp_eff2,sec:exp_stab} were performed using the C++ library \texttt{libMesh} \cite{KPS06}, while for the experiments of \cref{sec:exp_conv0,sec:intdiff} we used the \texttt{Eigen} library \cite{GuB10}.

\subsection{Robertson's stiff test problem}\label{sec:exp_conv0}
First, we study the convergence of the mRKC scheme on a popular stiff test problem, Robertson's nonlinear chemical reaction model \cite{EHL75,HaW02}:
\begin{equation}\label{eq:robertson}
\begin{aligned}
y_1' =& -0.04\, y_1+10^4\, y_2\, y_3, & y_1(0)=&1, \\
y_2' =& \,0.04\, y_1-10^4\, y_2\, y_3-3\cdot 10^7\, y_2^2,\qquad \qquad & y_2(0)=&2\cdot 10^{-5},\\
y_3' =& \,3\cdot 10^7\, y_2^2 , & y_3(0)=& 10^{-1},
\end{aligned}
\end{equation}
where $t\in [0,100]$. With this set of parameters and initial conditions, the only term inducing severe stiffness is $-10^4\,y_2\,y_3$. Thus, we let
\begin{align}
\ff(y)=& \begin{pmatrix}
0 \\ -10^4\, y_2\, y_3\\ 0
\end{pmatrix}, &
\fs(y) =& \begin{pmatrix}
-0.04\, y_1+10^4\, y_2\, y_3\\ 0.04\, y_1-3\cdot 10^7\, y_2^2 \\ 3\cdot 10^7\, y_2^2
\end{pmatrix}, &
f(y)=&\ff(y)+\fs(y).
\end{align}
Now, we solve \eqref{eq:robertson} either with the RKC or the mRKC scheme using step sizes $\tau = 1/2^k$, $k=0,\ldots,7$. For comparison, we use a reference solution obtained with the standard fourth-order Runge--Kutta scheme using $\tau=10^{-4}$. In \cref{fig:conv0_err}, we observe that both the RKC and the mRKC method achieve first-order convergence. In fact, both errors are hardly distinguishable, indicating that the error introduced by the approximation of $f$ by $\fe$ is negligible. We observe in \cref{fig:conv0_eta} that the mean value of $\eta$ during integration is indeed considerably smaller than $\tau$.
\begin{figure}
		\begin{subfigure}[t]{\subfigsize\textwidth}
			\centering
			\begin{tikzpicture}[scale=\plotimscale]
			\begin{loglogaxis}[height=\aspectratio*\plotimsized\textwidth,width=\plotimsized\textwidth,legend columns=1,legend style={draw=\legendboxdraw,fill=\legendboxfill,at={(1,1)},anchor=north east},log basis y={10},log basis x={2},legend cell align={left},x dir=reverse,
			xlabel={$\tau$}, ylabel={$\ell_\infty$ error},label style={font=\normalsize},tick label style={font=\normalsize},legend image post style={scale=\legendmarkscale},legend style={nodes={scale=\legendfontscale, transform shape}}]
			\addplot[color=\localcolor,solid,line width=\plotlinewidth pt,mark=\localmark,mark size=\plotmarksizeu pt] table [x=dt,y=err,col sep=comma] 
			{data/robertson/res_mRKC.csv};\addlegendentry{mRKC}
			\addplot[color=\classicalcolor,line width=\plotlinewidth pt,mark=\classicalmark,mark size=\plotmarksized pt] table [x=dt,y=err,col sep=comma] 
			{data/robertson/res_RKC.csv};\addlegendentry{RKC}
			\addplot[black,dashed,line width=\plotdashedlinewidth pt,domain=0.0078125:1] (x,0.00025*x);\addlegendentry{$\bigo{\tau}$}
			\end{loglogaxis}
			\end{tikzpicture}
			\caption{$\ell_\infty$-error vs. the step size $\tau$.}
			\label{fig:conv0_err}
		\end{subfigure}\hfill%
		\begin{subfigure}[t]{\subfigsize\textwidth}
			\centering
			\begin{tikzpicture}[scale=\plotimscale]
			\begin{loglogaxis}[height=\aspectratio*\plotimsized\textwidth,width=\plotimsized\textwidth,legend columns=1,legend style={draw=\legendboxdraw,fill=\legendboxfill,at={(1,1)},anchor=north east},log basis y={10},log basis x={2},legend cell align={left},x dir=reverse,
			xlabel={$\tau$}, ylabel={$[t]$},label style={font=\normalsize},tick label style={font=\normalsize},legend image post style={scale=\legendmarkscale},legend style={nodes={scale=\legendfontscale, transform shape}}]
			\addplot[color=\colorone,solid,line width=\plotlinewidth pt,mark=\markone,mark size=\plotmarksizeu pt] table [x=dt,y=eta,col sep=comma] 
			{data/robertson/res_mRKC.csv};\addlegendentry{Mean of $\eta$}
			\addplot[color=\colortwo,line width=\plotlinewidth pt,mark=\marktwo,mark size=\plotmarksized pt] table [x=dt,y=dt,col sep=comma] 
			{data/robertson/res_mRKC.csv};\addlegendentry{$\tau$}
			\end{loglogaxis}
			\end{tikzpicture}
			\caption{Comparing the size of $\eta$ and $\tau$.}
			\label{fig:conv0_eta}
		\end{subfigure}%
	\caption{Robertson's stiff test problem. Convergence and comparison of $\eta$ against $\tau$.}
	\label{fig:conv0}
\end{figure}
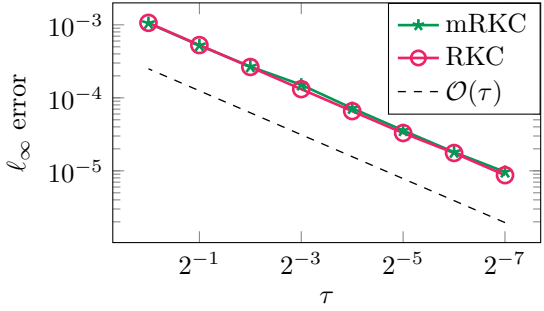
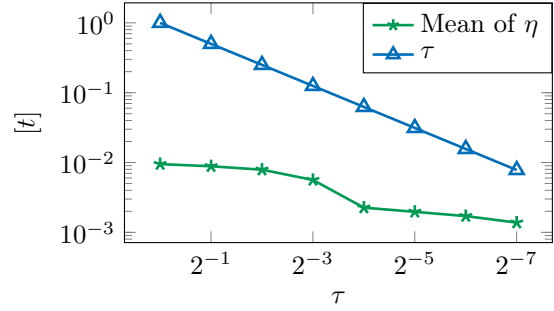

Next, we compare the two schemes for a fixed step size $\tau=1$. In \cref{fig:conv0_stages}, we display the number of stages taken by the mRKC and the RKC method at each time step with respect to $t\in [0,100]$. Moreover, \cref{fig:conv0_rho} depicts the evolution of the spectral radii $\rho,\rhof,\rhos$ of the Jacobians of $f,\ff,\fs$, respectively. We observe that $\rhos$ decreases with time and consequently the mRKC scheme decreases the number $s$ of expensive function evaluations $\fs$ per step.
In contrast, $\rho$ increases and thus the RKC scheme must increase the number $s$ of $\fs$ function evaluations, although this term does not introduce any stiffness; indeed, $\rho$ increases only because of the term contained in $\ff$.
Finally, we notice in \cref{fig:conv0_stages} that the mRKC scheme increases the number $m$ of (cheap) function evaluations $\ff$ because of the increase in $\rhof$ and $\eta$; indeed, $\eta$ also increases due to the decrease in $s$ and \eqref{eq:defsmeta}. This added cost, however, is much smaller than that from the many additional (expensive) evaluations of $\fs$ required by the RKC method.
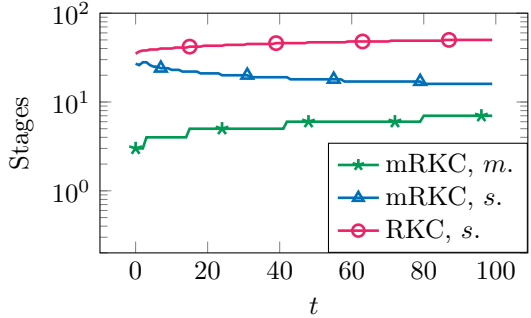
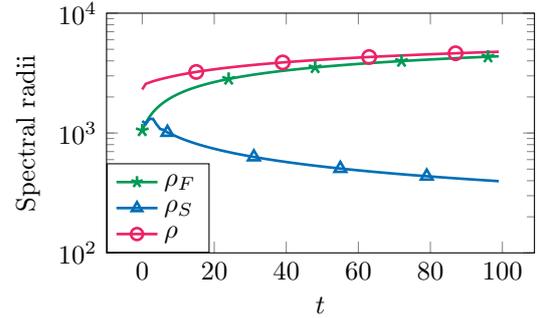
\begin{figure}
		\begin{subfigure}[t]{\subfigsize\textwidth}
			\centering
			\begin{tikzpicture}[scale=\plotimscale]
			\begin{semilogyaxis}[height=\aspectratio*\plotimsized\textwidth,width=\plotimsized\textwidth,legend columns=1,legend style={draw=\legendboxdraw,fill=\legendboxfill,at={(1,0)},anchor=south east},log basis y={10},log basis x={2},legend cell align={left},ymin=0.2,ymax=100,
			xlabel={$t$}, ylabel={Stages},label style={font=\normalsize},tick label style={font=\normalsize},legend image post style={scale=\legendmarkscale},legend style={nodes={scale=\legendfontscale, transform shape}},grid=none]
			\addplot[color=\localcolorF,solid,line width=\plotlinewidth pt,mark=\localmarkF,mark size=\plotmarksizeu pt,mark repeat=24,mark phase=0] table [x=t,y=m,col sep=comma] 
			{data/robertson/rho_sm_mRKC.csv};\addlegendentry{mRKC, $m$.}
			\addplot[color=\localcolorS,solid,line width=\plotlinewidth pt,mark=\localmarkS,mark size=\plotmarksizeu pt,mark repeat=24,mark phase=8] table [x=t,y=s,col sep=comma] 
			{data/robertson/rho_sm_mRKC.csv};\addlegendentry{mRKC, $s$.}
			\addplot[color=\classicalcolor,line width=\plotlinewidth pt,mark=\classicalmark,mark size=\plotmarksizeu pt,mark repeat=24,mark phase=16] table [x=t,y=s,col sep=comma] 
			{data/robertson/rho_s_RKC.csv};\addlegendentry{RKC, $s$.}
			\end{semilogyaxis}
			\end{tikzpicture}
			\caption{Stages needed by RKC and mRKC vs. time $t$, for a fixed step size $\tau=1$.}
			\label{fig:conv0_stages}
		\end{subfigure}\hfill%
		\begin{subfigure}[t]{\subfigsize\textwidth}
			\centering
			\begin{tikzpicture}[scale=\plotimscale]
			\begin{semilogyaxis}[height=\aspectratio*\plotimsized\textwidth,width=\plotimsized\textwidth,legend columns=1,legend style={draw=\legendboxdraw,fill=\legendboxfill,at={(0,0)},anchor=south west},log basis y={10},legend cell align={left},ymin=100,ymax=10000,
			xlabel={$t$}, ylabel={Spectral radii},label style={font=\normalsize},tick label style={font=\normalsize},legend image post style={scale=\legendmarkscale},legend style={nodes={scale=\legendfontscale, transform shape}},grid=none]
			\addplot[color=\localcolorF,solid,line width=\plotlinewidth pt,mark=\localmarkF,mark size=\plotmarksizeu pt,mark repeat=24,mark phase=0] table [x=t,y=rhoF,col sep=comma] 
			{data/robertson/rho_sm_mRKC.csv};\addlegendentry{$\rhof$}
			\addplot[color=\localcolorS,solid,line width=\plotlinewidth pt,mark=\localmarkS,mark size=\plotmarksizeu pt,mark repeat=24,mark phase=8] table [x=t,y=rhoS,col sep=comma] 
			{data/robertson/rho_sm_mRKC.csv};\addlegendentry{$\rhos$}
			\addplot[color=\classicalcolor,line width=\plotlinewidth pt,mark=\classicalmark,mark size=\plotmarksizeu pt,mark repeat=24,mark phase=16] table [x=t,y=rho,col sep=comma] 
			{data/robertson/rho_s_RKC.csv};\addlegendentry{$\rho$}
			\end{semilogyaxis}
			\end{tikzpicture}
			\caption{Evolution of the spectral radii $\rho$, $\rhof$, $\rhos$ vs. time $t$.}
			\label{fig:conv0_rho}
		\end{subfigure}%
	\caption{Robertson's stiff test problem. Comparison of spectral radii and number of stages taken by the mRKC and the RKC scheme.}
	\label{fig:conv0_rho_stages}
\end{figure}

\subsection{Heat equation in the unit square}\label{sec:exp_conv2}
Next, we verify the space-time convergence properties of the mRKC method. To do so, we consider the heat equation in the unit square $\Omega=[0,1]\times [0,1]$,
\begin{equation}\label{eq:par1}
\begin{aligned}
\dt u-\Delta u &=g \qquad &&\text{in }\Omega\times [0,T],\\
u&=0 &&\text{in } \partial\Omega \times [0,T],\\
u&=0 &&\text{in } \Omega\times\{0\},
\end{aligned}
\end{equation}
where $T=1/2$ and $g$ is chosen such that $u(\bx,t)=\sin(\pi x_1)^2\sin(\pi x_2)^2\sin(\pi t)^2$ is the exact solution. 

Starting from a mesh of $2^j\times 2^j$ simplicial elements with $j=2,\ldots,5$, we locally refine twice all the elements inside the square $\OF=(1/4,3/4)\times (1/4,3/4)$. Each refinement step is performed by splitting all edges of any simplex, i.e. every triangle is split into four self-similar children. 
Let $\mathcal{M}$ be the set of elements in the mesh and $\mathcal{M}_F=\{T\in\mathcal{M}\,:\, \overline{T}\cap \overline{\Omega}_F\neq \emptyset\}$ the set of refined elements or their direct neighbors. Then $h=H/4$ is the diameter of the elements inside of $\OF$, with $H$ the diameter of the elements outside of $\OF$. 

Next, we discretize \eqref{eq:par1} in space with first-order DG-FE \cite{PiE12} on the mesh $\mathcal{M}$. After inverting the block-diagonal mass matrix, the resulting system is 
\begin{equation}
y'=A\, y+G,  \qquad\qquad y(0)=y_0,
\end{equation}
where $A\in\Rb^{N\times N}$ and $G\in C([0,T],\Rb^N)$ corresponds to the spatial discretization of $g(\cdot,t)$. Let $D\in\Rb^{N\times N}$ be a diagonal matrix with $D_{ii}=1$ if the $i$th degree of freedom belongs to an element in $\mathcal{M}_F$ and $D_{ii}=0$ otherwise. We also introduce 
\begin{align}\label{eq:splitpar}
A_F=&DA, & A_S=&(I-D)A &\mbox{and} &&\ff(y)=&A_F\, y, & \fs(t,y)=&A_S\, y+G(t),
\end{align}
with $I$ the identity. It is well-known that the spectral radii $\rhos$ and $\rhof$ of $A_S$ and $A_F$ behave as $\bigo{1/H^2}$ and $\bigo{1/h^2}=\bigo{16/H^2}$, respectively. 

We now consider a sequence of meshes with $j=2,\ldots,5$ and solve \eqref{eq:par1} either with
the mRKC or the RKC scheme using the same step size $\tau=1/2^j$. The parameters $s$ and $m$ for mRKC are chosen according to \eqref{eq:defsmetaweak}. In \cref{fig:conv2}, we display the $H^1(\Omega)$ errors at final time for mRKC and RKC. Both methods yield space-time first-order convergence and result in similar errors. In \cref{fig:stages}, we show the number of stages needed by RKC and mRKC. For both schemes, $s$ increases as the mesh size $H$ decreases, but for mRKC, $s$ is much smaller, since it only depends on the coarse elements, while $m$ remains constant due to the constant ratio between $\rhof$ and $\rhos$.
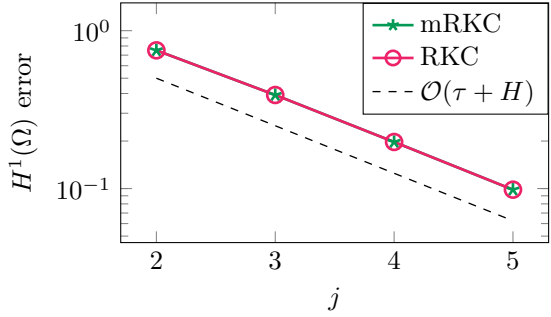
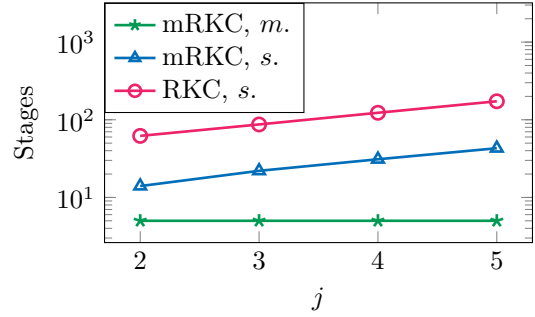
\begin{figure}
		\begin{subfigure}[t]{\subfigsize\textwidth}
			\centering
			\begin{tikzpicture}[scale=\plotimscale]
			\begin{semilogyaxis}[height=\aspectratio*\plotimsized\textwidth,width=\plotimsized\textwidth,legend columns=1,legend style={draw=\legendboxdraw,fill=\legendboxfill,at={(1,1)},anchor=north east},log basis y={10},ymax=1.5,legend cell align={left},
			xlabel={$j$}, ylabel={$H^1(\Omega)$ error},label style={font=\normalsize},tick label style={font=\normalsize},legend image post style={scale=\legendmarkscale},legend style={nodes={scale=\legendfontscale, transform shape}}]
			\addplot[color=\localcolor,solid,line width=\plotlinewidth pt,mark=\localmark,mark size=\plotmarksizeu pt] table [x=k,y=errh1,col sep=comma] 
			{data/heat_eq/sol_mrkc.csv};\addlegendentry{mRKC}
			\addplot[color=\classicalcolor,line width=\plotlinewidth pt,mark=\classicalmark,mark size=\plotmarksized pt] table [x=k,y=errh1,col sep=comma] 
			{data/heat_eq/sol_rkc.csv};\addlegendentry{RKC}
			\addplot[black,dashed,line width=\plotdashedlinewidth pt,domain=2:5] (x,2/2^x);\addlegendentry{$\bigo{\tau+H}$}
			\end{semilogyaxis}
			\end{tikzpicture}
			\caption{Convergence of RKC and mRKC.}
			\label{fig:conv2}
		\end{subfigure}\hfill%
		\begin{subfigure}[t]{\subfigsize\textwidth}
			\centering
			\begin{tikzpicture}[scale=\plotimscale]
			\begin{semilogyaxis}[height=\aspectratio*\plotimsized\textwidth,width=\plotimsized\textwidth,legend columns=1,legend style={draw=\legendboxdraw,fill=\legendboxfill,at={(0,1)},anchor=north west},log basis y={10},ymax=3200,legend cell align={left},
			xlabel={$j$}, ylabel={Stages},label style={font=\normalsize},tick label style={font=\normalsize},legend image post style={scale=\legendmarkscale},legend style={nodes={scale=\legendfontscale, transform shape}},grid=none]
			\addplot[color=\localcolorF,solid,line width=\plotlinewidth pt,mark=\localmarkF,mark size=\plotmarksizeu pt] table [x=k,y=m,col sep=comma] 
			{data/heat_eq/sol_mrkc.csv};\addlegendentry{mRKC, $m$.}
			\addplot[color=\localcolorS,solid,line width=\plotlinewidth pt,mark=\localmarkS,mark size=\plotmarksizeu pt] table [x=k,y=s,col sep=comma] 
			{data/heat_eq/sol_mrkc.csv};\addlegendentry{mRKC, $s$.}
			\addplot[color=\classicalcolor,line width=\plotlinewidth pt,mark=\classicalmark,mark size=\plotmarksizeu pt] table [x=k,y=s,col sep=comma] 
			{data/heat_eq/sol_rkc.csv};\addlegendentry{RKC, $s$.}
			\end{semilogyaxis}
			\end{tikzpicture}
			\caption{Number of stages needed by RKC and mRKC vs. refinement level $j$.}
			\label{fig:stages}
		\end{subfigure}%
	\caption{Heat equation in the unit square. Space-time convergence and number of stages.}
	\label{fig:convstages}
\end{figure}

\subsection{Diffusion across a narrow channel}\label{sec:exp_eff2}
To illustrate the efficiency of the mRKC method in a situation where geometry constraints require local mesh refinement, we consider the heat equation
\begin{equation}\label{eq:par3}
\begin{aligned}
\dt u-\Delta u &=g \qquad &&\text{in }\Omega_\delta\times [0,T],\\
\nabla u\cdot \bm{n}&=0 &&\text{in } \partial\Omega_\delta \times [0,T],\\
u&=0 &&\text{in } \Omega_\delta\times\{0\},
\end{aligned}
\end{equation}
with $T=0.1$ inside $\Omega_\delta$, which consists of two $10\times 5$ rectangles linked by a narrow $\delta\times 0.05$ channel of width $\delta>0$, see \cref{fig:eff2_sol}. The right-hand side $g(\bx,t)=\sin(10\pi t)^2 e^{-5\Vert \bx-\bm{c}\Vert^2}$ corresponds to a smoothed Gaussian point source centered at $\bm{c}$ in the middle of the upper rectangle.
\begin{figure}
		\hfil%
		\begin{subfigure}[t]{\subfigsizemed\textwidth}
			\begin{tikzpicture}[spy using outlines= {circle, magnification=6, connect spies}]
			\coordinate (a) at (-2.19,0);
			\coordinate (b) at (1,0);
			\coordinate (spypoint) at (b);
			\coordinate (magnifyglass) at (2.2,-1.2);
			\node at (a) {\includegraphics[width=0.15\textwidth]{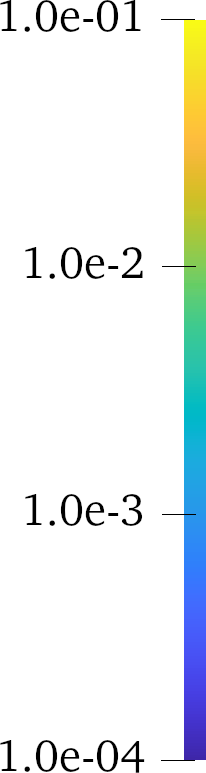}};
			\node at (b) {\includegraphics[width=0.64\textwidth]{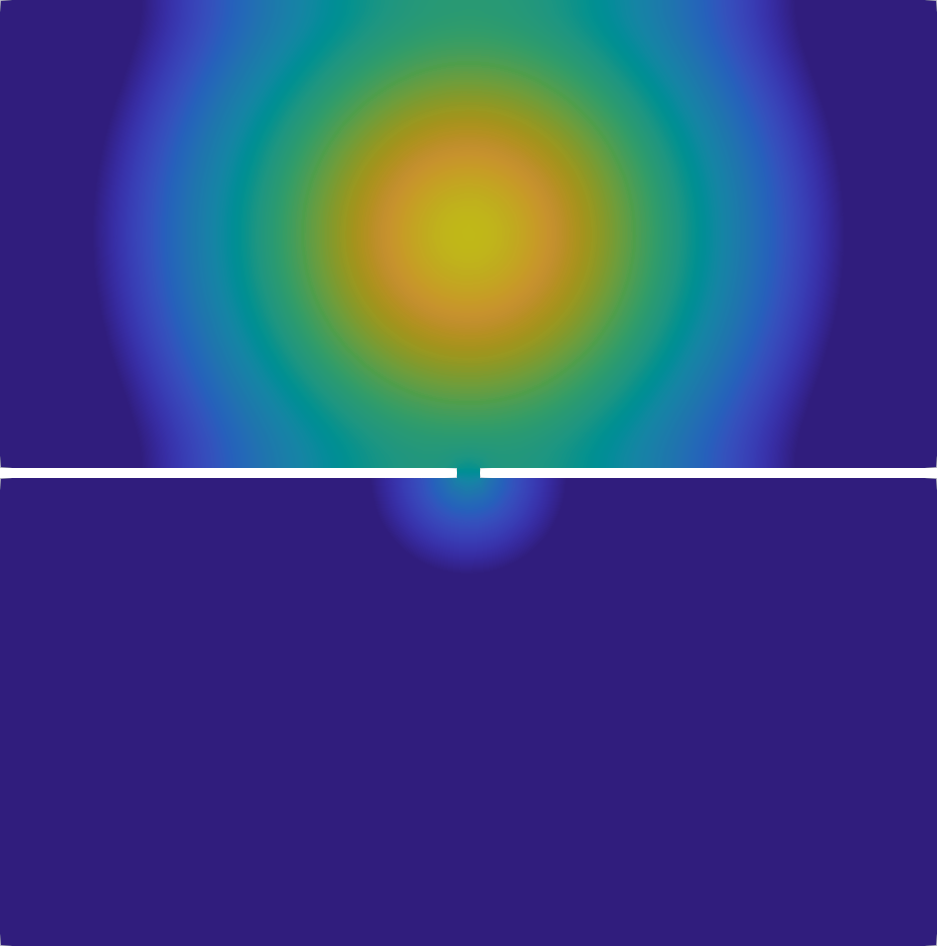}};
			\spy [WildStrawberry, size=2cm,magnification=6] on (spypoint) in node[fill=white] at (magnifyglass);
			\end{tikzpicture}
			\caption{Solution for $\delta=1/2^2$.}
			\label{fig:eff2_sola}
		\end{subfigure}\hfil%
		\begin{subfigure}[t]{\subfigsizemed\textwidth}
			\begin{tikzpicture}[spy using outlines= {circle, magnification=6, connect spies}]
			\coordinate (a) at (-2.19,0);
			\coordinate (b) at (1,0);
			\coordinate (spypoint) at (b);
			\coordinate (magnifyglass) at (2.2,-1.2);
			\node at (a) {\includegraphics[width=0.15\textwidth]{images/narrow_channel/bar.png}};
			\node at (b) {\includegraphics[width=0.64\textwidth]{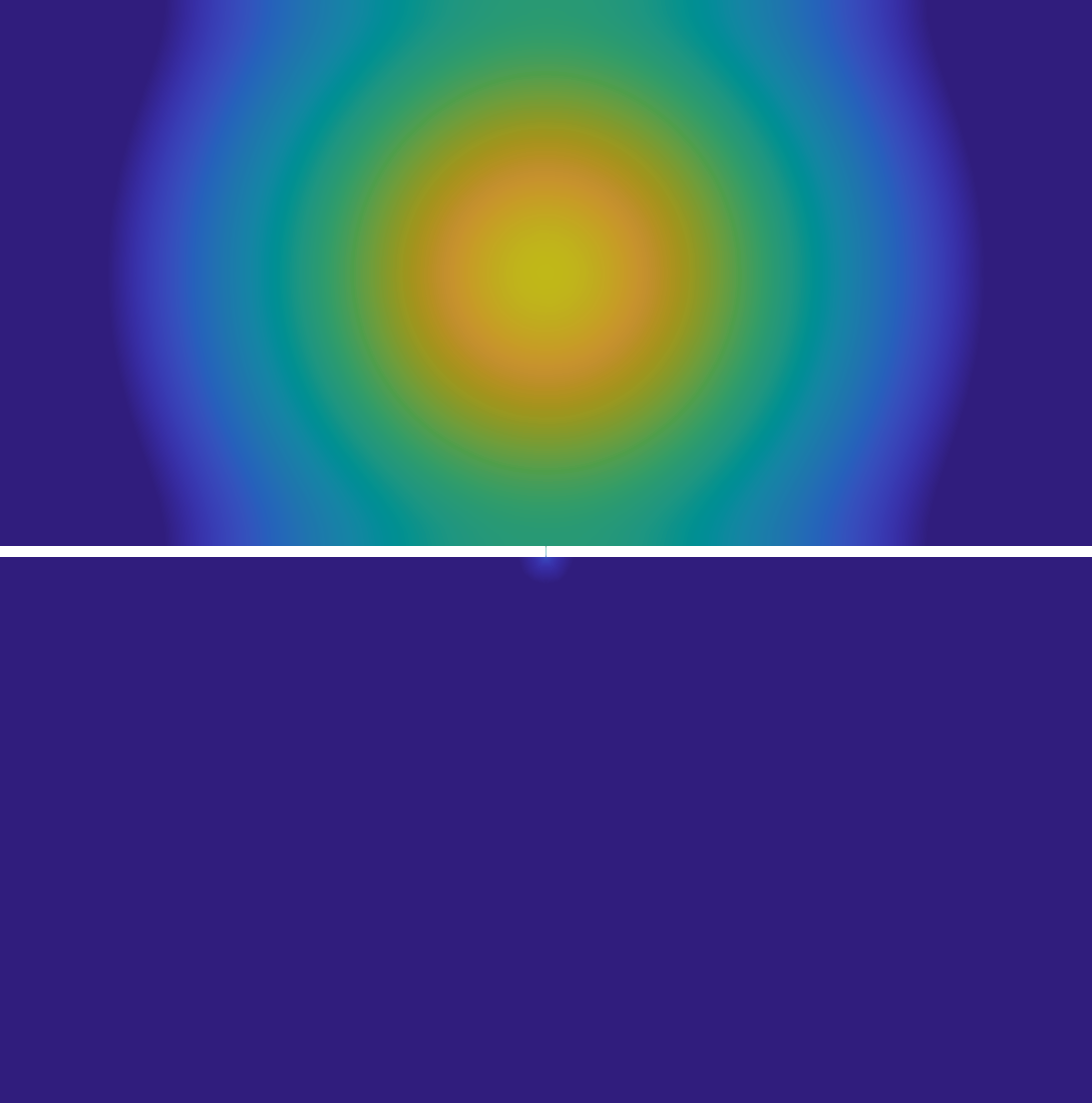}};
			\spy [WildStrawberry, size=2cm,magnification=6] on (spypoint) in node[fill=white] at (magnifyglass);
			\end{tikzpicture}
			\caption{Solution for $\delta=1/2^7$.}
			\label{fig:eff2_solb}
		\end{subfigure}%
	\hfil%
	\caption{Narrow channel. Numerical solutions of \cref{eq:par3} at $t=1$ using mRKC for a channel width $\delta=1/2^2$ or $\delta=1/2^7$.}
	\label{fig:eff2_sol}
\end{figure}

Inside $\Omega_\delta$, we use a Delaunay triangulation with maximal element size $H\approx 0.015$. As $\delta$ approaches zero, the elements inside the channel become increasingly smaller and the system stiffer. For each $\delta>0$, we define a neighborhood $\Omega_{F,\delta}\subset\Omega_\delta$ of the channel and $\mathcal{M}_F$, $A$, $A_F$, $A_S$, $\ff$, $\fs$ as in \cref{sec:exp_conv2}. Here, $\Omega_{F,\delta}$ is chosen such that the spectral radius of $A_S$ is almost independent of $\delta$ and only that of $A_F$ increases with decreasing $\delta$. Hence, $\Omega_{F,\delta}$ contains the channel together with all neighboring elements of mesh size smaller than $H$, see \cref{fig:eff2_grida,fig:eff2_gridb}. 
\begin{figure}
		\hfil%
		\begin{subfigure}[t]{\subfigsizemed\textwidth}
			\centering
			\includegraphics[trim=0cm 0cm 0cm 0cm, clip, width=0.64\textwidth]{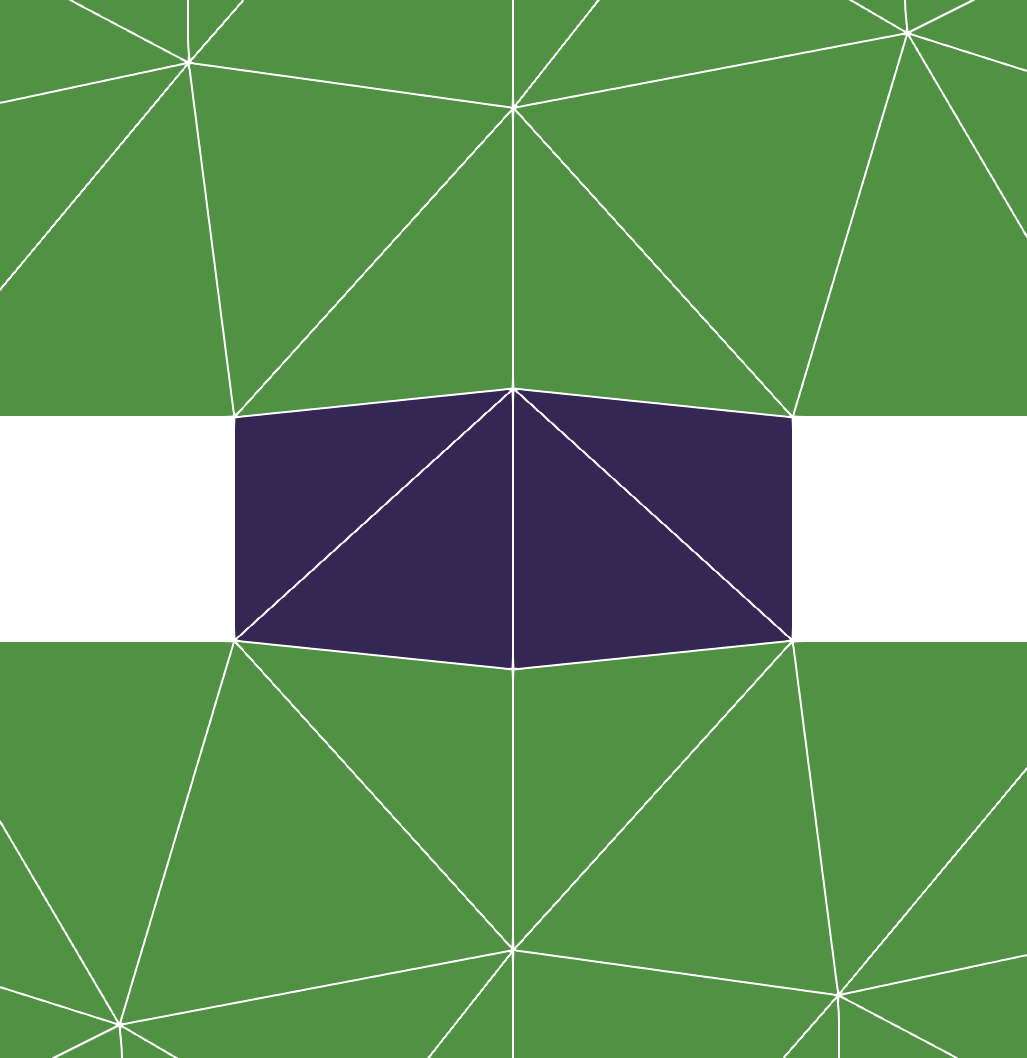}
			\caption{Zoom over the channel with $\delta=1/2^2$.}
			\label{fig:eff2_grida}
		\end{subfigure}\hfil%
		\begin{subfigure}[t]{\subfigsizemed\textwidth}
			\centering
			\includegraphics[trim=0cm 0cm 0cm 0cm, clip, width=0.64\textwidth]{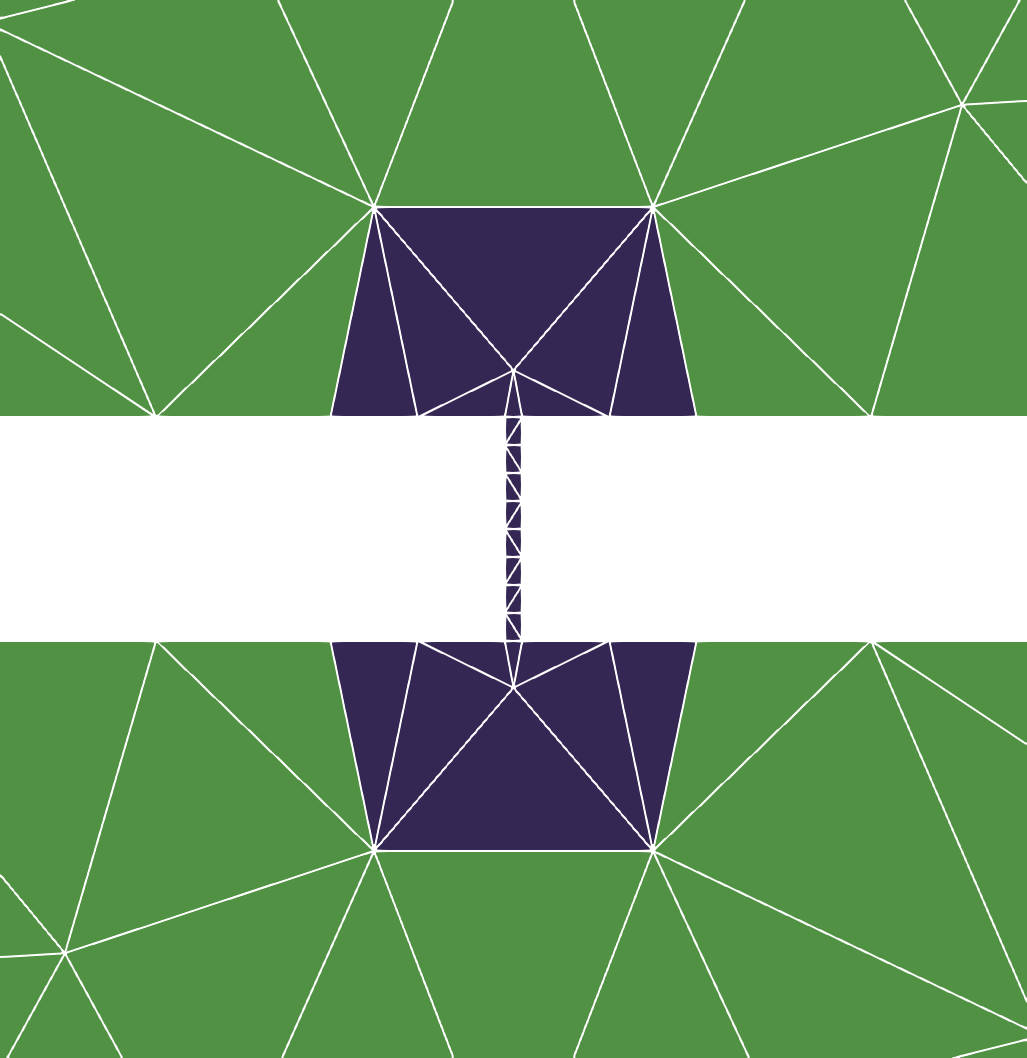}
			\caption{Zoom over the channel with $\delta=1/2^7$.}
			\label{fig:eff2_gridb}
		\end{subfigure}%
	\hfil%
	\caption{Narrow channel. Zoom of the FE mesh for a channel width $\delta=1/2^2$ or $1/2^7$, with the subdomain $\Omega_{F,\delta}$ (in blue).}
	\label{fig:eff2_grid}
\end{figure}

For varying channel width $\delta=1/2^{k}$, $k=0,\ldots,15$, we now solve \eqref{eq:par3} with the RKC and mRKC method using the choice of parameters \eqref{eq:defsmetaweak} with $\tau=0.01$. In \cref{fig:eff2_speed}, the relative speed-up defined as the ratio between the computational times of RKC and mRKC always exceeds one and reaches a value as high as 40. 
Note that the relative error between the two solutions in $L^2(\Omega_\delta)$ or $H^1(\Omega_\delta)$ norm is at most $3\cdot 10^{-4}$, as shown in \cref{fig:eff2_relerr}.

In \cref{fig:eff2_rho}, we display for varying $\delta$ also the spectral radii $\rho,\rhof,\rhos$ of $A,A_F,A_S$, respectively; note that $\rhof$ and $\rho$ essentially coincide. For large $\delta$, we also have $\rhof\approx\rhos$ since the typical element size is sufficiently small to resolve the channel (\cref{fig:eff2_grida}). For $\delta$ small, we observe that $\rho$, $\rhof$ increase as $1/\delta^2$ while $\rhos$ remains almost constant. \Cref{fig:eff2_stages} shows that the number $s$ of stages in the mRKC scheme remains constant, as does $\rhos$ in \cref{fig:eff2_rho}, while $m$ increases (as $\rhof$). 
For large $\delta$, we have $\rhof\approx\rhos$ and thus $m=1$; then, the RKC and mRKC schemes coincide. Indeed, as is shown in \cref{fig:eff2_relerr}, for $m=1$ then the relative error between the RKC and mRKC solutions is of the order of machine precision.

In \cref{fig:eff2_comptime}, we observe that for $\delta$ large the CPU times of the two methods are similar; thus, despite $\rhof\approx\rhos$, there is no loss in efficiency and the speed-up is at least one (\cref{fig:eff2_speed}). For moderate values of $\delta$, the cost of RKC increases proportionally to $1/\delta$, while the cost of mRKC is hardly affected. For even smaller $\delta$, the number of evaluations of $\ff$ increases and so does its cost with respect to $\fs$ (see \cref{fig:eff2_cf}), since the number of elements in $\mathcal{M}_F$ increases (\cref{fig:eff2_grid}). In this regime, evaluation of $\ff$ dominates the computational cost of mRKC, which increases linearly in $1/\delta$, too. Still, the mRKC method remains about forty times faster than the classical RKC method for this particular discretization inside $\Omega_\delta$, see \cref{fig:eff2_speed}. 

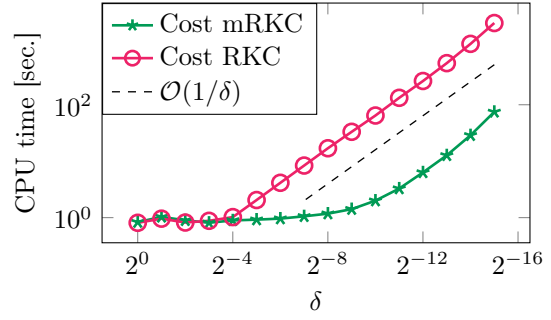
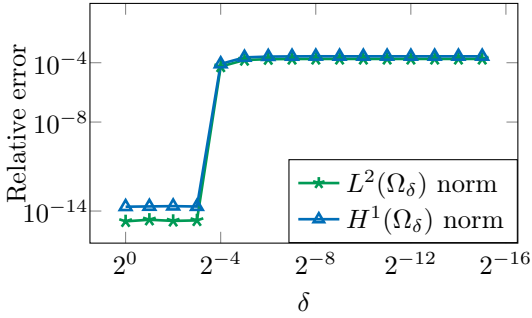
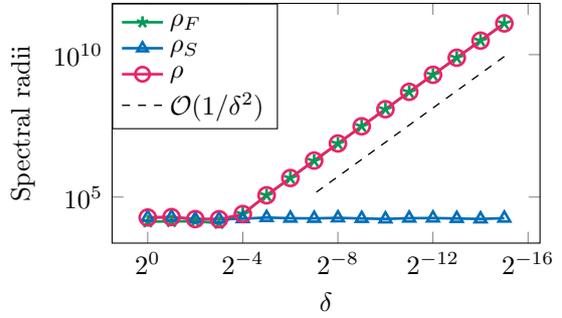
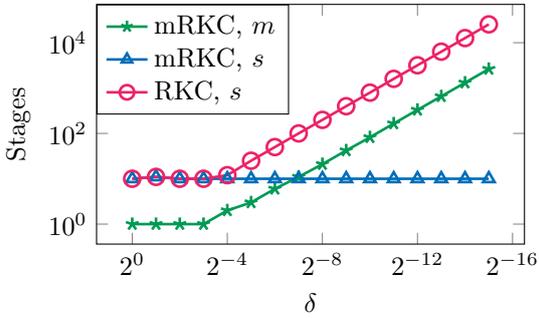
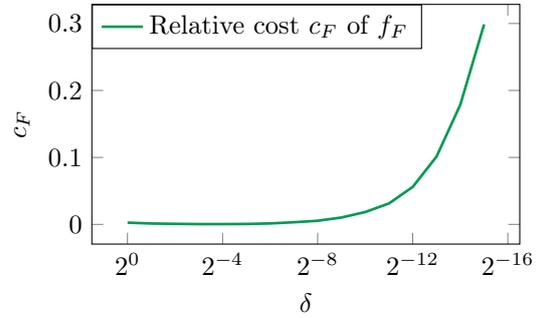
\begin{figure}
		\begin{subfigure}[t]{\subfigsize\textwidth}
			\centering
			\begin{tikzpicture}[scale=\plotimscale]
			\begin{semilogxaxis}[height=\aspectratio*\plotimsized\textwidth,width=\plotimsized\textwidth,legend columns=1,legend cell align={left},legend style={draw=\legendboxdraw,fill=\legendboxfill,at={(0,1)},anchor=north west},log basis x={2},x dir=reverse,ytick={1,10,20,30,40,50},ymax=50,
			xlabel={$\delta$}, ylabel={Speed-up},label style={font=\normalsize},tick label style={font=\normalsize},legend image post style={scale=\legendmarkscale},legend style={nodes={scale=\legendfontscale, transform shape}}]
			\addplot[color=\colorone,solid,line width=\plotlinewidth pt,mark=none,mark size=\plotmarksizeu pt] table [x=delta,y=speedup,col sep=comma] 
			{data/narrow_channel/results.csv};\addlegendentry{Speed-up}
			\addplot[black,dashed,line width=\plotdashedlinewidth pt,domain=3.0517578125e-05:1] (x,1);\addlegendentry{$1$}
			\end{semilogxaxis}
			\end{tikzpicture}
			\caption{Relative speed-up of mRKC over RKC.}
			\label{fig:eff2_speed}
		\end{subfigure}\hfill%
		\begin{subfigure}[t]{\subfigsize\textwidth}
			\centering
			\begin{tikzpicture}[scale=\plotimscale]
			\begin{loglogaxis}[height=\aspectratio*\plotimsized\textwidth,width=\plotimsized\textwidth,legend columns=1,legend cell align={left},legend style={draw=\legendboxdraw,fill=\legendboxfill,at={(0,1)},anchor=north west},log basis x={2},log basis y={10},x dir=reverse,ylabel style={yshift=-2pt},
			xlabel={$\delta$}, ylabel={CPU time [sec.]},label style={font=\normalsize},tick label style={font=\normalsize},legend image post style={scale=\legendmarkscale},legend style={nodes={scale=\legendfontscale, transform shape}}]
			\addplot[color=\localcolor,solid,line width=\plotlinewidth pt,mark=\localmark,mark size=\plotmarksizeu pt] table [x=delta,y=mRKC_cost,col sep=comma] 
			{data/narrow_channel/results.csv};\addlegendentry{Cost mRKC}
			\addplot[color=\classicalcolor,solid,line width=\plotlinewidth pt,mark=\classicalmark,mark size=\plotmarksized pt] table [x=delta,y=RKC_cost,col sep=comma] 
			{data/narrow_channel/results.csv};\addlegendentry{Cost RKC}
			\addplot[black,dashed,line width=\plotdashedlinewidth pt,domain=3.0517578125e-05:0.0078] (x,1/x/64);\addlegendentry{$\bigo{1/\delta}$}
			\end{loglogaxis}
			\end{tikzpicture}
			\caption{Total CPU time w.r.t. $\delta$.}
			\label{fig:eff2_comptime}
		\end{subfigure}\\ \vspace{0.5cm}%
		\begin{subfigure}[t]{\subfigsize\textwidth}
			\centering
			\begin{tikzpicture}[scale=\plotimscale]
			\begin{loglogaxis}[height=\aspectratio*\plotimsized\textwidth,width=\plotimsized\textwidth,legend columns=1,legend cell align={left},legend style={draw=\legendboxdraw,fill=\legendboxfill,at={(1,0)},anchor=south east},log basis x={2},log basis y={10},x dir=reverse, ytick={1e-14,1e-8,1e-4},ymax=1,ylabel style={yshift=-8pt},yticklabel style={xshift=3pt},
			xlabel={$\delta$}, ylabel={Relative error},label style={font=\normalsize},tick label style={font=\normalsize},legend image post style={scale=\legendmarkscale},legend style={nodes={scale=\legendfontscale, transform shape}}]
			\addplot[color=\colorone,solid,line width=\plotlinewidth pt,mark=\markone,mark size=\plotmarksizeu pt] table [x=delta,y=relerrl2,col sep=comma] 
			{data/narrow_channel/results.csv};\addlegendentry{$L^2(\Omega_\delta)$ norm}
			\addplot[color=\colortwo,solid,line width=\plotlinewidth pt,mark=\marktwo,mark size=\plotmarksizeu pt] table [x=delta,y=relerrh1,col sep=comma] 
			{data/narrow_channel/results.csv};\addlegendentry{$H^1(\Omega_\delta)$ norm }
			\end{loglogaxis}
			\end{tikzpicture}
			\caption{RKC and mRKC solutions' relative error $\Vert u^{\mRKCop}-u^{\RKCop}\Vert/\Vert u^{\RKCop}\Vert$.}
			\label{fig:eff2_relerr}
		\end{subfigure}\hfill%
		\begin{subfigure}[t]{\subfigsize\textwidth}
			\centering
			\begin{tikzpicture}[scale=\plotimscale]
			\begin{loglogaxis}[height=\aspectratio*\plotimsized\textwidth,width=\plotimsized\textwidth,legend columns=1,legend style={draw=\legendboxdraw,fill=\legendboxfill,at={(0,1)},anchor=north west},log basis x={2},log basis y={10},x dir=reverse,legend cell align={left},
			xlabel={$\delta$}, ylabel={Spectral radii},label style={font=\normalsize},tick label style={font=\normalsize},legend image post style={scale=\legendmarkscale},legend style={nodes={scale=\legendfontscale, transform shape}}]
			\addplot[color=\localcolorF,solid,line width=\plotlinewidth pt,mark=\localmarkF,mark size=\plotmarksizeu pt] table [x=delta,y=mRKC_rhoF,col sep=comma] 
			{data/narrow_channel/results.csv};\addlegendentry{$\rhof$}
			\addplot[color=\localcolorS,solid,line width=\plotlinewidth pt,mark=\localmarkS,mark size=\plotmarksizeu pt] table [x=delta,y=mRKC_rhoS,col sep=comma] 
			{data/narrow_channel/results.csv};\addlegendentry{$\rhos$}
			\addplot[color=\classicalcolor,solid,line width=\plotlinewidth pt,mark=\classicalmark,mark size=\plotmarksized pt] table [x=delta,y=RKC_rho,col sep=comma] 
			{data/narrow_channel/results.csv};\addlegendentry{$\rho$}
			\addplot[black,dashed,line width=\plotdashedlinewidth pt,domain=3.0517578125e-05:0.0078] (x,8/x/x);\addlegendentry{$\bigo{1/\delta^2}$}
			\end{loglogaxis}
			\end{tikzpicture}
			\caption{Spectral radii w.r.t. $\delta$.}
			\label{fig:eff2_rho}
		\end{subfigure}\\ \vspace{0.5cm}%
		\begin{subfigure}[t]{\subfigsize\textwidth}
			\centering
			\begin{tikzpicture}[scale=\plotimscale]
			\begin{loglogaxis}[height=\aspectratio*\plotimsized\textwidth,width=\plotimsized\textwidth,legend columns=1,legend style={draw=\legendboxdraw,fill=\legendboxfill,at={(0,1)},anchor=north west},log basis x={2},log basis y={10},x dir=reverse,legend cell align={left},legend cell align={left},
			xlabel={$\delta$}, ylabel={Stages},label style={font=\normalsize},tick label style={font=\normalsize},legend image post style={scale=\legendmarkscale},legend style={nodes={scale=\legendfontscale, transform shape}}]
			\addplot[color=\localcolorF,solid,line width=\plotlinewidth pt,mark=\localmarkF,mark size=\plotmarksizeu pt] table [x=delta,y=mRKC_m,col sep=comma] 
			{data/narrow_channel/results.csv};\addlegendentry{mRKC, $m$}
			\addplot[color=\localcolorS,solid,line width=\plotlinewidth pt,mark=\localmarkS,mark size=\plotmarksizeu pt] table [x=delta,y=mRKC_s,col sep=comma] 
			{data/narrow_channel/results.csv};\addlegendentry{mRKC, $s$}
			\addplot[color=\classicalcolor,solid,line width=\plotlinewidth pt,mark=\classicalmark,mark size=\plotmarksized pt] table [x=delta,y=RKC_s,col sep=comma] 
			{data/narrow_channel/results.csv};\addlegendentry{RKC, $s$}
			\end{loglogaxis}
			\end{tikzpicture}
			\caption{Number of stages needed by RKC and mRKC.}
			\label{fig:eff2_stages}
		\end{subfigure} \hfill%
		\begin{subfigure}[t]{\subfigsize\textwidth}
			\centering
			\begin{tikzpicture}[scale=\plotimscale]
			\begin{semilogxaxis}[height=\aspectratio*\plotimsized\textwidth,width=\plotimsized\textwidth,legend columns=1,legend cell align={left},legend style={draw=\legendboxdraw,fill=\legendboxfill,at={(0,1)},anchor=north west},log basis x={2},x dir=reverse,y tick label style={/pgf/number format/.cd,fixed,precision=2,/tikz/.cd},
			xlabel={$\delta$}, ylabel={$\costf$},label style={font=\normalsize},tick label style={font=\normalsize},legend image post style={scale=\legendmarkscale},legend style={nodes={scale=\legendfontscale, transform shape}}]
			\addplot[color=\colorone,solid,line width=\plotlinewidth pt,mark=none,mark size=\plotmarksizeu pt] table [x=delta,y=mRKC_cF,col sep=comma] 
			{data/narrow_channel/results.csv};\addlegendentry{Relative cost $\costf$ of $\ff$}
			\end{semilogxaxis}
			\end{tikzpicture}
			\caption{Relative evaluation cost of $\ff(y)$ w.r.t. $\ff+\fs$ as a function of $\delta$.}
			\label{fig:eff2_cf}
		\end{subfigure}%
	\caption{Narrow channel. Speed-up, error, spectral radii and stages number w.r.t. channel width $\delta$.}
	\label{fig:eff2_results}
\end{figure}
Finally, we compare mRKC against the second-order accurate version of the RKC scheme (RKC2) from \cite{SSV98,VHS90} and the implicit Euler (IE) method. To do so, we consider the two channel widths $\delta=1/2^6,1/2^{12}$ and solve \cref{eq:par3} with the RKC, RKC2, mRKC and IE schemes. For both values of $\delta$ we display in \cref{fig:WvsP} the computational times against the final error, with varying step size $\tau=T/2^j$, $j=0,\ldots,16$ and $T=0.1$. Since the exact solution is unknown, the final error is computed against a reference solution obtained from the second-order RKC2 scheme with step size $\tau=10^{-6}$. For the wider channel with $\delta=1/2^6$, we observe in \cref{fig:WvsPa} that the mRKC scheme is more efficient than RKC2 for most of the step sizes, but becomes less efficient at higher accuracy. The IE and the mRKC method are about equally efficient. In contrast, for the narrow channel with $\delta=1/2^{12}$, that is, in a situation of even more severe stiffness, we observe in \cref{fig:WvsPb} that mRKC is always much faster than RKC2: the speed-up ranges from ten to seventy times faster, depending on the step size imposed by the desired accuracy. In this case of extreme stiffness, the IE method always remains slightly faster than mRKC; clearly, the IE method is also particularly efficient here thanks to the symmetry and linearity of the discrete Laplacian.
\begin{figure}
		\begin{subfigure}[t]{\subfigsize\textwidth}
			\centering
			\begin{tikzpicture}[scale=\plotimscale]
				\begin{loglogaxis}[height=\aspectratio*\plotimsized\textwidth,width=\plotimsized\textwidth,legend columns=1,legend cell align={left},legend style={draw=\legendboxdraw,fill=\legendboxfill,at={(1,0)},anchor=south east},log basis x={10},log basis y={10},x dir=reverse, 
					ylabel style={yshift=0pt},yticklabel style={xshift=2pt},
					xlabel={Relative error in $L^2(\Omega_\delta)$ norm}, ylabel={CPU time [sec.]},label style={font=\normalsize},tick label style={font=\normalsize},legend image post style={scale=\legendmarkscale},legend style={nodes={scale=\legendfontscale, transform shape}}]
					\addplot[color=\classicalcolor,solid,line width=\plotlinewidth pt,mark=\classicalmark,mark size=\plotmarksized pt] table [x=relerrl2,y=rkc_step_inc_sub,col sep=comma] 
					{data/narrow_channel/RKC1_k_7.csv};\addlegendentry{RKC}
					\addplot[color=\colortwo,solid,line width=\plotlinewidth pt,mark=\marktwo,mark size=\plotmarksizeu pt] table [x=relerrl2,y=rkc_step_inc_sub,col sep=comma] 
					{data/narrow_channel/RKC2_k_7.csv};\addlegendentry{RKC2}
					\addplot[color=\localcolor,solid,line width=\plotlinewidth pt,mark=\localmark,mark size=\plotmarksizeu pt] table [x=relerrl2,y=rkc_step_inc_sub,col sep=comma] 
					{data/narrow_channel/mRKC_k_7.csv};\addlegendentry{mRKC}
					\addplot[color=\colorfour,solid,line width=\plotlinewidth pt,mark=\markfour,mark size=\plotmarksizeu pt] table [x=relerrl2,y=linsystemtot,col sep=comma] 
					{data/narrow_channel/IE_k_7.csv};\addlegendentry{IE}
				\end{loglogaxis}
			\end{tikzpicture}
			\caption{$\delta=1/2^6$}
			\label{fig:WvsPa}
		\end{subfigure}\hfill%
		\begin{subfigure}[t]{\subfigsize\textwidth}
			\centering
			\begin{tikzpicture}[scale=\plotimscale]
				\begin{loglogaxis}[height=\aspectratio*\plotimsized\textwidth,width=\plotimsized\textwidth,legend columns=1,legend cell align={left},legend style={draw=\legendboxdraw,fill=\legendboxfill,at={(1,0)},anchor=south east},log basis x={10},log basis y={10},x dir=reverse, 
					ylabel style={yshift=0pt},yticklabel style={xshift=2pt},
					xlabel={Relative error in $L^2(\Omega_\delta)$ norm}, ylabel={CPU time [sec.]},label style={font=\normalsize},tick label style={font=\normalsize},legend image post style={scale=\legendmarkscale},legend style={nodes={scale=\legendfontscale, transform shape}}]
					\addplot[color=\classicalcolor,solid,line width=\plotlinewidth pt,mark=\classicalmark,mark size=\plotmarksized pt] table [x=relerrl2,y=rkc_step_inc_sub,col sep=comma] 
					{data/narrow_channel/RKC1_k_13.csv};\addlegendentry{RKC}
					\addplot[color=\colortwo,solid,line width=\plotlinewidth pt,mark=\marktwo,mark size=\plotmarksizeu pt] table [x=relerrl2,y=rkc_step_inc_sub,col sep=comma] 
					{data/narrow_channel/RKC2_k_13.csv};\addlegendentry{RKC2}
					\addplot[color=\localcolor,solid,line width=\plotlinewidth pt,mark=\localmark,mark size=\plotmarksizeu pt] table [x=relerrl2,y=rkc_step_inc_sub,col sep=comma] 
					{data/narrow_channel/mRKC_k_13.csv};\addlegendentry{mRKC}
					\addplot[color=\colorfour,solid,line width=\plotlinewidth pt,mark=\markfour,mark size=\plotmarksizeu pt] table [x=relerrl2,y=linsystemtot,col sep=comma] 
					{data/narrow_channel/IE_k_13.csv};\addlegendentry{IE}
				\end{loglogaxis}
			\end{tikzpicture}
			\caption{$\delta=1/2^{12}$}
			\label{fig:WvsPb}
		\end{subfigure}%
	\caption{Narrow channel. Work vs. accuracy diagram of the RKC, RKC2, mRKC and implicit Euler (IE) scheme for two channel widths $\delta$. Symbols represent step sizes $\tau=T/2^j$, for $j=0,\ldots,16$.}
	\label{fig:WvsP}
\end{figure}
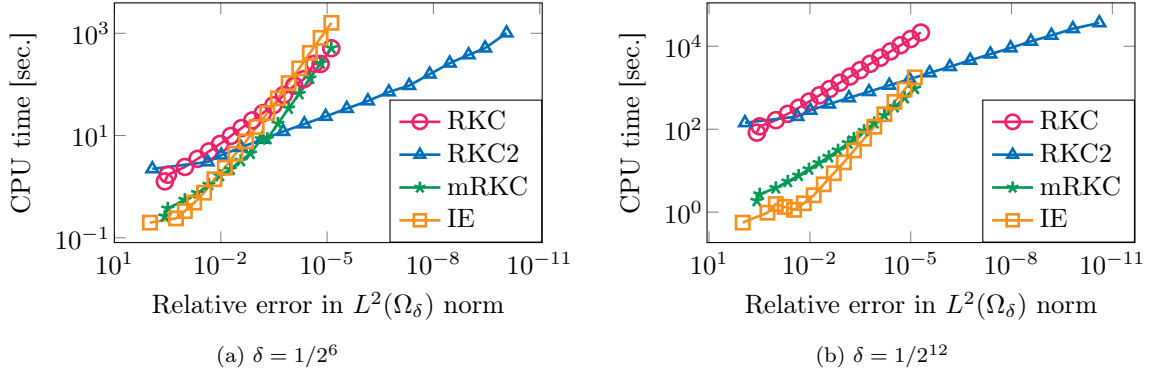

\subsection{Integro-differential equation}\label{sec:intdiff}
To compare the mRKC and implicit Euler scheme on a nonsymmetric and truly nonlinear problem, we now consider the one-dimensional integro-differential problem from \cite{VaV92}, also considered in \cite{AbV13,Zbi11}:
\begin{equation}\label{eq:intdiff}
\begin{aligned}
	\dt u(t,x)  &=\Delta u(t,x)-\sigma  \int_0^1 \frac{u(t,s)^4}{(1+|x-s|)^2}\dif s \qquad && (t,x)\in [0,1]\times [0,1],\\
	u(0,x)&= \cos(x\pi/2)^2 && x\in [0,1],\\
	u(t,0) &= 1-\sqrt{t}/2 && t\in[0,1],\\
	\partial_x u(t,1) &= 0 && t\in[0,1],
\end{aligned}
\end{equation}
with $\sigma=0.01$. Problem \cref{eq:intdiff} models an idealized temperature profile of air near the ground. We discretize \cref{eq:intdiff} in space on a uniform grid of $N$ cells using central finite differences  for the Laplacian and the composite trapezoidal rule for the integral term. For the mRKC scheme, we assign the Laplacian to $\ff$ and the integral term to $\fs$. For the implicit Euler method, the Jacobian (evaluated analytically) and its LU decomposition are computed only once per time step following \cite[IV.8]{HaW02}.

For $N=100, 3200$ (varying problem size and degree of stiffness), we apply the RKC, RKC2, mRKC and IE schemes for $\tau=1/2^{j}$ with $j=2,\ldots,12$. For each run, we monitor the computational time and the $L^2([0,1])$ error against a reference solution: the resulting efficiency graphs are shown in \cref{fig:eff_intdiff}.  For $N=100$, IE method is always faster than RKC, which remains the most expensive method. For large $\tau$, IE is comparable to mRKC, but as $\tau$ decreases, the mRKC scheme becomes significantly more efficient. The second-order scheme RKC2 becomes faster than mRKC only at high accuracy. For the stiffer case with $N=3200$, the mRKC scheme is the fastest method, as its efficiency is only marginally affected by the increased stiffness, while the cost of the direct solver in the Newton iteration clearly starts to dominate the overall cost of IE. For those parameter settings, the RKC and RKC2 methods were overly expensive and could not be run to completion.
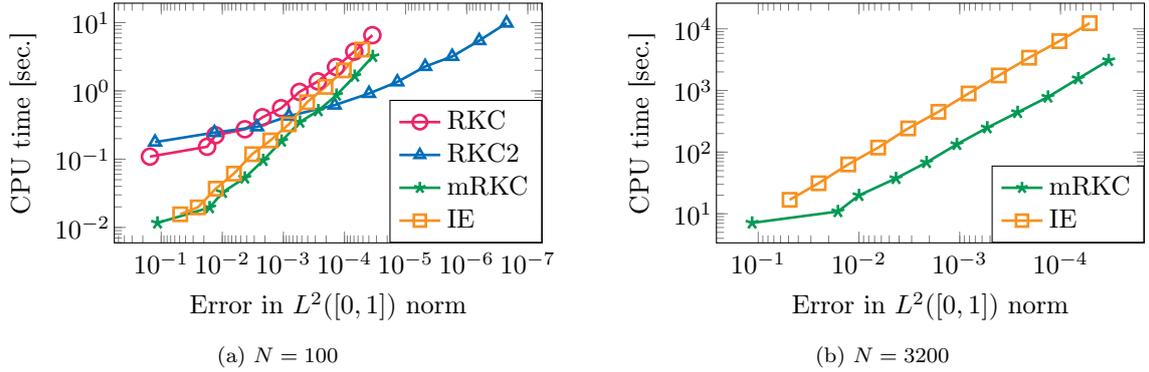
\begin{figure}
	\begin{subfigure}[t]{\subfigsize\textwidth}
		\centering
		\begin{tikzpicture}[scale=\plotimscale]
			\begin{loglogaxis}[height=\aspectratio*\plotimsized\textwidth,width=\plotimsized\textwidth,legend columns=1,legend cell align={left},legend style={draw=\legendboxdraw,fill=\legendboxfill,at={(1,0)},anchor=south east},log basis x={10},log basis y={10},x dir=reverse, 
				ylabel style={yshift=0pt},yticklabel style={xshift=2pt},
				xlabel={Error in $L^2([0,1])$ norm}, ylabel={CPU time [sec.]},label style={font=\normalsize},tick label style={font=\normalsize},legend image post style={scale=\legendmarkscale},legend style={nodes={scale=\legendfontscale, transform shape}}]
				\addplot[color=\classicalcolor,solid,line width=\plotlinewidth pt,mark=\classicalmark,mark size=\plotmarksized pt] table [x=err,y=cpu,col sep=comma] 
				{data/integrodifferential/RKC1_N_100.csv};\addlegendentry{RKC}
				\addplot[color=\colortwo,solid,line width=\plotlinewidth pt,mark=\marktwo,mark size=\plotmarksizeu pt] table [x=err,y=cpu,col sep=comma] 
				{data/integrodifferential/RKC2_N_100.csv};\addlegendentry{RKC2}
				\addplot[color=\localcolor,solid,line width=\plotlinewidth pt,mark=\localmark,mark size=\plotmarksizeu pt] table [x=err,y=cpu,col sep=comma] 
				{data/integrodifferential/mRKC_PAR_N_100.csv};\addlegendentry{mRKC}
				\addplot[color=\colorfour,solid,line width=\plotlinewidth pt,mark=\markfour,mark size=\plotmarksizeu pt] table [x=err,y=cpu,col sep=comma] 
				{data/integrodifferential/IE_LU_N_100.csv};\addlegendentry{IE}
			\end{loglogaxis}
		\end{tikzpicture}
		\caption{$N=100$}
		\label{fig:eff_intdiffa}
	\end{subfigure}\hfill%
	\begin{subfigure}[t]{\subfigsize\textwidth}
		\centering
		\begin{tikzpicture}[scale=\plotimscale]
			\begin{loglogaxis}[height=\aspectratio*\plotimsized\textwidth,width=\plotimsized\textwidth,legend columns=1,legend cell align={left},legend style={draw=\legendboxdraw,fill=\legendboxfill,at={(1,0)},anchor=south east},log basis x={10},log basis y={10},x dir=reverse, 
				ylabel style={yshift=0pt},yticklabel style={xshift=2pt},
				xlabel={Error in $L^2([0,1])$ norm}, ylabel={CPU time [sec.]},label style={font=\normalsize},tick label style={font=\normalsize},legend image post style={scale=\legendmarkscale},legend style={nodes={scale=\legendfontscale, transform shape}}]
				\addplot[color=\localcolor,solid,line width=\plotlinewidth pt,mark=\localmark,mark size=\plotmarksizeu pt] table [x=err,y=cpu,col sep=comma] {data/integrodifferential/mRKC_PAR_N_3200.csv};\addlegendentry{mRKC}
				\addplot[color=\colorfour,solid,line width=\plotlinewidth pt,mark=\markfour,mark size=\plotmarksizeu pt] table [x=err,y=cpu,col sep=comma] {data/integrodifferential/IE_LU_N_3200.csv};\addlegendentry{IE}
			\end{loglogaxis}
		\end{tikzpicture}
		\caption{$N=3200$}
		\label{fig:eff_intdiffb}
	\end{subfigure}%
	\caption{Integro-differential problem. Work vs. accuracy diagram of the RKC, RKC2, mRKC and implicit Euler (IE) scheme for two mesh sizes $1/N$. Symbols represent step sizes $\tau=1/2^j$, for $j=2,\ldots,12$.}
	\label{fig:eff_intdiff}
\end{figure}

\subsection{Reaction-convection-diffusion problem}\label{sec:exp_stab}
In \cref{sec:stabanalysis} we proved that the stability conditions of the mRKC method are the same for the $2\times 2$ model problem \eqref{eq:twodim} and for the scalar multirate test equation \eqref{eq:mtesteq}. The splitting of the discrete Laplace operator in \eqref{eq:splitpar} in fact is similar to that in \eqref{eq:tauAfAc1} for the $2\times 2$ model problem. Thus, one could expect that the stability conditions \eqref{eq:defsmeta} are also necessary for more general parabolic problems. However, spatial discretizations of parabolic problems are much more complex than \eqref{eq:twodim}. Here we shall demonstrate via numerical experiment that the weaker stability conditions \eqref{eq:defsmetaweak} in fact are also necessary and sufficient for general parabolic reaction-convection-diffusion problems, such as
\begin{equation}\label{eq:parmod}
\begin{aligned}
\dt u-\nabla\cdot(K\nabla u)+\bm{\beta}\cdot\nabla u +\mu u &=g \qquad &&\text{in }\Omega\times [0,T],\\
u&=0 &&\text{in } \partial\Omega \times [0,T],\\
u&=u_0 &&\text{in } \Omega\times\{0\}.
\end{aligned}
\end{equation}
These experiments also illustrate that the mRKC method indeed requires no scale separation.

We now consider three distinct parameter regimes. First, we let $\Omega=[0,2]\times [0,1]$, $K=I_{2\times 2}$, $\bm{\beta}=\bm{0}$ and $\mu=0$. Inside $\Omega$, we build a $16\times 8$ uniform mesh and refine twice the elements inside of $\OF=(1,2)\times (0,1)$ (see \cref{fig:stab_test_dom}). Again, we use DG-FE 
for the spatial discretization, which yields the two matrices $A_F$ and $A_S$, as described in \cref{sec:exp_conv2}. Next, we set $\tau=1$, $s,m,\eta$ as in \eqref{eq:defsmeta} and $A_\eta$ as in \eqref{eq:defAeta}. One step of the mRKC scheme is given by $y_1=\Pe_s(\tau A_\eta)y_0$. We recall that a necessary condition for stability of the scheme (at least for linear problems) is $\tau \rho_\eta\leq \beta s^2$, where $\rhoe$ is the spectral radius of $A_\eta$.

Let $\overline\beta$ be as in \eqref{eq:defsmetaweak}, $\overline\eta\in [0,\eta]$, $\overline m$ such that $\overline\eta\rhof\leq\overline\beta\overline m^2$,
\begin{equation}
\overline A_\eta=\Ps_{\overline m}(\overline \eta A_F)A
\end{equation}
and $\overline \rho_\eta$ be the spectral radius of $\overline A_\eta$. We wish to study for which $\overline\eta$ it holds $\tau\overline{\rho}_\eta\leq\beta s^2$. In \cref{fig:rho0}, we display $\tau \overline\rho_\eta$ for $\overline\eta\in (0,\eta)$ with respect to $w(\overline{\eta})=-\overline\eta\beta s^2/\tau$: for $|w(\overline \eta)|\geq 2$, it holds $\tau\overline\rho_\eta\leq\beta s^2$ and thus the scheme is stable. Observe that $|w(\overline \eta)|\geq 2$ is equivalent to $\overline \eta \geq 2\tau/(\beta s^2)$, as in \eqref{eq:defsmetaweak}.
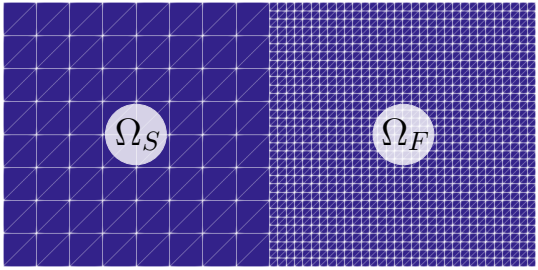
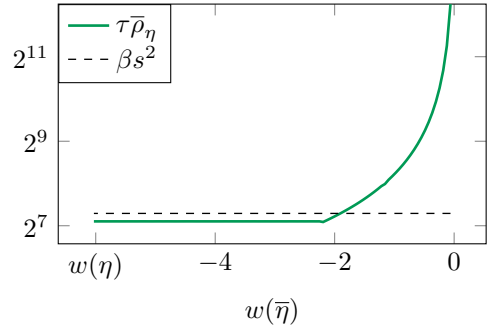
\begin{figure}
		\begin{subfigure}[t]{\subfigsize\textwidth}
			\centering
			\begin{tikzpicture}
			\node at (0,0) {\includegraphics[width=0.96\textwidth]{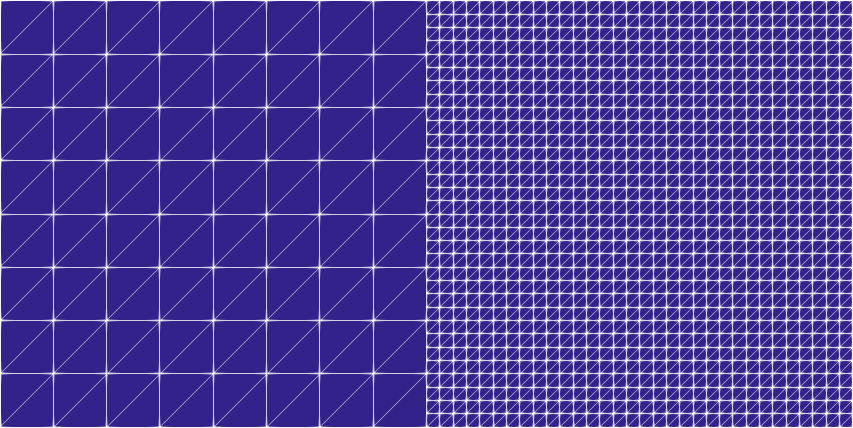}};
			\draw[white,fill=white,opacity=0.8] (1.76,0.) circle (0.4cm);
			\draw[white,fill=white,opacity=0.8] (-1.76,0.) circle (0.4cm);
			\node at (1.79,0.) {{\Large $\Omega_F$}};
			\node at (-1.73,0.) {{\Large $\Omega_S$}};
			\draw[white,fill=white,opacity=0.0] (0,-2.4) circle (0.4cm);
			\end{tikzpicture}
			\caption{Domains $\Omega_S$ and $\Omega_F$}
			\label{fig:stab_test_dom}
		\end{subfigure}\hfill%
		\begin{subfigure}[t]{\subfigsize\textwidth}
			\centering
			\begin{tikzpicture}[scale=\plotimscale]
			\begin{semilogyaxis}[height=\aspectratio*\plotimsized\textwidth,width=\plotimsized\textwidth,legend columns=1,legend style={draw=\legendboxdraw,fill=\legendboxfill,at={(0,1)},anchor=north west},log basis y={2},scaled x ticks=false,
			xtick={-6,-4,-2,0}, xticklabels={$w(\eta)$, $-4$, $-2$, $0$}, ymax=4850,
			xlabel={$w(\overline\eta)$},label style={font=\normalsize},tick label style={font=\normalsize},legend image post style={scale=\legendmarkscale},legend style={nodes={scale=\legendfontscale, transform shape}}]
			\addplot[color=\colorone,solid,line width=\plotlinewidth pt,mark=none,mark size=\plotmarksizeu pt] table [x=weta,y=rho,col sep=comma] 
			{data/num_stab/rho_0.csv};\addlegendentry{$\tau \overline \rho_\eta$}
			\addplot[color=black,dashed,line width=\plotdashedlinewidth pt,mark=none,mark size=\plotmarksizeu pt] table [x=weta,y=bssq, col sep=comma] 
			{data/num_stab/rho_0.csv};\addlegendentry{$\beta s^2$}
			\end{semilogyaxis}
			\end{tikzpicture}
			\caption{$\tau\overline\rho_\eta$ vs. $w(\overline\eta)$, first problem.}
			\label{fig:rho0}
		\end{subfigure}%
	\caption{Reaction-convection-diffusion problem. Mesh and spectral radius $\tau \overline{\rho}_\eta$ of the first problem setting.}
	\label{fig:stab_test}
\end{figure}
Since the smallest (in magnitude) nonzero eigenvalues of the discrete Laplacians, $A_F$ and $A_S$,
do not depend on the mesh size, but only depends on the size of the domain, 
they essentially coincide; hence, this problem exhibits no scale separation assumption. Nevertheless
the mRKC scheme remains stable, as expected from theory.

Finally, we consider two additional cases that further corroborate the previous findings. First, we set $\Omega=[0,1]\times [0,1]$, $K=I_{2\times 2}$, $\bm{\beta}=(1,1)^\top$ and $\mu=1$.  In  $\Omega$,
we build a $8\times 8$ uniform mesh and refine three times the elements inside the small inner square $\OF=(1/4,3/4)\times (1/4,3/4)$. In \cref{fig:rho1}, we show again $\tau \overline\rho_\eta$ for $\overline\eta\in (0,\eta)$ with respect to $w(\overline\eta)$: for $|w(\overline \eta)|\geq 2$, $\tau\overline\rho_\eta\leq\beta s^2$ holds. Next, we use a uniform $32\times 32$ mesh in $\Omega=[0,1]\times [0,1]$ which is refined twice in the lower left corner $\OF=(0,1/32)\times (0,1/32)$. We also set $\bm{\beta}=0$, $\mu=0$, $K(\bx)=1$ for $x_1\geq x_2$ and $K(\bx)=0.1$ elsewhere. The results, shown in \cref{fig:rho2}, again confirm the stability of the mRKC with parameters chosen according to \eqref{eq:defsmetaweak}.
\begin{figure}
		\begin{subfigure}[t]{\subfigsize\textwidth}
			\centering
			\begin{tikzpicture}[scale=\plotimscale]
			\begin{semilogyaxis}[height=\aspectratio*\plotimsized\textwidth,width=\plotimsized\textwidth,legend columns=1,legend style={draw=\legendboxdraw,fill=\legendboxfill,at={(0,1)},anchor=north west},log basis y={2},ymax=32768,scaled x ticks=false,
			xtick={-6,-4,-2,0}, xticklabels={$w(\eta)$, $-4$, $-2$, $0$}, 
			xlabel={$w(\overline\eta)$},label style={font=\normalsize},tick label style={font=\normalsize},legend image post style={scale=\legendmarkscale},legend style={nodes={scale=\legendfontscale, transform shape}}]
			\addplot[color=\colorone,solid,line width=\plotlinewidth pt,mark=none,mark size=\plotmarksizeu pt] table [x=weta,y=rho,col sep=comma] 
			{data/num_stab/rho_1.csv};\addlegendentry{$\tau \overline \rho_\eta$}
			\addplot[color=black,dashed,line width=\plotdashedlinewidth pt,mark=none,mark size=\plotmarksizeu pt] table [x=weta,y=bssq, col sep=comma] 
			{data/num_stab/rho_1.csv};\addlegendentry{$\beta s^2$}
			\end{semilogyaxis}
			\end{tikzpicture}
			\caption{$\tau\overline\rho_\eta$ vs. $w(\overline\eta)$, second problem.}
			\label{fig:rho1}
		\end{subfigure}\hfill%
		\begin{subfigure}[t]{\subfigsize\textwidth}
			\centering
			\begin{tikzpicture}[scale=\plotimscale]
			\begin{semilogyaxis}[height=\aspectratio*\plotimsized\textwidth,width=\plotimsized\textwidth,legend columns=1,legend style={draw=\legendboxdraw,fill=\legendboxfill,at={(0,1)},anchor=north west},log basis y={2},ymax=524288,scaled x ticks=false,
			xtick={-6,-4,-2,0}, xticklabels={$w(\eta)$, $-4$, $-2$, $0$}, 
			xlabel={$w(\overline\eta)$},label style={font=\normalsize},tick label style={font=\normalsize},legend image post style={scale=\legendmarkscale},legend style={nodes={scale=\legendfontscale, transform shape}}]
			\addplot[color=\colorone,solid,line width=\plotlinewidth pt,mark=none,mark size=\plotmarksizeu pt] table [x=weta,y=rho,col sep=comma] 
			{data/num_stab/rho_2.csv};\addlegendentry{$\tau \overline \rho_\eta$}
			\addplot[color=black,dashed,line width=\plotdashedlinewidth pt,mark=none,mark size=\plotmarksizeu pt] table [x=weta,y=bssq, col sep=comma] 
			{data/num_stab/rho_2.csv};\addlegendentry{$\beta s^2$}
			\end{semilogyaxis}
			\end{tikzpicture}
			\caption{$\tau\overline\rho_\eta$ vs. $w(\overline\eta)$, third problem.}
			\label{fig:rho2}
		\end{subfigure}%
	\caption{Stability experiment. Illustration of $\tau\overline\rho_\eta$ versus $w(\overline\eta)$, second and third problem setting.}
	\label{fig:num_stab}
\end{figure}
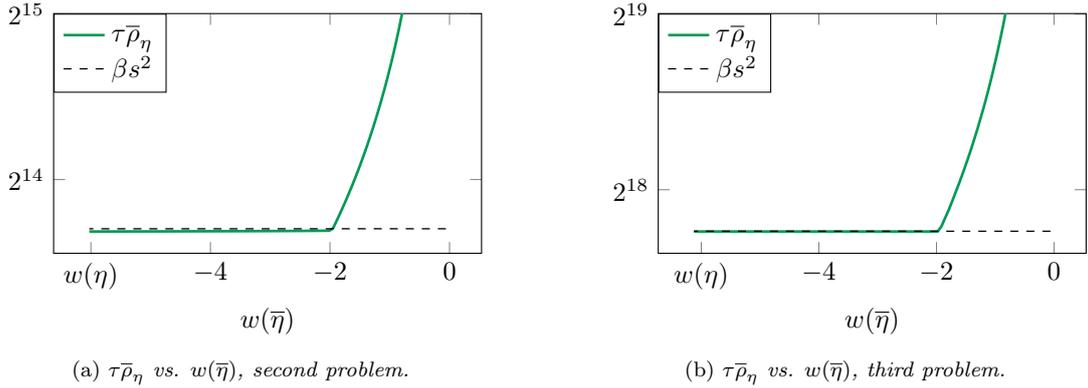

\section{Conclusion}
Starting from the stiff differential equation $y'=\fs(y) + \ff(y)$, where $\ff$ represents a few severely stiff, but cheap, ``fast'' components, we have proposed a modified equation \cref{eq:odemod} whose stiffness no longer
depends on $\ff$.  It involves an averaged force $\fe$ of $f = \fs+\ff$, evaluated by solving the stiff, but cheap, auxiliary problem \cref{eq:defu} over short time and thus forms the basis of the following multirate strategy:
Solve \cref{eq:odemod} with an explicit numerical method whose stability conditions are determined by the mildly stiff, but expensive, ``slow'' components $\fs$,
while solving \cref{eq:defu} with a separate explicit method whose stability conditions are determined by $\ff$, whenever an evaluation of $\fe$ is needed.
In Theorems \ref{thm:fecontr} and \ref{thm:erryeta}, we have proved that the modified equation \eqref{eq:odemod} approximates the original problem to first-order accuracy while preserving its contractivity properties. The stability analysis of the multirate test equation \cref{eq:modtesteq} underpins the reduced stiffness of the modified equation, which no longer depends on the fastest components $\ff$ for $\eta$ sufficiently large -- see \Cref{thm:stifflinscamod}.

By discretizing \eqref{eq:odemod} with an $s$-stage (explicit) Runge-Kutta-Chebyshev (RKC) method while 
evaluating $\fe$ with one step of a separate $m$-stage RKC method, we have devised a new
multirate RKC method. The resulting mRKC method, given by \eqref{eq:defsmeta}--\eqref{eq:defbue},
is fully explicit, stable, and first-order accurate, as proved in \cref{thm:stab_mrkc,thm:firstorder}, without
the need for interpolation or extrapolation of missing stage values. Thanks to the reduced stiffness in  \eqref{eq:odemod}, the number of expensive $\fs$ evaluations is greatly reduced and independent of
the severe stiffness induced by just a few degrees of freedom in $\ff$, without any assumption about scale separation.

For semi-discrete parabolic problems, where $\fs$ and $\ff$ correspond to discretized diffusion operators
in the coarse and locally refined regions of the mesh, respectively,
the mRKC method permits to overcome the crippling effect on explicit time integrators
due to a few tiny elements or grid cells. 
In particular, for diffusion dominated problems in complex geometry, the mRKC method is up to forty times faster than a standard first-order RKC method; it is also up to seventy times cheaper than a second-order RKC2 \cite{SSV98,VHS90} method for moderately high error tolerances.
Thus, the  mRKC method recovers 
the well-known efficiency of RKC methods for large-scale, possibly nonlinear, parabolic problems
without sacrificing explicitness, even in the presence of local mesh refinement. When compared to the implicit Euler method, the mRKC scheme's performance depends on the degree of stiffness, problem size and nonlinearity. In all our numerical experiments, mRKC performed similarly, or even better, than IE, without the need for solving any linear systems. Moreover, our numerical experiments suggest that with increasing problem size, the efficiency of IE rapidly decreases, whereas mRKC remains only marginally affected.

The multirate strategy introduced here also paves the way for higher order extensions and for developing explicit stabilized multirate methods for stiff stochastic differential equations \cite{AbR22b}.

\section*{Acknowledgments}
This research is partially supported by the Swiss National Science Foundation, grant no. 20020\_172710. The second author thanks the EPFL for the opportunity to perform this research there during his sabbatical leave. 

\appendix
\section{Proofs of lemmas} \label{app:proofs}

In this section we prove \cref{lemma:disc_stab,lemma:ddRincr}, needed in the proof of \cref{thm:stab_mrkc}.

\begin{proof}[Proof of \cref{lemma:disc_stab}.]
	For the only if part we follow the lines of the proof of \cref{lemma:cont_stab} and find that $\Ps_m'(0)|w|\geq 1$ is a necessary condition. The identity $\Ps_m'(0)=P_m''(0)/2$ follows from the definition of $\Ps_m(z)$ in \cref{eq:PhimTfrac}.
	
	Now, let us suppose $2/P_m''(0)\leq |w|$ and show $\beta_\varepsilon(z)=\Ps_m(z)(z+w)\geq w$, where $\varepsilon$ is the damping. For $z=0$ it is clear, independently of $\varepsilon$. We will show $\beta_0(z)>w$ for all $z<0$, since $\beta_\varepsilon(z)$ depends continuously on $\varepsilon$ there exists $\varepsilon_m>0$ such that $\beta_\varepsilon(z)\geq w$ for all $\varepsilon\leq \varepsilon_m$. We have
	\begin{align}
	\beta_0'(z)&=\frac{P_m'(z)}{z}(z+w)-\frac{P_m(z)-1}{z^2}w,\\
	\beta_0''(z)&=\frac{P_m''(z)}{z}(z+w)-2\frac{P_m'(z)}{z^2}w+2\frac{P_m(z)-1}{z^3}w
	\end{align}
	and since $\beta_0(z)$ is a polynomial of degree $m$ then $\beta_0'(z)$ has at most $m-1$ zeros. We are going to locate the zeros $z_{m-1}<\dots<z_3<z_2$ of $\beta_0'(z)$. Then we will use the fact that $\beta_0'(z)$ has at most one zero on the right of $z_2$. In order to help the understanding of the proof we plot $\beta_0'(z)$ in \cref{fig:db} for two values of $m$.
	\begin{figure}
			\begin{subfigure}[t]{\subfigsize\textwidth}
				\centering
				\begin{tikzpicture}[scale=\plotimscale]
				\begin{axis}[height=\aspectratio*\plotimsized\textwidth,width=\plotimsized\textwidth,ymin=-1.1,ymax=1.1,legend columns=1,legend style={draw=\legendboxdraw,fill=\legendboxfill,at={(0,1)},anchor=north west},
				xlabel={$z$}, ylabel={},label style={font=\normalsize},tick label style={font=\normalsize},legend image post style={scale=\legendmarkscale},legend style={nodes={scale=\legendfontscale, transform shape}},grid=major]
				\addplot[color=\colorone,line width=\plotlinewidth pt,mark=none] table [x=z,y=db,col sep=comma] 
				{data/text/db_m_8.csv};\addlegendentry{$\beta_0'(z)$ for $m=8$.}
				\addplot[color=\colorthree,line width=\plotlinewidth pt,only marks,mark=\markthree,mark size=\plotmarksizeu pt] table [x=z,y=dbz,col sep=comma] 
				{data/text/db_m_8_zeros.csv};
				\draw (-12,0.15) node [fill=white,circle,inner sep=0pt,minimum size=1pt]{$z_2$};
				\draw (-57,0.15) node [fill=white,circle,inner sep=0pt,minimum size=1pt]{$z_4$};
				\draw (-102.25,0.15) node [fill=white,circle,inner sep=0pt,minimum size=1pt]{$z_6$};
				\end{axis}
				\end{tikzpicture}
			\end{subfigure}\hfill%
			\begin{subfigure}[t]{\subfigsize\textwidth}
				\centering
				\begin{tikzpicture}[scale=\plotimscale]
				\begin{axis}[height=\aspectratio*\plotimsized\textwidth,width=\plotimsized\textwidth,ymin=-1.1,ymax=1.1,legend columns=1,legend style={draw=\legendboxdraw,fill=\legendboxfill,at={(1,1)},anchor=north east},
				xlabel={$z$}, ylabel={},label style={font=\normalsize},tick label style={font=\normalsize},legend image post style={scale=\legendmarkscale},legend style={nodes={scale=\legendfontscale, transform shape}},grid=major]
				\addplot[color=\colorone,line width=\plotlinewidth pt,mark=none] table [x=z,y=db,col sep=comma] 
				{data/text/db_m_9.csv};\addlegendentry{$\beta_0'(z)$ for $m=9$.}
				\addplot[color=\colorthree,line width=\plotlinewidth pt,only marks,mark=\markthree,mark size=\plotmarksizeu pt] table [x=z,y=dbz,col sep=comma] 
				{data/text/db_m_9_zeros.csv};
				\draw (-12,0.15) node [fill=white,circle,inner sep=0pt,minimum size=1pt]{$z_2$};
				\draw (-60,0.15) node [fill=white,circle,inner sep=0pt,minimum size=1pt]{$z_4$};
				\draw (-114.5,0.15) node [fill=white,circle,inner sep=0pt,minimum size=1pt]{$z_6$};
				\draw (-150,0.15) node [fill=white,circle,inner sep=0pt,minimum size=1pt]{$z_8$};
				\end{axis}
				\end{tikzpicture}
			\end{subfigure}%
		\caption{Plot of $\beta_0'(z)$ for $m=8,9$ and $|w|=2/P_m''(0)$.}
		\label{fig:db}
	\end{figure}
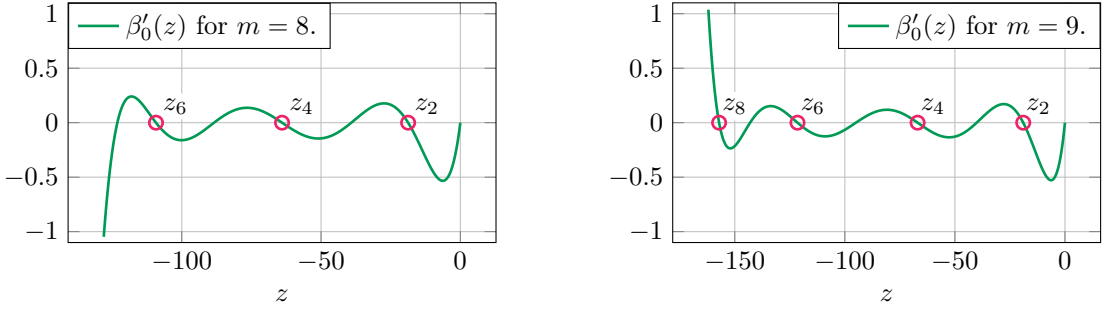
	
	Since $P_m(z)=T_m(1+z/m^2)$ and $T_m(\cos(\theta))=\cos(m\theta)$, choosing $z_{2k}$ such that
	\begin{align}
	1+\frac{z_{2k}}{m^2}&= \cos(\theta_{2k}) & \mbox{with} && \theta_{2k}&=\frac{2k\pi}{m}
	\end{align}
	it yields
	\begin{equation}
	\beta_0'(z_{2k})=0  \quad\mbox{for}\quad  2k=2,4,\ldots,2\lfloor\frac{m-1}{2}\rfloor.
	\end{equation}
	Since $z_{2k}$ is a local maximum of $P_m(z)$ then $P_m''(z_{2k})<0$ and $\beta_0''(z_{2k})<0$. In a neighborhood of $z_{2k}$ we have
	\begin{equation}
	\beta_0'(z)=\beta_0''(z_{2k})(z-z_{2k})+\bigo{(z-z_{2k})^2},
	\end{equation}
	hence for $\delta>0$ small and $2k=2,4,\ldots,2(\lfloor\frac{m-1}{2}\rfloor-1)$ we have $\beta_0'(z_{2k+2}+\delta)<0$ and $\beta_0'(z_{2k}-\delta)>0$, implying that there exists $z_{2k+1}\in [z_{2k+2},z_{2k}]$ such that $\beta_0'(z_{2k+1})=0$. If $m$ is odd then $2\lfloor\frac{m-1}{2}\rfloor=m-1$ and we located the zeros $z_j$ for $j=2,3,\ldots,m-1$. If $m$ is even then $2\lfloor\frac{m-1}{2}\rfloor=m-2$, but $P_m(-2m^2)=1$ and $P_m'(-2m^2)=-1$ and hence $\beta_0'(-2m^2)<0$. Thus, since $\beta_0'(z_{m-2}-\delta)>0$ there exists $z_{m-1}\in ]-2m^2,z_{m-2}[$ such that $\beta_0'(z_{m-1})=0$. Finally, we located $z_j$ for $j=2,3,\ldots,m-1$ for $m$ even and odd. We will show $\beta_0(z)>w$ for $z\in [z_2,0[$ and then for $z\in[-2m^2,z_2]$.
	
	Let $z\in [z_2,0[$, if $z=z_1$ then $\beta_0'(z)=0$ and else $\beta_0'(z)<0$. Indeed, for $z$ close to zero we have
	\begin{equation}
	\beta_0'(z)=1+\frac{1}{2}P_m''(0)w+ (P_m''(0)+\frac{1}{3}P_m'''(0)w)z+\bigo{z^2}. 
	\end{equation}
	If $2/P_m''(0)<|w|$ then $1+\frac{1}{2}P_m''(0)w<0$ and $\beta_0'(z)<0$ in the neighborhood of zero. If $2/P_m''(0)=|w|$ then 
	\begin{equation}
	\beta_0'(z)=\left(P_m''(0)-\frac{2}{3}\frac{P_m'''(0)}{P_m''(0)}\right)z+\bigo{z^2} = \frac{m^2+1}{5m^2}z+\bigo{z^2},
	\end{equation}
	and $\beta_0'(z)<0$ in the neighborhood of zero as well. If there exists $\bar z\in ]z_2,0[$ such that $\beta_0'(\bar z)>0$ we can take $\delta>0$ small enough to have $\bar z\in ]z_2+\delta,-\delta[$ and $\beta_0'(z_2+\delta)<0$ and $\beta_0'(-\delta)<0$. Hence, $\beta_0'$ would change sign twice in the interval $]z_2+\delta,-\delta[$, which is impossible since $\beta_0'$ has at most one zero on the right of $z_2$. Hence, $\beta_0'(z)<0$ for all $z\in [z_2,0[$ except at most one point, since $\beta_0(0)=w$ it follows $\beta_0(z)>w$ for all $z\in [z_2,0[$. 
	
	We consider now $z\leq z_2$. Using $1-\cos(\theta)=2\sin(\theta/2)^2$ it holds $z_2=-2m^2\sin(\pi/m)$ and 
	\begin{align}
	\beta_0(z)&=\Ps_m(z)(z+w)=\frac{P_m(z)-1}{z}(z+w)\geq -2\frac{z+w}{z}\geq -2-2\frac{w}{z}\geq w P_m''(0)-2\frac{w}{z_2}\\
	&= \left(\frac{m^2-1}{3 m^2}+\frac{1}{m^2\sin(\pi/m)^2}\right)w>  \left(\frac{1}{3}+\frac{1}{m^2\sin(\pi/m)^2}\right)w.
	\end{align}
	Thus, if $m^2\sin(\pi/m)^2\geq 3/2$ then $\beta_0(z)>w$. For $m=2$ it is clearly true. We let $g(x)=x^2\sin(\pi/x)^2$ and show that $g(x)$ is strictly increasing in $x\in [2,\infty[$, which implies $m^2\sin(\pi/m)^2\geq 3/2$ for all $m\geq 2$. We have
	\begin{equation}
	g'(x) = 2\sin(\pi/x)(x\sin(\pi/x)-\pi\cos(\pi/x))\geq 0
	\end{equation}
	if and only if $x\sin(\pi/x)-\pi\cos(\pi/x)\geq 0$, which is equivalent to $\tan(\pi/x)\geq \pi/x$. The latter holds true since $\tan(\theta)\geq \theta$ for $\theta\in [0,\pi/2]$. 
\end{proof}

The next lemma has been used in the proof of \cref{thm:stab_mrkc}.
\begin{lemma}\label{lemma:ddRincr}
$\Pi_m''(0)=T_m(\vz)T_m''(\vz)/T_m'(\vz)^2$ is increasing for $\vz\geq 1$.
\end{lemma}
\begin{proof}
	Let $\Pi_m(z,\vz)\coloneqq T_m(\vz)^{-1}T_m(\vz+\vu z)$ be the stability polynomial, with the dependency on $\vz$ made explicit (recall that $\vu$ depends on $\vz$, too). We will show that for every $\vz\geq 1$ there exists $\delta_0=\delta_0(\vz)>0$ such that
	\begin{equation}\label{eq:increasing}
		\Pi_m''(0,\vz)\leq \Pi_m''(0,\vz+\delta)\qquad \forall\,0\leq\delta\leq\delta_0,
	\end{equation}
where the derivative is with respect to the first variable. Indeed, if $\Pi_m''(0,\vz)$ was not increasing then \cref{eq:increasing} would be violated for some $\vz$.

Recall that $\Pi_m(z,\vz)$ satisfies $\Pi_m(0,\vz)=\Pi_m'(0,\vz)=1$ and $\Pi_m(z_i,\vz)=(-1)^iT_m(\vz)^{-1}$ for $i=1,\ldots,m-1$ and some $z_{m-1}<\cdots<z_1<0$, due to the oscillatory behavior of the Chebyshev polynomial $T_m(\vz+\vu z)$.  

Let $\Pi_m(z,\overline\vz)$ be another stability polynomial with higher damping $\overline{\vz}=\vz+\delta$, where $\delta>0$ (\cref{eq:increasing} is obviously satisfied for $\delta=0$). We denote by $\overline z_i$ the extrema of $\Pi_m(z,\overline{\vz})$.
Since $\Pi_m(z,\overline\vz)$ has higher damping than $\Pi_m(z,\vz)$, it has a shorter stability domain. However, by taking $\delta$ small enough we can ensure that $|\Pi_m(z,\overline\vz)|$ is bounded for every $z>z_{m-1}$, with $z_{m-1}$ the last extremum of $\Pi_m(z,\vz)$. Therefore, for the time being, we may assume that there exists $\delta_0(\vz)>0$ such that for $\delta\leq\delta_0(\vz)$ it holds $|\Pi_m(z,\overline\vz)|\leq T_m(\overline\vz)^{-1}$ for all $z\in (z_{m-1},z_1)$.

In every interval $(z_{i+1},z_i)$, for $i=1,\ldots,m-2$, $\Pi_m(z,\vz)$ crosses $\Pi_m(z,\overline\vz)$ at least once. Indeed, since $T_m(\vz)$ is strictly increasing for $\vz\geq 1$, $|\Pi_m(z,\overline\vz)|\leq T_m(\overline\vz)^{-1}<T_m(\vz)^{-1}$ and $\Pi_m(z,\vz)$ takes all values in $(-T_m(\vz)^{-1},T_m(\vz)^{-1})$.

Now we proceed by contradiction and assume that $\Pi_m''(0,\overline\vz)<\Pi_m''(0,\vz)$. Then, for all $z\in[-\epsilon,0)$, for some small $\epsilon>0$,
\begin{equation}\label{eq:ineqP}
	\Pi_m(z,\overline\vz) = 1+z+\frac{1}{2}\Pi_m''(0,\overline\vz)z^2+\bigo{z^3}<1+z+\frac{1}{2}\Pi_m''(0,\vz)z^2+\bigo{z^3}=\Pi_m(z,\vz).
\end{equation}
Since $\Pi_m(z_1,\vz)=-T_m(\vz)^{-1}<-T_m(\overline\vz)^{-1}\leq\Pi_m(z_1,\overline\vz)$, \cref{eq:ineqP} implies that $\Pi_m(z,\vz)$ crosses $\Pi_m(z,\overline\vz)$ in the interval $(z_1,-\epsilon)$. Hence, $\Pi_m(z,\vz)$ and $\Pi_m(z,\overline\vz)$ intersect $m-1$ times in $(z_{m-1},-\epsilon)$. From the first-order conditions we also have $\Pi_m(0,\vz)=\Pi_m(0,\overline\vz)$, $\Pi_m'(0,\vz)=\Pi_m'(0,\overline\vz)$, and hence $\Pi_m(z,\vz)=\Pi_m(z,\overline\vz)$ for all $z$, which leads to the contradiction.

It remains to prove the existence of $\delta_0(\vz)>0$ such that, for $\delta\leq \delta_0(\vz)$, $|\Pi_m(z,\overline\vz)|\leq T_m(\overline\vz)^{-1}$ for all $z\in (z_{m-1},z_1)$. For $z\in(-(1+\overline\vz)/\overline\vu,\overline z_1)$ it holds $-1\leq\overline\vz+\overline\vu z<1$ and thus $|\Pi_m(z,\overline\vz)|\leq T_m(\overline\vz)^{-1}$. Indeed, $\Pi_m(z,\overline\vz)$ reaches its first minimum at $\overline z_1$ and then oscillates until the end of the stability domain. We need to prove $ (z_{m-1},z_1)\subset (-(1+\overline\vz)/\overline\vu,\overline z_1)$. The upper bound $z_1<\overline z_1$ is implied by the higher damping of $\Pi_m(z,\overline\vz)$, which ``compresses'' the stability polynomial towards the origin. For the lower bound, we use the identity $T_m(\cos(\theta))=\cos(m\theta)$, which implies $\vz+\vu z_{m-1}=\cos((m-1)\pi/m)$. Thus, $-(1+\overline\vz)/\overline\vu\leq z_{m-1}$ if 
\begin{equation}\label{eq:conddeltaz}
	(\overline\vz+1)\vu \geq (\vz-\cos((m-1)\pi/m))\overline\vu.
\end{equation}
If $\delta=0$, hence $\vz=\overline\vz$ and $\vu=\overline\vu$, relation \cref{eq:conddeltaz} holds with strict inequality. Therefore, by continuity, there exists $\delta(\vz)>0$ such that \cref{eq:conddeltaz} holds for all $0\leq \delta\leq\delta_0(\vz)$.

\end{proof}


\begin{thebibliography}{10}
\expandafter\ifx\csname url\endcsname\relax
  \def\url#1{\texttt{#1}}\fi
\expandafter\ifx\csname doi\endcsname\relax
  \def\doi#1{\burlalt{doi:#1}{http://dx.doi.org/#1}}\fi
\expandafter\ifx\csname urlprefix\endcsname\relax\def\urlprefix{URL }\fi
\expandafter\ifx\csname href\endcsname\relax
  \def\href#1#2{#2}\fi
\expandafter\ifx\csname burlalt\endcsname\relax
  \def\burlalt#1#2{\href{#2}{#1}}\fi

\bibitem{Abd02}
A.~Abdulle.
\newblock {Fourth order Chebyshev methods with recurrence relation}.
\newblock {\em SIAM J. Sci. Comput.}, 23(6):2041--2054, 2002.

\bibitem{AbM01}
A.~Abdulle and A.~A. Medovikov.
\newblock {Second order Chebyshev methods based on orthogonal polynomials}.
\newblock {\em Numer. Math.}, 18:1--18, 2001.

\bibitem{AbP12}
A.~Abdulle and G.~A. Pavliotis.
\newblock {Numerical methods for stochastic partial differential equations with
  multiple scales}.
\newblock {\em J. Comput. Phys.}, 231(6):2482--2497, 2012.

\bibitem{AbR22b}
A.~Abdulle and G.~{Rosilho de Souza}.
\newblock {Explicit stabilized multirate method for stiff stochastic
  differential equations}.
\newblock {\em (in Press) SIAM J. Sci. Comput.}, \burlalt{arXiv:2010.15193
  [math.NA]}{http://arxiv.org/abs/2010.15193}.

\bibitem{AbR20a}
A.~Abdulle and G.~{Rosilho de Souza}.
\newblock {Instabilities and order reduction phenomenon of an interpolation
  based multirate Runge--Kutta--Chebyshev method}.
\newblock {\em Tech. Report, EPFL}, 2020, \burlalt{arXiv:2003.03154
  [math.NA]}{http://arxiv.org/abs/2003.03154}.

\bibitem{AbV13}
A.~Abdulle and G.~Vilmart.
\newblock {PIROCK: A swiss-knife partitioned implicit-explicit orthogonal
  Runge-Kutta Chebyshev integrator for stiff diffusion-advection-reaction
  problems with or without noise}.
\newblock {\em J. Comput. Phys.}, 242:869--888, 2013.

\bibitem{And79}
J.~F. Andrus.
\newblock {Numerical solution of systems of ordinary differential equations
  into subsytems}.
\newblock {\em SIAM J. Numer. Anal.}, 16(4):605--611, 1979.

\bibitem{BeO84}
M.~J. Berger and J.~Oliger.
\newblock {Adaptive mesh refinement for hyperbolic partial differential
  equations}.
\newblock {\em J. Comput. Phys.}, 53(3):484--512, 1984.

\bibitem{DQD91}
C.~N. Dawson, D.~Qiang, and T.~F. Dupont.
\newblock {A finite difference domain decomposition algorithm for numerical
  solution of the heat equation}.
\newblock {\em Math. Comput.}, 57(195):63--71, 1991.

\bibitem{PiE12}
D.~A. {Di Pietro} and A.~Ern.
\newblock {\em {Mathematical aspects of discontinuous Galerkin methods}},
  volume~69 of {\em Math{\'{e}}matiques et Applications}.
\newblock Springer, Berlin and Heidelberg, 2012.

\bibitem{DDD13}
T.~Dumont, M.~Duarte, S.~Descombes, M.~A. Dronne, M.~Massot, and V.~Louvet.
\newblock {Simulation of human ischemic stroke in realistic 3D geometry}.
\newblock {\em Commun. Nonlinear Sci. Numer. Simul.}, 18(6):1539--1557, 2013.

\bibitem{E03}
W.~E.
\newblock {Analysis of the heterogeneous multiscale method for ordinary
  differential equations}.
\newblock {\em Commun. Math. Sci.}, 1(3):423--436, 2003.

\bibitem{EnT05}
B.~Engquist and Y.~Tsai.
\newblock {Heterogeneous multiscale methods for stiff ordinary differential
  equations}.
\newblock {\em Math. Comput.}, 74(252):1707--1743, 2005.

\bibitem{EnL97}
C.~Engstler and C.~Lubich.
\newblock {Multirate extrapolation methods for differential equations with
  different time scales}.
\newblock {\em Computing}, 58(2):173--185, 1997.

\bibitem{EHL75}
W.~H. Enright, T.~E. Hull, and B.~Lindberg.
\newblock {Comparing numerical methods for stiff systems of O.D.E:s}.
\newblock {\em BIT Numer. Math.}, 15(1):10--48, 1975.

\bibitem{ELV94}
R.~E. Ewing, R.~D. Lazarov, and A.~Vassilev.
\newblock {Finite difference scheme for parabolic problems on composite grids
  with refinement in time and space}.
\newblock {\em SIAM J. Numer. Anal.}, 31(6):1605--1622, 1994.

\bibitem{ELV90}
R.~E. Ewing, R.~D. Lazarov, and P.~S. Vassilevski.
\newblock {Finite difference schemes on grids with local refinement in time and
  space for parabolic problems I. Derivation, stability, and error analysis}.
\newblock {\em Computing}, 45(3):193--215, 1990.

\bibitem{GaH13}
M.~J. Gander and L.~Halpern.
\newblock {Techniques for locally adaptive time stepping developed over the
  last two decades}.
\newblock {\em Lect. Notes Comput. Sci. Eng.}, 91(1):377--385, 2013.

\bibitem{GIK03}
C.~W. Gear, G.~Ioannis, and G.~Kevrekidis.
\newblock {Projective methods for stiff differential equations: problems with
  gaps in their eigenvalue spectrum}.
\newblock {\em SIAM J. Sci. Comput.}, 24(4):1091--1106, 2003.

\bibitem{GeW84}
C.~W. Gear and D.~R. Wells.
\newblock {Multirate linear multistep methods}.
\newblock {\em BIT Numer. Math.}, 24(4):484--502, 1984.

\bibitem{GMM15}
M.~J. Grote, M.~Mehlin, and T.~Mitkova.
\newblock {Runge--Kutta-based explicit local time-stepping methods for wave
  propagation}.
\newblock {\em SIAM J. Sci. Comput.}, 37(2):A747--A775, 2015.

\bibitem{GuB10}
G.~Guennebaud and B.~Jacob.
\newblock {Eigen v3}, 2010.
\newblock \urlprefix\url{http://eigen.tuxfamily.org/}.

\bibitem{GuL60}
A.~Guillou and B.~Lago.
\newblock {Domaine de stabilit{\'{e}} associ{\'{e}} aux formules
  d'int{\'{e}}gration num{\'{e}}rique d'{\'{e}}quations diff{\'{e}}rentielles,
  {\`{a}} pas s{\'{e}}par{\'{e}}s et {\`{a}} pas li{\'{e}}s. Recherche de
  formules {\`{a}} grand rayon de stabilit{\'{e}}.}
\newblock In {\em 1er Congr. Ass. Fran. Calc. AFCAL}, pages 43--56, Grenoble,
  1960.

\bibitem{GKR01}
M.~G{\"{u}}nther, A.~Kv{\ae}rn{\o}, and P.~Rentrop.
\newblock {Multirate partitioned Runge--Kutta methods}.
\newblock {\em BIT Numer. Math.}, 41(3):504--514, 2001.

\bibitem{GuR93}
M.~G{\"{u}}nther and P.~Rentrop.
\newblock {Multirate ROW methods and latency of electric circuits}.
\newblock {\em Appl. Numer. Math.}, 13(1-3):83--102, 1993.

\bibitem{GuS16}
M.~G{\"{u}}nther and A.~Sandu.
\newblock {Multirate generalized additive Runge Kutta methods}.
\newblock {\em Numer. Math.}, 133(3):497--524, 2016.

\bibitem{HLW06}
E.~Hairer, C.~Lubich, and G.~Wanner.
\newblock {\em {Geometric numerical integration}}, volume~31 of {\em Springer
  Series in Computational Mathematics}.
\newblock Springer, Heidelberg, 2006.

\bibitem{HNW08}
E.~Hairer, S.~P. N{\"{o}}rsett, and G.~Wanner.
\newblock {\em {Solving ordinary differential equations I}}, volume~8 of {\em
  Springer Series in Computational Mathematics}.
\newblock Springer-Verlag, Berlin, 2008.

\bibitem{HaW02}
E.~Hairer and G.~Wanner.
\newblock {\em {Solving ordinary differential equations II}}, volume~14 of {\em
  Springer Series in Computational Mathematics}.
\newblock Springer-Verlag, Berlin, 2002.

\bibitem{HoL99}
M.~Hochbruck and C.~Lubich.
\newblock {A Gautschi-type method for oscillatory second-order differential
  equations}.
\newblock {\em Numer. Math.}, 83:403--426, 1999.

\bibitem{HoO10}
M.~Hochbruck and A.~Ostermann.
\newblock {Exponential integrators}.
\newblock {\em Acta Numer.}, 19:209--286, 2010.

\bibitem{Hof76}
E.~Hofer.
\newblock {A partially implicit method for large stiff systems of ODEs with
  only few equations introducing small time-constants}.
\newblock {\em SIAM J. Numer. Anal.}, 13(5):645--663, 1976.

\bibitem{KPS06}
B.~S. Kirk, J.~W. Peterson, R.~H. Stogner, and G.~F. Carey.
\newblock {libMesh : a C++ library for parallel adaptive mesh
  refinement/coarsening simulations}.
\newblock {\em Eng. Comput.}, 22(3-4):237--254, 2006.

\bibitem{KnW98}
O.~Knoth and R.~Wolke.
\newblock {Implicit-explicit Runge-Kutta methods for computing atmospheric
  reactive flows}.
\newblock {\em Appl. Numer. Math.}, 28(2-4):327--341, 1998.

\bibitem{Kva99}
A.~Kv{\ae}rn{\o}.
\newblock {Stability of multirate Runge--Kutta schemes}.
\newblock In {\em Proc. 10th Coll. Differ. Equations}, volume~1A, pages
  97--105, 1999.

\bibitem{Leb94}
V.~I. Lebedev.
\newblock {How to solve stiff systems of differential equations by explicit
  methods}.
\newblock In {\em Numer. methods Appl.}, pages 45--80. CRC, Boca Raton, FL,
  1994.

\bibitem{LeM94}
V.~I. Lebedev and A.~A. Medovikov.
\newblock {Explicit methods of second order for the solution of stiff systems
  of ODEs}.
\newblock {\em Russ. Acad. Sci.}, 1994.

\bibitem{Lin72}
B.~Lindberg.
\newblock {IMPEX: a program package for solution of systems of stiff
  differential equations}.
\newblock Technical report, Dept. of Information Processing, Royal Inst. of
  Tech., Stockholm, 1972.

\bibitem{Med98}
A.~A. Medovikov.
\newblock {High order explicit methods for parabolic equations}.
\newblock {\em BIT Numer. Math.}, 38(2):372--390, 1998.

\bibitem{MAM06}
R.~Minero, M.~J.~H. Anthonissen, and R.~M.~M. Mattheij.
\newblock {A local defect correction technique for time-dependent problems}.
\newblock {\em Numer. Methods Partial Differ. Equ.}, 22(1):128--144, 2006.

\bibitem{Mir17}
T.~Mirzakhanian.
\newblock {\em {Multi-rate Runge--Kutta--Chebyshev time stepping for parabolic
  equations on adaptively refined meshes}}.
\newblock Master thesis, Boise State University, 2017.
\newblock \doi{10.18122/B2V715}.

\bibitem{Ric60}
J.~R. Rice.
\newblock {Split Runge--Kutta method for simultaneous equations}.
\newblock {\em J. Res. Natl. Bur. Stand. Sect. B, Math. Math. Phys.},
  64B(3):151--170, 1960.

\bibitem{RLS21}
S.~Roberts, J.~Loffeld, A.~Sarshar, C.~S. Woodward, and A.~Sandu.
\newblock {Implicit multirate GARK methods}.
\newblock {\em J. Sci. Comput.}, 87(4), 2021.
\newblock \doi{10.1007/s10915-020-01400-z}.

\bibitem{RSS20}
S.~Roberts, A.~Sarshar, and A.~Sandu.
\newblock {Coupled Multirate Infinitesimal GARK Schemes for Stiff Systems with
  Multiple Scales}.
\newblock {\em SIAM J. Sci. Comput.}, 42(3):A1609--A1638, 2020.

\bibitem{Ros20}
G.~{Rosilho De Souza}.
\newblock {\em {Numerical methods for deterministic and stochastic differential
  equations with multiple scales and high contrasts}}.
\newblock PhD thesis, EPFL, Lausanne, 2020.
\newblock \doi{10.5075/epfl-thesis-7445}.

\bibitem{San19}
A.~Sandu.
\newblock {A class of multirate infinitesimal GARK methods}.
\newblock {\em SIAM J. Numer. Anal.}, 57(5):2300--2327, 2019.

\bibitem{SaG15}
A.~Sandu and M.~G{\"{u}}nther.
\newblock {A generalized-structure approach to additive Runge-Kutta methods}.
\newblock {\em SIAM J. Numer. Anal.}, 53(1):17--42, 2015.

\bibitem{SRS19}
A.~Sarshar, S.~Roberts, and A.~Sandu.
\newblock {Design of high-order decoupled multirate GARK schemes}.
\newblock {\em SIAM J. Sci. Comput.}, 41(2):A816--A847, 2019.

\bibitem{SHV07}
V.~Savcenco, W.~Hundsdorfer, and J.~Verwer.
\newblock {A multirate time stepping strategy for stiff ordinary differential
  equations}.
\newblock {\em BIT Numer. Math.}, 47(1):137--155, 2007.

\bibitem{SaM10}
V.~Savcenco and R.~M.~M. Mattheij.
\newblock {Multirate numerical integration for stiff ODEs}.
\newblock In {\em Prog. Ind. Math. ECMI 2008}, volume~15, pages 327--332.
  Springer, Heidelberg, 2010.

\bibitem{SeR19}
J.~M. Sexton and D.~R. Reynolds.
\newblock {Relaxed Multirate Infinitesimal Step Methods for Initial-Value
  Problems}.
\newblock {\em Preprint}, 2019, \burlalt{arXiv:1808.03718
  [math.NA]}{http://arxiv.org/abs/1808.03718}.

\bibitem{ShV00}
G.~I. Shishkin and P.~N. Vabishchevich.
\newblock {Interpolation finite difference schemes on grids locally refined in
  time}.
\newblock {\em Comput. Methods Appl. Mech. Eng.}, 190(8-10):889--901, 2000.

\bibitem{SkA89}
S.~Skelboe and P.~U. Andersen.
\newblock {Stability properties of backward Euler multirate formulas}.
\newblock {\em SIAM J. Sci. Stat. Comput.}, 10(5):1000--1009, 1989.

\bibitem{SSV98}
B.~P. Sommeijer, L.~Shampine, and J.~G. Verwer.
\newblock {RKC: An explicit solver for parabolic PDEs}.
\newblock {\em J. Comput. Appl. Math.}, 88(2):315--326, 1998.

\bibitem{TrV91}
R.~Trompert and J.~Verwer.
\newblock {A static-regridding method for two-dimensional parabolic partial
  differential equations}.
\newblock {\em Appl. Numer. Math.}, 8(1):65--90, 1991.

\bibitem{TrV93a}
R.~Trompert and J.~Verwer.
\newblock {Analysis of local uniform grid refinement}.
\newblock {\em Appl. Numer. Math.}, 13(1-3):251--270, 1993.

\bibitem{TrV93b}
R.~Trompert and J.~Verwer.
\newblock {Analysis of the implicit Euler local uniform grid refinement
  method}.
\newblock {\em SIAM J. Sci. Comput.}, 14(2):259--278, 1993.

\bibitem{TrV93c}
R.~Trompert and J.~Verwer.
\newblock {Runge--Kutta methods and local uniform grid refinement}.
\newblock {\em Math. Comput.}, 60(202):591--616, 1993.

\bibitem{HoS80}
P.~J. {Van der Houwen} and B.~P. Sommeijer.
\newblock {On the internal stability of explicit, $m$-stage Runge--Kutta
  methods for large $m$-values}.
\newblock {\em Z. Angew. Math. Mech.}, 60(10):479--485, 1980.

\bibitem{Van03}
E.~Vanden-Eijnden.
\newblock {Numerical techniques for multi-scale dynamical systems with
  stochastic effects}.
\newblock {\em Commun. Math. Sci.}, 1(2):385--391, 2003.

\bibitem{VaV92}
A.~S. {Vasudeva Murthy} and J.~G. Verwer.
\newblock {Solving parabolic integro-differential equations by an explicit
  integration method}.
\newblock {\em J. Comput. Appl. Math.}, 39(1):121--132, 1992.

\bibitem{Ver80}
J.~G. Verwer.
\newblock {An implementation of a class of stabilized explicit methods for the
  time integration of parabolic equations}.
\newblock {\em ACM Trans. Math. Softw.}, 6(2):188--205, 1980.

\bibitem{Ver96}
J.~G. Verwer.
\newblock {Explicit Runge--Kutta methods for parabolic partial differential
  equations}.
\newblock {\em Appl. Numer. Math.}, 22(1-3):359--379, 1996.

\bibitem{VHS90}
J.~G. Verwer, W.~Hundsdorfer, and B.~P. Sommeijer.
\newblock {Convergence properties of the Runge--Kutta--Chebyshev method}.
\newblock {\em Numer. Math.}, 57(1):157--178, 1990.

\bibitem{WKG09}
J.~Wensch, O.~Knoth, and A.~Galant.
\newblock {Multirate infinitesimal step methods for atmospheric flow
  simulation}.
\newblock {\em BIT Numer. Math.}, 49(2):449--473, 2009.

\bibitem{Zbi11}
C.~J. Zbinden.
\newblock {Partitioned Runge-Kutta-Chebyshev methods for
  diffusion-advection-reaction problems}.
\newblock {\em SIAM J. Sci. Comput.}, 33(4):1707--1725, 2011.

\end{thebibliography}

\end{document}